\documentclass[a4paper,10pt]{amsart}
\usepackage[utf8]{inputenc}
\usepackage{enumerate}
\usepackage{stmaryrd}
\usepackage{amsthm}
\usepackage{amsmath}
\usepackage{bm}
\usepackage{amsfonts}
\usepackage{amssymb}
\usepackage{mathtools}
\usepackage{mathrsfs}
\usepackage{setspace}
\usepackage{xcolor}
\usepackage{textgreek}
\usepackage{todonotes}
\usepackage[a4paper, left=3cm, right=3cm, top=3cm, bottom=3cm]{geometry}
\usepackage{thmtools} 
\usepackage{bm}
\mathtoolsset{showonlyrefs}

\usepackage[pdfdisplaydoctitle,colorlinks,breaklinks,urlcolor=blue,linkcolor=blue,citecolor=blue]{hyperref} 

\newenvironment{acknowledgements}{%
  \begin{abstract}
}{%
  \end{abstract}
}
\newtheorem{thm}{Theorem}[section]

\newtheorem*{thm*}{Theorem}
\newtheorem{definition}[thm]{Definition}
\newtheorem{cor}[thm]{Corollary}

\newtheorem{lem}[thm]{Lemma}

\newtheorem{prop}[thm]{Proposition}

\newcommand{\one}{\mathbf{1}}

\newcommand{\N}{\mathbb{N}}

\newcommand{\R}{\mathbb{R}}

\newcommand{\PP}{\mathbb{P}}
\newcommand{\norm}[1]{\left\|#1\right\|}

\newcommand{\dvg}{\mathord{{\rm div}}}

\renewcommand{\O}{\Omega}

\theoremstyle{remark}
\newtheorem{rmk}[thm]{Remark}

\numberwithin{equation}{section}
\allowdisplaybreaks

\title[Zero-noise and LDP for stochastic transport]{Zero-noise selection and Large Deviations in $L^\infty_t L^p_x$ for the stochastic transport equation beyond DiPerna-Lions}

\author[G. Crippa]{Gianluca Crippa}
\address{Departement Mathematik und Informatik, Universit\"at Basel, Spiegelgasse 1, CH-4051 Basel, Switzerland.} 
\email{gianluca.crippa@unibas.ch}

\author[E. Luongo]{Eliseo Luongo}
\address{Fakultät für Mathematik, Universität Bielefeld, 33501 Bielefeld, Germany} 
\email{eluongo@math.uni-bielefeld.de  }
\author[U. Pappalettera]{Umberto Pappalettera}
\address{Fakultät für Mathematik, Universität Bielefeld, 33501 Bielefeld, Germany} 
\email{upappale@math.uni-bielefeld.de}

\date\today

\begin{document}

\begin{abstract}
We consider $L^\infty_t L^p_x$ solutions of the stochastic transport equation with drift in $L^\infty_t W^{1,q}_x$. We show strong existence and pathwise uniqueness of solutions in a regime of parameters $p,q$ for which non-unique weak solutions of the deterministic transport equation exist.
When the intensity of the noise goes to zero, we prove that the solutions of the stochastic transport equation converge to the unique renormalized solution of the transport equation in the sense of DiPerna-Lions.
Furthermore, we show that the convergence is governed by a Large Deviations Principle in the space $L^\infty_t L^p_x$.
Since the space $L^\infty_t L^p_x$ is not separable, the weak convergence approach to Large Deviations by Budhiraja, Dupuis, and Maroulas is not directly applicable.
\end{abstract}

\maketitle


\section{Introduction}
This work is devoted to the study of random perturbations of the transport equation 
\begin{align} \label{eq:transport}
\begin{cases}
\partial_t \rho
+
b \cdot \nabla \rho
=
0,
\\
\rho|_{t=0} = \rho_0 \in L^p_x,
\end{cases}
\end{align}
posed on the domain $[0,T] \times \R^d$, for some $T<\infty$ and $d \geq 2$.
Here we consider the case where the unknown $\rho$ belongs to the space $L^\infty_tL^p_x$, and $b$ is a given velocity field of class $L^\infty_t W^{1,q}_x$ for some exponents $p,q \in [1,\infty]$. Suitable assumptions on $\dvg \, b $ are also imposed.
The Cauchy problem \eqref{eq:transport} has been addressed thoroughly in the seminal paper \cite{diperna1989ordinary} by DiPerna and Lions (see also \cite{Am04}).
According to the theory developed therein, weak solutions to \eqref{eq:transport} are unique when $1/p + 1/q \leq 1$. 
However, to give meaning to weak solutions, it is actually sufficient to require that the product $b \rho$ is well-defined, when integrated against compactly supported test functions. 
By Sobolev embedding, $W^{1,q}_x \subset L^{q^*}_x$, where $1/q^* = 1/q-1/d$. In particular, $b \rho \in L^1_x$ whenever $1/q+1/p \leq 1+1/d$.
In this regime weak solutions exist but might not be unique.
In fact, building upon a research line initiated by Modena and Székelyhidi Jr. in \cite{MoSz18}, Brué, Colombo and Kumar have constructed in \cite{BrCoKu} a vector field $b \in L^\infty_t W^{1,q}_x \cap L^\infty_tL^{(p-1)/p}_x$ leading to non-unique weak solutions to \eqref{eq:transport} of class $L^\infty_t L^p_x$ for every $p,q$ satisfying $(d-1)/dp + 1/q > 1$.
Notice that the $L^{(p-1)/p}_x$ space integrability of $b$ in \cite{BrCoKu} guarantees that the product $b\rho$ is well-defined.

In the case $1/q+1/p > 1+1/d$, weak solutions to \eqref{eq:transport} are not generally defined as the product $b\rho$ is not.
One of the key insights of \cite{diperna1989ordinary} is that, for arbitrary exponents $p,q \in [1,\infty]$, one can make sense of equation \eqref{eq:transport} by introducing the notion of renormalized solutions, namely assuming that for every admissible $\beta \in C_b(\R)\cap C^1(\R),$ cf. \cite[Page 521]{diperna1989ordinary}, the following Cauchy problem is satisfied in distributional sense 
\begin{align*} 
\begin{cases}
\partial_t \beta(\rho)
+
b \cdot \nabla \beta(\rho)
=
0,
\\
\beta(\rho)|_{t=0} = \beta(\rho_0),
\end{cases}
\end{align*}
i.e. when testing against a test function $\phi \in C^\infty_c([0,T) \times \R^d)$.
Quite importantly, renormalized solutions are weak solutions when the product $b\rho$ is well-defined, and renormalized solutions exist and are unique in the whole range of exponents $p,q \in [1,\infty ]$, cf. \cite[Theorem II.3]{diperna1989ordinary}. This makes the notion of renormalized solutions a very satisfactory selection criterion among non-unique solutions of \eqref{eq:transport}. 

On the other hand, there are several regularizing procedures that restore uniqueness of solutions when applied to \eqref{eq:transport}. 
For instance, well-posedness of \eqref{eq:transport} can be improved by regularizing the initial condition $\rho_0$, and/or adding a diffusivity term $\nu \Delta$, $\nu >0$, to the right-hand-side of the equation. 
Furthermore, building upon the theory developed in \cite{CrSp15} and \cite{CiCrSp21}, in \cite{BoCiCr22} it is proved that this procedure selects the unique renormalized solution of \eqref{eq:transport} in the vanishing-diffusivity limit $\nu \downarrow 0$, despite the fact that in this regime multiple weak solutions may exist. 

In this paper we are interested in a different regularization for the transport equation \eqref{eq:transport}, that is obtained with a random perturbation of the equation by a rough transport noise of H\"older space regularity $\alpha \in (0,1/2)$ and intensity $\varepsilon>0$:
\begin{align} \label{eq:stoch_transport}
\begin{cases}
\partial_t \rho^\varepsilon
+
b \cdot \nabla \rho^\varepsilon
+
\varepsilon \circ\, \partial_t W \cdot \nabla \rho^\varepsilon
=
0,
\\
\rho^\varepsilon|_{t = 0}
= \rho^\varepsilon_0 \in L^p_x,
\end{cases}
\end{align}
where $W:=\sqrt{\mathcal{Q}}\mathcal{W}$, $\mathcal{Q}$ is the covariance operator of the noise $W$, and $\mathcal{W}$ is a cylindrical Wiener process in the space $\mathbb{H}:=\{u \in (L^2_x)^d \,:\; \dvg \,u = 0\}$ defined on a given filtered probability space $(\Omega,\mathcal{F},\{ \mathcal{F}_t \}_{t \geq 0},\PP)$ with complete and right-continuous filtration, see \cite[Section 4.1.2]{DaPZab}.
The stochastic integral in \eqref{eq:stoch_transport} is formally meant in the sense of Stratonovich.
We assume that the noise is space-homogeneous and the matrix $Q(x,y)=Q(x-y) := \mathbb{E}[W_1(x) \otimes W_1(y)]$, $x,y \in \R^d$, is explicitly given by 
\begin{align*}
Q(z) = Z_Q \int_{\R^d} \left( I_d - \frac{\xi \otimes \xi}{|\xi|^2}  \right) \frac{e^{i z \cdot\xi}}{(1+|\xi|^2)^{d/2+\alpha}} d\xi,\qquad 
0 <\alpha <\frac{1}{2},
\end{align*}
where $Z_Q>0$ is a renormalizing constant such that $Q(0) = 2I_d$. This noise structure was introduced by Kraichnan in \cite{Kr68} to describe turbulent advection of passive scalars and has been popularized by physicists since then, cf. the lecture notes \cite{Ga02}.

Under this assumption, there exist divergence-free vector fields $\{\sigma_k\}_{k \in \N} \subset C^\infty(\R^d,\R^d) \cap L^2 (\R^d,\R^d)$ and i.i.d. Wiener processes  $\{W^k\}_{k \in \N}$ on $(\Omega,\mathcal{F},\{\mathcal{F}_t\}_{t \geq 0},\PP)$ such that $W$ can be represented as $W_t(x) = \sum_{k \in \N} \sigma_k(x) W^k_t$, and $Q(x-y)=\sum_{k \in \N}\sigma_k(x) \otimes \sigma_k(y)$. 
Stochastic It\=o integrals with respect to $W$ are well-defined continuous local martingales for every $\{\mathcal{F}_t\}_{t \geq 0}$-progressively measurable process $f$ such that $\int_0^T \langle Q \ast f_s,f_s \rangle ds < \infty$ $\PP$-almost surely, where $\langle \cdot,\cdot \rangle$ denotes the $L^2_x$ inner product, see for instance \cite[Lemma 2.8]{GaLu23+}.

Analytically weak, probabilistically strong solutions to \eqref{eq:stoch_transport} are rigorously defined as $\{\mathcal{F}_t\}_{t \geq 0}$-progressively measurable processes $\rho^\varepsilon\in L^{\infty}_\omega L^{\infty}_t L^2_x$ with continuous paths with values in $H^{-s}_x$ for some $s>0$, satisfying the stochastic equation in It\=o sense
\begin{align} \label{eq:stoch_trans_Ito}
d \rho^\varepsilon 
+ 
b \cdot \nabla \rho^\varepsilon dt
+
\varepsilon \sum_{k \in \N} \sigma_k \cdot \nabla \rho^\varepsilon dW^k_t
=
\varepsilon^2 \Delta \rho^\varepsilon dt,
\end{align}
when testing against a test function $\phi \in C^\infty_c(\R^d \times [0,T))$, or equivalently $\phi \in C^\infty_c(\R^d)$. We refer to \autoref{prop_well_posed} below and \autoref{def:weak_sol} for details. Let us point out that the term $\varepsilon^2 \Delta \rho^\varepsilon$ is a consequence of formally rewriting the Stratonovich integrals in It\=o form and we will always consider \eqref{eq:stoch_trans_Ito} in our rigorous analysis.
For technical convenience, we shall consider the following more general equation, with a  diffusivity parameter $\kappa \in [0,1)$ and a drift $g$ taking values in the Cameron-Martin space associated to our noise (more details in next subsection): 
\begin{align} \label{eq:rho_general}
        d \rho^{\varepsilon,\kappa}
        +
        (b+g) \cdot \nabla \rho^{\varepsilon,\kappa} dt
        +
        \varepsilon \sum_{k \in \N} \sigma_k \cdot\nabla \rho^{\varepsilon,\kappa} dW^k_t
        =
        (1+\kappa)\varepsilon^2 \Delta \rho^{\varepsilon,\kappa} dt.
    \end{align}

Before presenting the main results of this work, we give a preliminary proposition establishing well-posedness of \eqref{eq:rho_general}, which is of independent interest.

\begin{prop}\label{prop_well_posed}
Fix $\alpha \in (0,1/2)$ and assume $\rho_0 \in L^2_x \cap L^p_x$ with $p \in [1,\infty]$, $b \in L^\infty_t W^{1,q}_x$ with $\frac{d}{2(1-\alpha)}<q \leq 2$ and $\dvg \,b \in L^1_t L^\infty_x \cap L^{\infty}_t H^{\vartheta}_x$ for some $\vartheta>0$, and let $g \in L^\infty_\omega L^2_t H^{d/2+\alpha}_x$ be a divergence-free, $\{\mathcal{F}_t\}_{t \geq 0}$-progressively measurable process. 
Let $\varepsilon \in (0,1)$ and $\kappa \in [0,1)$.
Then:
\begin{itemize}
\item 
Pathwise uniqueness for \eqref{eq:rho_general} holds in the class of analytically weak solutions of class $L^{\infty}_\omega L^{\infty}_t L^2_x$ with continuous paths in $H^{-s}_x$, for some $s>0$.
\item 
There exists a probabilistically strong, analytically weak solution to \eqref{eq:rho_general} of class 
\begin{align*}
\rho \in L^\infty_\omega L^\infty_t (L^2_x \cap L^p_x) \cap L^2_\omega L^2_t H^{1-\alpha-\delta}_x \cap L^n_\omega C^{\gamma}_t H^{{-\sigma}}_x,
\end{align*}
for arbitrary $\delta \in (0,\alpha)$, $n \in [2,\infty)$, $\gamma <1/2$, and $\sigma>d/2+ 1$.
Moreover, if $p \in (1,\infty)$ then $\rho$ has weakly continuous trajectories in $L^p_x$ $\PP$-almost surely.
\end{itemize}
\end{prop}

The integrability assumptions $\rho_0 \in L^2_x$ and $b \in L^\infty_t W^{1,q}_x$ with $q>\frac{d}{2(1-\alpha)}$ above play a key role in the proof of the Sobolev regularity of solutions, which is absent in the DiPerna-Lions theory (cf. \cite{albertiCrippaMazzuccato}) and crucial in our proofs. In fact, we do not claim our result is sharp in term of sole well-posedness of \eqref{eq:rho_general}, but this falls outside the scope of this paper. 
On the other hand, the $L^p_x$ integrability is only there for the sake of generality. 
The proof of \autoref{prop_well_posed}, given in the Appendix, is based on a technical adaptation of the computations performed by Galeati, Grotto, and Maurelli in \cite{GaGrMa24+} for the homogeneous Sobolev norm $L^2_\omega L^2_t \dot{H}^{1-\alpha-\delta}_x$, without assuming their condition $\rho_0 \in L^2_x \cap \dot{H}^{-s}_x$ for some $s>0$.
The proof of \autoref{prop_well_posed}, given in the Appendix, is based on a technical adaptation of the computations performed by Galeati, Grotto, and Maurelli in \cite{GaGrMa24+} for the homogeneous Sobolev norm $L^2_\omega L^2_t \dot{H}^{1-\alpha-\delta}_x$, without assuming their condition $\rho_0 \in L^2_x \cap \dot{H}^{-s}_x$ for some $s>0$.

It is worth comparing the range $\frac{d}{2(1-\alpha)}<q \leq 2$ of the proposition above with the well-posedness regime of the deterministic transport equation \eqref{eq:transport}.
First, notice that for $p=2$ and $q\geq 2$, weak solutions of \eqref{eq:transport} are renormalized, and hence unique. For our purposes, the most interesting and challenging case is when \eqref{eq:transport} admits non-unique solutions. In particular, $q \leq 2$ is not really a limitation that we are interested to drop at the present moment.
In fact, the regime $p=2$ and $\frac{d}{2(1-\alpha)}<q \leq 2$ actually comprehends cases where weak solutions of \eqref{eq:transport} are not unique, see \cite[Theorem 1.1]{BrCoKu}. 
Thus, taking e.g. $d=2$ and $1/(1-\alpha) < q < 4/3$, one has both non-unique $L^\infty_t L^2_x$ solutions of \eqref{eq:transport}, and unique solutions of the stochastic transport equation \eqref{eq:rho_general}.
We can therefore interpret \autoref{prop_well_posed} as a \emph{regularization by noise} phenomenon, see \cite{flandoli_book}, and the references in \autoref{ssec:biblio}. 

\subsection{Main results}\label{sec_main_res}
In view of the discussion above, we can look at \eqref{eq:stoch_transport} as a regularization mechanism for \eqref{eq:transport}. Indeed, \autoref{prop_well_posed} guarantees well-posedness of \eqref{eq:stoch_transport} for any $\varepsilon >0$, despite the possible non-uniqueness of solutions to the deterministic transport equation \eqref{eq:transport} for $\max\{2,p\}<q'$, the H\"older conjugate of $q$. For example, as explained above, non-uniqueness is true when $p=2$, see \cite[Theorem 1.1]{BrCoKu}.
It is therefore natural to ask whether the regularization mechanism \eqref{eq:stoch_transport} selects a unique weak solution of \eqref{eq:transport} as the noise intensity $\varepsilon \downarrow 0$, and if this is the case, how fast the convergence is.
In general, identification of the \emph{zero-noise} limit of ill-posed deterministic systems is a difficult problem.
Very few examples are known, mostly pivoting around the Peano non-uniqueness phenomenon (see Bafico and Baldi \cite{BaBa81}, its revisitations \cite{DeFl14, delarue2019zero}, and \cite{AtFl09} where notably the zero-noise limit is different from the vanishing-viscosity limit), and for a system of two collapsing Vlasov-Poisson point charges \cite{delarue2014noise}. We also point out \cite{grotto2024zero}, where a system of collapsing Euler point vortices from \cite{MaPu,grotto2022burst} is investigated by numerical simulations.
Let us, lastly, mention the recent preprint \cite{pitcho2025zero}.

Here we are able to identify the zero-noise limit of \eqref{eq:stoch_transport}. In order to state our results, we need some preliminary definitions.

In what follows we denote $\tilde{H}^s_x$, $s \in \R$, the weighted Sobolev space on $\R^d$, that is the closure of compactly supported smooth functions under the norm $\| \phi \|_{\tilde{H}^s_x} := \| \phi w \|_{H^s_x}$, where the weight $w$ is defined as $w(x) := (1+|x|^2)^{-d/2-1}$. 
Let us also introduce the Polish spaces
\begin{align*}
\tilde{H}^{s-}_x := \bigcap_{n \in \N} \tilde{H}^{s-1/n}_x,
\quad
\mbox{with distance }\,
d_{\tilde{H}^{s-}}(u,v):=\sum_{n\geq 1}\frac{1}{2^n}\left(1\wedge \norm{u-v}_{\tilde{H}^{s-1/n}_x}\right).
\end{align*}
In order to simplify the notation, we will write $\tilde{H}_x^-$ in place of $\tilde{H}_x^{0-}$. 

Hereafter we fix $p \in (1,\infty)$ and we denote $B$ the closed unit ball in $L^2_x \cap L^p_x$ endowed with the strong topology, $\mathcal{B}$ the closed ball of radius $\exp(\|\dvg \ b\|_{L^1_t L^\infty_x})$ in $L^2_x\cap L^p_x$ endowed with the weak topology (which is metrizable with a distance $d_\mathcal{B}$), and the spaces
\begin{align*}
\mathcal{E}
:=
C_t \tilde{H}^{-}_x
\cap 
C_t \mathcal{B},
\qquad
\mathscr{E} := L^\infty_t (L^2_x \cap  L^p_x)\cap \mathcal{E}.
\end{align*}
We endow $\mathcal{E}$ with its natural distance, denoted by $d_{\mathcal{E}}$, and $\mathscr{E}$ with the $L^{\infty}_t(L^2_x\cap L^p_x)$ distance, denoted hereafter by $d_{\mathscr{E}}$ for notational simplicity.
As elements in $\mathscr{E}$ have weakly continuous trajectories in $L^2_x\cap L^p_x$ and the $L^2_x$ and $L^p_x$ norms are lower-semicontinous with respect to the weak convergence, it is easy to see that $d_\mathscr{E}(u,v) = \sup_{t \in [0,T]} \| u_t-v_t \|_{L^2_x}+\sup_{t \in [0,T]} \| u_t-v_t \|_{L^p_x}$, and therefore $d_\mathscr{E}$ is bi-Lipschitz equivalent to the natural distance $d_\mathscr{E}+d_\mathcal{E}$ on $\mathscr{E}$. Even though the sets $\mathcal{E}$ and $\mathscr{E}$ are the same, they are very different from the topological viewpoint. 
The metric space $(\mathcal{E},d_\mathcal{E})$ is Polish. The metric space $(\mathscr{E},d_\mathscr{E})$ is complete but fails to be separable. This lack of separability is a major difficulty in our setting, as we will make clear later.

The first result of this work  provides a positive answer to the question of convergence toward a unique weak solution of the transport equation \eqref{eq:transport}. The limit is the unique renormalized solution of \eqref{eq:transport} given by \cite[Theorem II.3]{diperna1989ordinary}, with convergence in probability. More precisely, we have:
\begin{thm} \label{thm:convergence}
Let $b \in L^\infty_t W^{1,q}_x$ with $\dvg \,b \in L^1_t L^\infty_x \cap L^{\infty}_t H^{\vartheta}_x$ for some $\frac{d}{2(1-\alpha)}< q \leq 2$ and $\vartheta>0$.
Fix any $p \in (1,\infty)$.
Then for every $\rho_0^\varepsilon \to \rho_0$ in $B$ as $\varepsilon \downarrow 0$, the unique solution of the stochastic transport equation \eqref{eq:stoch_transport} converges in probability, as $\mathcal{E}$-valued random variables, to the unique renormalized solution $\rho$ of \eqref{eq:transport} with initial condition $\rho_0$. 
If in addition $\dvg \, b = 0$, then $\rho^\varepsilon$ converges in probability to $\rho$ as $\mathscr{E}$-valued random variables.
\end{thm} 

The proof of the first part of \autoref{thm:convergence} relies on relatively standard stochastic compactness estimates for $\rho^\varepsilon$ in the space $ L^\infty_t (L^2_x \cap L^p_x) \cap  C^\gamma_t H^{-\sigma}_x$, for some $\sigma>d/2+1$ and uniformly in $\varepsilon \in (0,1)$, implying tightness in $\mathcal{E}$ by classical compact embedding results. 
The identification of the limit as the unique renormalized solution of \eqref{eq:transport}, among the possibly non-unique weak solutions, follows from the DiPerna-Lions stability results for renormalized solutions.
On the other hand, the convergence in $\mathscr{E}$ is much more delicate and does not follow from any tightness argument in $\mathscr{E}$. The key idea is that, when $\dvg \, b = 0$, renormalized solutions to the transport equation have constant $L^2_x$ norm and this permits to show that $\| \rho^\varepsilon_t \|_{L^2_x} \to \| \rho_t \|_{L^2_x}$ as $\varepsilon \downarrow 0$, uniformly in $t \in [0,T]$. In view of \autoref{lem_Backward_convergence}, this suffices to promote weak convergence in $L^2_x\cap L^p_x$ into a strong one for any $p \in (1,\infty)$.
In this part of the proof, we crucially use the Sobolev regularization in $L^2_\omega L^2_t H^{1-\alpha-\delta}_x$ provided by \autoref{prop_well_posed} for $\varepsilon >0$, more specifically when proving \autoref{lem:beta_rho} in the Appendix. 
In particular, we would not be able to prove the second part of \autoref{thm:convergence} if $W$ was replaced by a spatially smooth noise, as the latter does not give, in general, any Sobolev regularity of solutions.
Finally, we remark that we cannot replace the space $L^\infty_t L^p_x$ with $C_t L^p_x$ in the above, since we do not know if solutions to \eqref{eq:stoch_transport} have strongly continuous trajectories. Nonetheless, the convergence with respect to the distance $d_\mathscr{E}$ is uniform in time.

Given the convergence in \autoref{thm:convergence}, one is naturally led to ask how fast the convergence takes place.
For example, for the vanishing viscosity scheme and assuming $\rho_0 \in L^\infty_x$, the authors of \cite{BoCiCr22} are able to give rates of convergence in $C_t L^p_x$ for every $p \in [1,\infty)$.
However, their proof does not generalize to $\rho_0 \notin L^\infty_x$.
In this paper, we are interested in giving probabilistic rates for the convergence in $\mathscr{E}$. Namely, we want to give precise asymptotic estimates for the probability that the solution $\rho^{\varepsilon,\rho_0}$ of \eqref{eq:stoch_transport} is not close to the unique renormalized solution of \eqref{eq:transport} with the same initial condition. This kind of asymptotics goes by the name of \emph{Large Deviations Estimates}.

We give some preliminary definitions from \cite{BuDu19}.
Let $(\mathscr{X},d_{\mathscr{X}})$ be a metric space.
Given $\rho \in \mathscr{X}$ and a closed set $F \subset \mathscr{X}$, we denote by $d_\mathscr{X}(\rho,F)$ the distance of $\rho$ from the closed set $F$, namely $d_\mathscr{X}(\rho,F) := \inf_{f \in F} d_\mathscr{X}(\rho,f)$.

We say that a family of maps $I_{\rho_0} : \mathscr{X} \to [0,\infty]$ parametrized by $\rho_0 \in B$ has compact
level sets on compacts if, for every $K \subset B$ compact and $M < \infty$, the set 
\begin{align*}
\Phi_{K}^M := \bigcup_{\rho_0 \in K} \Phi_{\rho_0}^M,
\qquad
\Phi_{\rho_0}^M := \{ \rho \in \mathscr{X} : I_{\rho_0}(\rho) \leq M\},
\end{align*}
is compact in $\mathscr{X}$. Taking $K := \{ \rho_0 \}$ for some $\rho_0 \in B$, the previous condition implies that the sublevels $\Phi_{\rho_0}^M$ are compact in $\mathscr{X}$ for every $\rho_0 \in B$ and $M<\infty$, in particular each map $I_{\rho_0}:\mathscr{X} \to [0,\infty]$ is lower-semicontinuous. 

\begin{definition}\label{def:unif_Laplace}
Let $I_{\rho_0}$ have compact level sets on compacts and let $\{\rho^{\varepsilon,\rho_0}\}$ be $\mathscr{X}$-valued random variables parametrized by $\varepsilon \in (0,1)$ and $\rho_0 \in B$.
\begin{itemize}
\item 
We say that $\{\rho^{\varepsilon,\rho_0}\}$ satisfies the Laplace principle (LP) on $\mathscr{X}$ uniformly on compacts with speed $\varepsilon^2$ and rate function $I_{\rho_0}$ if for every $K \subset B$ compact and every bounded and continuous functions $h: \mathscr{X} \to \R$ it holds
\begin{align} 
\lim_{\varepsilon \downarrow 0}
 \sup_{\rho_0 \in K}
 \left|
\varepsilon^2 \log \mathbb{E} \left[ \exp(-\varepsilon^{-2}h(\rho^{\varepsilon,\rho_0}))  \right]
+
\inf_{\rho \in \mathscr{X}} \left\{ h(\rho) + I_{\rho_0}(\rho)\right\}
\right|
=
0.
\end{align}
\item 
We say that $\{\rho^{\varepsilon,\rho_0}\}$ satisfies the Large Deviations Principle (LDP) on $\mathscr{X}$ uniformly on compacts with speed $\varepsilon^2$ and rate function $I_{\rho_0}$ if for every $K \subset B$ compact, $\delta>0$, and $M \in (0,\infty)$ it holds
\begin{align} \label{eq:upper_FW}
\limsup_{\varepsilon \downarrow 0}
\sup_{\rho_0 \in K}
\sup_{m \leq M}
\left( \varepsilon^2 \log \PP \{ d_\mathscr{X} (\rho^{\varepsilon,\rho_0}, \Phi_{\rho_0}^m) \geq \delta\} + m \right) 
\leq 0,
\end{align}
\begin{align} \label{eq:lower_FW}
\liminf_{\varepsilon \downarrow 0}
\inf_{\rho_0 \in K}
\inf_{\rho \in \Phi_{\rho_0}^{M}}
\left( \varepsilon^2 \log \PP (d_\mathscr{X}(\rho^{\varepsilon,\rho_0},\rho)<\delta) 
+ I_{\rho_0}(\rho)
\right)
\geq 0.
\end{align}
\end{itemize}
\end{definition}

In this work we establish uniform Laplace and Large Deviations Principles that govern the convergence given by \autoref{thm:convergence}. 
More precisely, we show that the convergence $\rho^{\varepsilon,\rho_0} \to \rho^{\rho_0}$ is governed by a LP on $\mathcal{E}$ and a LDP on $\mathscr{E}$, uniformly as the initial condition $\rho_0$ varies in a compact subset of $B$ (the closed unit ball of $L^2_x\cap L^p_x$).
It is worth to mention that the uniform LP implies the uniform LDP by \cite[Proposition 1.14]{BuDu19}.

For the uniform LP in $\mathcal{E}$ we rely on the weak convergence approach developed by Budhiraja, Dupuis, and Maroulas in \cite{Bud_Dup} (see also \cite{budhiraja2000variational} and the book \cite{BuDu19}). 
Roughly speaking, using this approach one can deduce the validity of a uniform LP on the Polish space $\mathcal{E}$
as a consequence of stability of the convergence $\rho^\varepsilon \to \rho$ (plus additional technical verifications that we detail in \autoref{Sec_LDP}).
More precisely, one has to show that the convergence in law $\rho^\varepsilon \to \rho$, as $\mathcal{E}$-valued random variables, is robust with respect to the addition of predictable drifts $g^\varepsilon$ in the equation for $\rho^\varepsilon$, with $g^\varepsilon$ converging in law as weakly $L^2_t \mathcal{H}_0$-valued random variables. Here $\mathcal{H}_0 := \mathcal{Q}^{-1/2} \mathbb{H} = H^{d/2+\alpha}_x\cap \mathbb{H}  $ denotes the Cameron-Martin space associated to our noise, with inner product $\langle f,g\rangle_{\mathcal{H}_0} = \langle \mathcal{Q}^{-1/2}f,\mathcal{Q}^{-1/2}g \rangle = \langle f, g \rangle_{H^{d/2+\alpha}_x}$ for every $f,g \in \mathcal{H}_0$, see \cite[Lemma 3.1]{BaGrMa24}.
In order to apply the weak convergence approach, we show in \autoref{Sec_Meas_Stab} that given a sequence of initial conditions $\rho^\varepsilon_0 \to \rho_0 \in B$, a constant $N<\infty$, and a sequence
\begin{align*}
\{g^\varepsilon\}_{\varepsilon \in (0,1)} \subset \mathcal{P}_2^N
:=
\{g \mbox{ is } \{\mathcal{F}_t\}_{t \geq0}\mbox{ predictable}\,, \, \|g\|^2_{L^2_t \mathcal{H}_0} \leq N\,\,\, \mathbb{P} \mbox{-almost surely}\},
\end{align*}
such that $g^\varepsilon$ converges in law to a random variable $g$ in the sense specified above, then the associated solution $\rho^{\varepsilon, \rho_0^\varepsilon, g^\varepsilon}$ of \eqref{eq:rho_general} converges in law to $\rho^{\rho_0,g}$ (the unique renormalized solution of \eqref{eq:transport} with drift $b+g$ and initial condition $\rho_0$) as $\mathcal{E}$-valued random variables.
Notice that here $\rho^{\rho_0,g}$ is random, as $g$ is random.
This kind of stability follows by arguments similar to those in \cite{diperna1989ordinary}, except for the fact that extra carefulness is required to deal with the extra randomness coming from $g$ and the $g^\varepsilon$'s. We have the following:
\begin{thm} \label{thm:LDP}
In the same setting of the first part of \autoref{thm:convergence}, the family of processes $\{\rho^{\varepsilon,\rho_0}\}$, $\varepsilon \in (0,1)$, $\rho_0 \in B$, satisfies a Laplace Principle on $\mathcal{E}$ uniformly on compacts with speed $\varepsilon^2$ and rate function
\begin{align} \label{def_functional}
I_{\rho_0}(\rho) := \inf \left\{ \frac12 \|g\|^2_{L^2_t\mathcal{H}_0} \,:\, g \in L^2_t\mathcal{H}_0, \, \rho =\rho^{\rho_0,g}\right\}.
\end{align}
In particular, $\{\rho^{\varepsilon,\rho_0}\}$ satisfies also a Large Deviations Principle on $\mathcal{E}$ uniformly on compacts with same speed and rate function.
\end{thm}
Before moving to the LDP on $\mathscr{E}$, we state a corollary of previous \autoref{thm:LDP} on the anomalous dissipation measure $\mathcal{D}^\varepsilon : \Omega \to \mathcal{M}_+$ appearing in the local energy balance of $\rho^\varepsilon$: 
\begin{align*} 
d|\rho^\varepsilon|^2
+
b \cdot \nabla |\rho^\varepsilon|^2 dt
+
\varepsilon \circ dW_t \cdot \nabla |\rho^\varepsilon|^2 
=
- d\mathcal{D}^\varepsilon.
\end{align*}
Here $\mathcal{M}_+ := \mathcal{M}_+([0,T] \times \R^d)$ denotes the space of non-negative Radon measures on $[0,T] \times \R^d$ endowed with the total variation distance $d_{TV} (\mu,\nu) := \sup \{ |\mu(A) - \nu(A)| : A \mbox{ measurable} \}$ and we assume $\dvg \,b = 0$.
The stochastic integral in the local energy balance above is rigorously understood in It\=o sense, although we prefer to write it in Stratonovich form to make it clear that it formally gives no direct contribution in the global energy balance of $\rho^\varepsilon$. Therefore, the measure $d\mathcal{D}^\varepsilon$ prescribes exactly how much of the $L^2_x$ norm of $\rho^\varepsilon$ is dissipated at each $(t,x) \in [0,T] \times \R^d$. Existence of the dissipation measure has been proved in \cite{DrGaPa25+} in the case $b=0$. As our setting is slightly different, for the sake of completeness we show in \autoref{lem:beta_rho} how to adapt their proof to this situation.

Since $\rho^\varepsilon$ converges to the unique renormalized solution of \eqref{eq:transport}, which has constant $L^2_x$ norm when $\dvg \, b =0$, it is natural to expect $d\mathcal{D}^\varepsilon \to 0$ in a certain sense. Here we characterize this convergence with a LDP.
Since the LDP is not uniform on compacts of $\mathcal{M}_+$ we need the following version of LDP.
\begin{definition}
In the same setting as \autoref{def:unif_Laplace}, we say that $\{\rho^{\varepsilon,\rho_0}\}$ satisfies the pointwise LDP on $\mathscr{X}$ with speed $\varepsilon^2$ and rate function $I_{\rho_0}$ if for every $\rho_0^\varepsilon \to \rho_0$ in $B$, for every $F \subset \mathscr{X}$ closed and $G \subset \mathscr{X}$ open it holds
\begin{align} \label{eq:upper_V} 
\limsup_{\varepsilon \downarrow 0}
 \varepsilon^2 \log \PP \{ \rho^{\varepsilon,\rho_0^\varepsilon} \in F \}   \leq -\inf_{\rho \in F} I_{\rho_0}(\rho),
\end{align}
\begin{align} \label{eq:lower_V}
\liminf_{\varepsilon \downarrow 0}
 \varepsilon^2 \log \PP \{ \rho^{\varepsilon,\rho_0^\varepsilon} \in G \}  \geq -\inf_{\rho \in G} I_{\rho_0}(\rho)  .
\end{align}
\end{definition}

 We have the following:
\begin{cor} \label{cor:LDP_dissipation}
Under the same assumptions of \autoref{thm:convergence}, let $\dvg \, b = 0$. 
Then the family $\{\mathcal{D}^\varepsilon\}_{\varepsilon \in (0,1)}$ satisfies a pointwise LDP in $\mathcal{M}_+$ with speed $\varepsilon^2$ and good rate function $J(\mu) := \infty \mathbf{1}_{\{\mu \neq 0\}}$.
\end{cor}

The last result of this paper consists in the uniform LDP on the space $\mathscr{E}$.
\begin{thm} \label{thm:strong_LDP}
In the same setting of \autoref{thm:convergence}, suppose in addition $\dvg \,b=0$. Then the family of processes $\{\rho^{\varepsilon,\rho_0}\}$, $\varepsilon \in (0,1)$, $\rho_0 \in B$, satisfies a Large Deviations Principle on $\mathscr{E}$ uniformly on compacts with speed $\varepsilon^2$ and rate function $I_{\rho_0}$.
\end{thm}

Proving this result is much more difficult than proving the LP on $\mathcal{E}$, for several technical reasons.

First of all, the space $\mathscr{E}$ is not separable, and the weak convergence approach of \cite{Bud_Dup} cannot be directly applied. More generally, most of the classical results in Probability Theory that one needs to invoke when applying the weak convergence approach fail in non-separable spaces like $\mathscr{E}$, e.g. Skorokhod Representation Theorem and Jakubowski's version thereof (to show the convergence in law $\rho^{\varepsilon,\rho_0^\varepsilon,g^\varepsilon} \to \rho^{\rho_0,g}$) and Doob Measurability Theorem (to rigorously define the solution map $\mathcal{G}^\varepsilon$ in \cite[Assumption 1]{Bud_Dup}).

In addition, a second major obstacle that we need to overcome in the proof of \autoref{thm:strong_LDP} is the lack of compactness in $\mathscr{E}$ for solutions of \eqref{eq:rho_general}. Indeed, even if we knew that the whole family $\{\rho^{\varepsilon,\rho_0^\varepsilon,g^\varepsilon}\}_{\varepsilon \in (0,1)}$ takes values in a separable subspace of $\mathscr{E}$, we still would not have any tightness in $\mathscr{E}$ to be able to extract converging-in-law subsequences. 

Furthermore, the arguments to improve the convergence from $\mathcal{E}$ to $\mathscr{E}$ in \autoref{thm:convergence} cannot be directly generalized in the presence of random drifts $g^\varepsilon$'s, not only because the convergence in law is not strong enough (the argument is done basically $\omega$-wise), but also because of deeper measurability issues.
Indeed, as the space $\mathscr{E}$ is not separable, the Borel $\sigma$-field on the product space $\mathscr{E} \times \mathscr{E}$ does not coincide in general with the product of the Borel $\sigma$-fields on $\mathscr{E}$, i.e. $\mathscr{B}(\mathscr{E} \times \mathscr{E}) \neq \mathscr{B}(\mathscr{E}) \otimes \mathscr{B}(\mathscr{E})$, cf. \cite[Page 244]{billingsley2013convergence}. 
As a consequence, given two $\mathscr{E}$-valued random variables $\rho$ and $\rho'$, it is not generally true that the pair $(\rho,\rho')$ is a $\mathscr{E} \times\mathscr{E}$-valued random variable, since it could fail to be measurable.
Therefore, even if we knew that $g^\varepsilon$, $g$ are defined on the same probability space $(\Omega,\mathcal{F},\PP)$ and $g^\varepsilon \to g$ $\PP$-almost surely, we could not study the convergence in probability
\begin{align*}
\PP \{ d_\mathscr{E}(\rho^{\varepsilon,g^\varepsilon},\rho^g) > \delta\} \to 0,
\qquad
\mbox{as } \varepsilon \downarrow 0,
\end{align*}
as the event $\{ d_\mathscr{E}(\rho^{\varepsilon,g^\varepsilon},\rho^g) > \delta\}$ could be non-measurable with respect to the $\sigma$-field $\mathcal{F}$.

In view of these issues, here we develop a novel general method for proving a uniform LDP on a non-separable  metric space $\mathscr{X}$ that builds upon and extends the weak convergence approach of \cite{Bud_Dup}.
For the sake of presentation, we describe the method in the particular setting relevant to us, but one could easily adapt the ideas to other families of stochastic processes $\{X^{\varepsilon,x_0}\}$ solving another SPDE and taking values in $\mathscr{X}$. 
The key ingredients of our argument are the following:
\begin{enumerate}
    \item 
    We need a Polish space $\mathcal{E}$ such that $\mathscr{E} \subset \mathcal{E}$ and the laws of the processes $\{\rho^{\varepsilon,\rho^\varepsilon_0,g^\varepsilon}\}_{\varepsilon \in (0,1)}$ are tight in $\mathcal{E}$, whenever $\{g^\varepsilon\}_{\varepsilon \in (0,1)} \subset \mathcal{P}^N_2$ for some $N < \infty$ and $\{\rho^\varepsilon_0\}_{\varepsilon \in (0,1)} \subset K$ compact in $B$;
    \item 
    We need that for every compact $F \subset \mathscr{E}$ and $\delta \geq 0$, the set $\{ \rho \in \mathscr{E} : d_\mathscr{E}(\rho,F) \leq \delta \}$ is measurable with respect to the $\sigma$-field on $\mathscr{E}$ generated by the subspace topology on $\mathscr{E} \subset \mathcal{E}$ (that is weaker than the topology on $\mathscr{E}$);
    \item 
    We need that, in an auxiliary probability space where $\rho^{\varepsilon,\rho^\varepsilon_0,g^\varepsilon} \to \rho^{\rho_0,g}$ almost surely as $\mathcal{E}$-valued random variables up to subsequences (for example the one usually obtained by Skorokhod Representation Theorem), the convergence in $\mathcal{E}$ can be improved to a convergence in $\mathscr{E}$ up to extracting sub-subsequences;
    \item 
    We need that, for every compact $F \subset \mathscr{E}$ and $L>0$, the expectation of $d_\mathscr{E}(\rho^{\varepsilon,\rho^\varepsilon_0,g^\varepsilon},F)\wedge L$ converges to the expectation of  $d_\mathscr{E}(\rho^{\rho_0,g},F)\wedge L$.
\end{enumerate}
We believe this approach is very powerful and can produce satisfactory results in several settings. 
For instance, it allows one to show the validity of \emph{sharp} LDPs, namely LDPs in the smallest possible space $\mathscr{X}$ in which a stochastic processes takes values. In fact, more often than not the weak convergence approach is applied in a (Polish) space $\mathcal{X}$ where the laws of the processes are tight as a consequence of uniform bounds in $\mathscr{X}$ and compact embeddings. In particular, the topology in $\mathcal{X}$ is often weaker than the one in $\mathscr{X}$.

\subsection{Bibliographic discussion} \label{ssec:biblio}
Non-uniqueness of weak solutions to the transport equation \eqref{eq:transport} has been established via convex integration techniques in the series of papers \cite{MoSz18,MoSa20,Mo20}. Their construction has been refined in \cite{BrCoDe21,BrCoKu}, where the authors were even able to prescribe the sign of $\rho$. 
We also mention the contributions \cite{ChLu21,ChLu24}, where the authors produce non-unique ``intermittent'' solutions (i.e. with low integrability in time), and the very recent \cite{Colombo^2Kumar}. Even if we are in the presence of a non-uniqueness phenomenon we can hope for a selection mechanism, due to the uniqueness of renormalized solutions described by \cite{diperna1989ordinary}. This has been investigated in the deterministic literature by \cite{BoCiCr22}, who introduced a viscous regularization of the system, and extending previous results by
\cite{CrSp15, CoDrEl22, NuSeWi21, CiCrSp21} for the 2D Euler equations. As described above, we address the problem of selection among weak solutions of the transport equation from a different viewpoint, employing regularizing features of the transport noise. 
Starting from the seminal work \cite{veretennikov1981strong}, it is nowadays well known that noise can help restore uniqueness in the solution theory of ordinary and partial differential equations.  Without the pretentiousness to be exhaustive, let us mention some instances of the realization of this phenomenon. As shown in \cite{DP_regular_1, DP_regular_2,cerrai2013pathwise}, additive noise can restore uniqueness in infinite dimensional stochastic differential equations. For genuine partial differential equations, the situation is much more complex and additive noise does not produce satisfactory results. On the contrary, a Stratonovich noise of transport type has proven to provide a much more robust framework to restore uniqueness in partial differential equations. 
Smooth transport noise has been considered in 
\cite{FlaGubPri10, delarue2014noise, flandoli2011full, Be_Fl, flandoli2021high,flandoli2021delayed, agresti2024global}, restoring uniqueness and/or preventing blow-up in several equations of interest. 
More recently, the Kraichnan noise of \cite{Kr68} has been considered in \cite{CoMa23+}, showing that it regularizes 2D Euler equations and providing uniqueness in a setting where it is false without noise, cf. \cite{vishik2018instability,vishik2018instability2}. The result of \cite{CoMa23+} has been subsequently extended by \cite{JiLu24,BaGaMa24,JiLu25+} to more general settings, such as mSQG and Boussinesq equations. However, the selection properties of this noise were unknown until now. This is the content of our main results \autoref{thm:convergence}, \autoref{thm:LDP}, \autoref{thm:strong_LDP}, which address this problem by identifying the zero-noise limit and establishing Large Deviations Principles. In particular, as already discussed above, we apply the so-called weak convergence approach to get \autoref{thm:LDP} and we generalize it to a more general setting in order to prove \autoref{thm:strong_LDP}. The weak convergence approach to Large Deviations was developed by \cite{budhiraja2000variational, Bud_Dup,salins2019uniform}, and it is nowadays a very popular method for establishing LDPs for the laws of solutions to SPDEs. We mention, among many contributions, the works \cite{FeGe23,brzezniak2017large, CeDe19, CePa19, cerrai2024nonlinear, bessaih2012large,galeati2024ldp,butori2024large}.

\section{Convergence to Renormalized Solution}\label{sec:convergence}

The goal of this section is to prove \autoref{thm:convergence}.
Its proof is a direct consequence of the following lemma.
\begin{lem}\label{lem_Backward_convergence}
Let $\rho^{\varepsilon}$ be the unique probabilistically strong, analytically weak solution of the stochastic transport equation \eqref{eq:stoch_trans_Ito} with initial condition $\rho_0\in L^{1}_x\cap L^{\infty}_x \cap B$ given by \autoref{prop_well_posed}. 
Then, as $\varepsilon \downarrow 0$, $\rho^\varepsilon$
converges in probability to the unique renormalized solution of the transport equation \eqref{eq:transport} given by \cite[Theorem II.3]{diperna1989ordinary}, as $\mathcal{E}$-valued random variables. If moreover $\operatorname{div}\, b=0$ the convergence in probability above holds as $L^\infty_tL^2_x$-valued random variables. 
\end{lem}

Before going on, let us mention that the notion of convergence in probability in the non-separable metric spaces $\mathscr{E}$ and $L^\infty_t L^2_x$, which is the content of \autoref{thm:convergence} and \autoref{lem_Backward_convergence}, is well-defined since the limit object is deterministic. 
Indeed, the events required to define the notion of convergence in probability are measurable with respect to the $\sigma$-field $\mathcal{F}$, cf. \cite[Page 27]{billingsley2013convergence}.

\begin{proof}[Proof of \autoref{thm:convergence}, assuming \autoref{lem_Backward_convergence}]
Let us denote by $\rho$ the unique renormalized solution of the transport equation with initial datum $\rho_0$ given by \cite{diperna1989ordinary}. Let us fix $\delta>0$.  Due to the stability of renormalized solutions given by \cite[Theorem II.7]{diperna1989ordinary}  we can find $0<\theta \ll 1$ small enough and initial conditions $\rho_0^1\in L^1_x\cap L^{\infty}_x$ and $ \rho_0^2\in L^2_x \cap L^p_x$ such that
\begin{align*}
    \rho_0=\rho_0^1+\rho_0^2,
    \quad 
    \norm{\rho_0^1}_{L^2_x\cap L^p_x}\leq 1,
    \quad  
    \norm{\rho_0^2}_{L^2_x \cap L^p_x}\leq \theta,
\end{align*}
and moreover, denoting by $\rho^1 $ (resp. ${\rho}^{2}$) the unique renormalized solutions of the transport equation with initial condition ${\rho}_0^1$ (resp. ${\rho}_0^2$):
\begin{align}\label{estimates_renormalized_splitting}
 \rho={\rho}^{1}+{\rho}^{2},
 \quad
 \ d_{\mathscr{E}}({\rho},\rho^1) \leq \exp(\|\dvg \, b \|_{L^1_tL^\infty_x})\theta \leq \frac{\delta}{8}.
\end{align}
Secondly, let us choose a sequence $\varepsilon_n \downarrow 0$ such that $\varepsilon_n<\overline{\varepsilon}$ for every $n \in \N$, where $\overline{\varepsilon}$ is taken sufficiently small such that $\lVert \rho_0-\rho_0^{\varepsilon_n}\rVert_{L^2_x \cap L^p_x} \leq \theta$ for every $\varepsilon_n$. This is possible since $\rho_0^\varepsilon \to \rho_0$ in $B$ by assumption.
Furthermore, let us introduce 
\begin{align*}
\rho_0^{\varepsilon_n,1} := \rho_0^1\in L^1_x\cap L^{\infty}_x,
\quad
\rho_0^{\varepsilon_n,2} := \rho_0^{\varepsilon_n}-\rho_0+\rho_0^2\in L^2_x \cap L^p_x.
\end{align*}
In this way we have
\begin{align*}
\rho_0^{\varepsilon_n,1}+\rho_0^{\varepsilon_n,2}=\rho_0^{\varepsilon_n}, 
\quad
\sup_{\varepsilon_n\leq \overline{\varepsilon}}\norm{\rho_0^{\varepsilon_n,2}}_{L^2_x \cap L^p_x}\leq 2\theta.    
\end{align*} 
Let us denote by $\rho^{\varepsilon_n,1}$ (resp. $\rho^{\varepsilon_n,2}$) the unique weak solution of the stochastic transport equation with initial condition $\rho_{0}^{\varepsilon_n,1}$ (resp. $\rho_0^{\varepsilon_n,2}$) given by \autoref{prop_well_posed}. 
By linearity of solutions to \eqref{eq:rho_general} and the estimate \eqref{eq:bound_rho_Linfty_Lp} on the $L^\infty_\omega L^\infty_t (L^2_x \cap L^p_x)$ norm of the solution of the stochastic transport, we have
\begin{align*}
    \rho^{\varepsilon_n}
=
\rho^{\varepsilon_n,1}+\rho^{\varepsilon_n,2}
\end{align*} and
\begin{align} 
\label{estimate_stochastc_splitting}
\sup_{\varepsilon_n\leq \overline{\varepsilon}}
\sup_{t\in [0,T]}& \norm{\rho^{\varepsilon_n,2}_t}_{ L^2_x \cap L^p_x}
\leq 2 \exp(\|\dvg \, b \|_{L^1_tL^\infty_x})\theta  
\leq \frac{\delta}{4},
\quad \mathbb{P} \mbox{-almost surely}.
\end{align}
Therefore, combining  \eqref{estimates_renormalized_splitting} and \eqref{estimate_stochastc_splitting} and using $d_{\mathcal{E}}(u,v) \leq 2 \sup_{t \in [0,T]}\| u_t-v_t \|_{L^2_x \cap L^p_x}$ we obtain
\begin{align*}
   d_{\mathcal{E}}(\rho^{\varepsilon_n},\rho)
   &\leq 
   d_{\mathcal{E}}(\rho^{\varepsilon_n},{\rho}^{\varepsilon_n,1})
   +
   d_{\mathcal{E}}(\rho^{\varepsilon_n,1},{\rho}^{1})
   +
   d_{\mathcal{E}}({\rho}^1,{\rho})
   \\
   &\leq 
   \frac{3}{4} \delta
   +
   d_{\mathcal{E}}(\rho^{\varepsilon_n,1},{\rho}^{1}),
   \quad \mathbb{P} \mbox{-almost surely}.
\end{align*}
In particular, it holds
\begin{align*}
\mathbb{P}\left(d_{\mathcal{E}}(\rho^{\varepsilon_n},{\rho})>\delta\right)
\leq 
\mathbb{P}\left(d_{\mathcal{E}}(\rho^{\varepsilon_n,1},{\rho}^{1})>\frac{\delta}{4}\right) \to 0
\end{align*}
as $n \to \infty$ by the first part of \autoref{lem_Backward_convergence}. Since the limit $\rho$ does not depend on the choice of the subsequence $\varepsilon_n$, we deduce that the whole sequence $\rho^\varepsilon$ is converging in probability towards $\rho$.

It only remains to prove the convergence $\rho^\varepsilon \to \rho$ in $\mathscr{E}$, assuming $\dvg \, b = 0$. 
Arguing as before we have along any subsequence $\varepsilon_n \downarrow 0$
\begin{align*}
\mathbb{P}\left(d_{\mathscr{E}}(\rho^{\varepsilon_n},{\rho})>\delta\right)
\leq 
\mathbb{P}\left(d_{\mathscr{E}}(\rho^{\varepsilon_n,1},{\rho}^{1})>\frac{\delta}{4}\right),
\end{align*}
where $\rho^{\varepsilon_n,1}$ and $\rho^1$ are as above.
By \autoref{lem_Backward_convergence} we know that $\rho^{\varepsilon_n,1}$ converges to $\rho^1$ in probability when $n \to \infty$, as $L^\infty_t L^2_x$ random variables. In order to get the convergence in $\mathscr{E}$ we just have to prove
\begin{align} \label{eq:convergence_in_probability_Lp}
\mathbb{P}\left( \sup_{t \in [0,T]}\|\rho^{\varepsilon_n,1}_t-{\rho}^{1}_t\|_{L^p_x}> \delta'\right) \to 0,
\quad
\forall \delta'>0.
\end{align}
Since both $\rho^{\varepsilon_n,1}$ and $\rho^1$ have initial condition $\rho_0^1 \in L^1_x \cap L^\infty_x$, by \eqref{eq:bound_rho_Linfty_Lp} we have
\begin{align*}
\sup_{t \in [0,T]}\|\rho^{\varepsilon_n,1}_t-{\rho}^{1}_t\|_{L^1_x \cap L^\infty_x}
\leq
2\| \rho_0^1 \|_{L^1_x \cap L^\infty_x}.
\end{align*}
Since $p \in (1,\infty)$, by interpolation we have for some $\beta = \beta(p) \in (0,1)$
\begin{align*}
\sup_{t \in [0,T]}\|\rho^{\varepsilon_n,1}_t-{\rho}^{1}_t\|_{L^p_x}
\leq 2 \| \rho_0^1 \|_{L^1_x \cap L^\infty_x}^{\beta}
\sup_{t \in [0,T]}\|\rho^{\varepsilon_n,1}_t-{\rho}^{1}_t\|_{L^2_x}^{1-\beta} 
,
\end{align*}
implying \eqref{eq:convergence_in_probability_Lp} since, by assumption, $\rho^{\varepsilon_n,1}$ converges to $\rho^1$ in probability as $L^\infty_t L^2_x$ random variables.
The thesis then follows again by the arbitrariness of the subsequence $\varepsilon_n$.
\end{proof}

The remainder of this section is devoted to the proof of \autoref{lem_Backward_convergence}.
We split the proof into two parts: 
\begin{itemize}
    \item In \autoref{weak_convergence_eps} we prove the convergence of $\rho^\varepsilon$ in probability as $\mathcal{E}$-valued random variables, and we identify the limit as an analytically weak solution of the deterministic transport equation \eqref{eq:transport} with initial condition $\rho_0 \in L^1_x \cap L^\infty_x \cap B$. By DiPerna-Lions theory, the limit is the unique renormalized solution.
    \item In \autoref{strong_convergence_eps} we upgrade the convergence in $\mathcal{E}$ to the convergence in $L^\infty_t L^2_x$.
\end{itemize}

As one of the ingredients of our proofs, we will rely on the following compactness criterion, which is a simple consequence of \cite[Corollary 5]{simon1986compact} and \cite[Lemma 2.1]{brzezniak2014existence}. We state and prove it here for future references. We will later apply the lemma with $\gamma,\sigma,$ and $\beta = 1-\alpha-\delta$ as in the statement of \autoref{prop_well_posed}.
\begin{lem}\label{Compactness_Criterion}
Let $\gamma,\sigma>0$ and $\beta \in \R$. 
\begin{enumerate}
\item 
A set $\mathcal{K}\subset \mathcal{E}$ is relatively compact in $\mathcal{E}$ if \begin{align}\label{time_compactness_lemma}
\sup_{u\in \mathcal{K}}\norm{u}_{C^{\gamma}_t\tilde{H}^{-\sigma}_x}<+\infty.
\end{align}
\item 
A set $\mathcal{K}'\subset \mathcal{E}\cap L^2_t{\tilde{H}}^{\beta-}_x$ is relatively compact in $\mathcal{E}\cap L^2_t{\tilde{H}}^{\beta-}_x$ if \eqref{time_compactness_lemma} holds and  
\begin{align}\label{space_compactness_lemma}
\sup_{u\in \mathcal{K}'}\norm{u}_{L^2_tH^{\beta}_x}<+\infty.
\end{align}
\end{enumerate}
\end{lem}
\begin{rmk}
Notice that a uniform bound in $L^2_t (L^2_x \cap L^p_x)$ is implicitly assumed in the definition of the space $\mathcal{E}$. 
\end{rmk}
\begin{proof}
    In order to prove the first claim, it is enough to show that for each $0<\eta\leq \sigma$ and sequence $\{u_n\}_{n \in \N} \subset \mathcal{K}$ there exists a (non-relabeled) subsequence and $u \in C_t\mathcal{B}\cap C_t\tilde{H}^{-\eta}_x$ such that $u_n\rightarrow u$ in $C_t\mathcal{B}\cap C_t\tilde{H}^{-\eta}_x$.   
By \cite[Lemma A.4]{BaGaMa24} we have 
\begin{align*}
    L^2_x \Subset \tilde{H}^{-\eta}_x \subset \tilde{H}^{-\sigma}_x,
\end{align*}
with the first embedding being compact.
Therefore by \cite[Corollary 5]{simon1986compact}, the embedding 
\begin{align}\label{first_compact_embedding}
    L^{\infty}_tL^2_x\cap C^{\gamma}_t\tilde{H}^{-\sigma}_x \Subset C_t\tilde{H}^{-\eta}_x
\end{align}
is compact. 
Next, denote $R:= \exp(\| \dvg \, b \|_{L^1_t L^\infty_x})$. In virtue of \cite[Chapter 3.1 Lemma 1.4]{temam2024navier} we have the embedding
\begin{align}\label{second_compact_embedding}
   \{ u \in  L^\infty_t(L^2_x \cap L^p_x) \,: \, \|u\|_{L^\infty_t (L^2_x\cap L^p_x)} \leq R \}\cap C^{\gamma}_t\tilde{H}^{-\sigma}_x \subset C_t \mathcal{B}.
\end{align}
We want to show that this embedding is compact.
Since $C_t\mathcal{B}$ is metric, it is enough to study sequential compactness. Let $\{u_k\}_{k \in \N}$ be a bounded sequence in the former space. By the compact embedding \eqref{first_compact_embedding} there exists $\eta > 0$, a (non-relabeled) subsequence of $u_k$'s, and $u \in C_t \tilde{H}^{-\eta}_x$ such that 
\begin{align*}
    u_{k}\rightarrow u\quad \text{in }C_t \tilde{H}^{-\eta}_x.
\end{align*}
The previous convergence and the uniform bound in $L^\infty_t L^2_x$ imply, by \cite[Lemma 2.1]{brzezniak2014existence}, that the limit $u$ belongs to $C_t\mathcal{B}_2$, where $\mathcal{B}_2$ is the closed ball in $L^2_x$ of radius $\exp(\|\dvg \ b\|_{L^1_t L^\infty_x})$ endowed with the weak topology, and moreover $u_{k}\rightarrow u$ in $C_t\mathcal{B}_2$. 
In particular for each $h\in L^2_x$
\begin{align*}
    \sup_{t\in [0,T]} \left| \int_{\R^d} (u_{k}(t,x)-u(t,x))h(x) dx\right|\rightarrow 0.
\end{align*}
In order to prove convergence in $C_t \mathcal{B}$, it suffices to show that the same holds also for each $h\in L^{p'}_x$, where $p' \in (1,\infty)$ is the H\"older conjugate of $p$. 
For each $c>0$, let $\tilde{h}\in C^{\infty}_c(\R^d)$ be such that $\|{\tilde{h}-h}\|_{L^{p'}_x}<c$, then
\begin{align*}
    \limsup_{k\rightarrow +\infty}&  \sup_{t\in [0,T]}\left| \int_{\R^d} (u_{k}(t,x)-u(t,x))h(x) dx \right|  
    \\
    &\leq 
    \limsup_{k\rightarrow +\infty}  \sup_{t\in [0,T]}\left| \int_{\R^d} (u_{k}(t,x)-u(t,x)) \tilde h(x) dx \right|  + 2\|{\tilde{h}-h}\|_{L^{p'}_x}  
    \leq 2c.
\end{align*}
Combining the results above the first claim follows by the completeness of $\mathcal{E}$. 

The second one is analogous up to observing that by \cite[Lemma A.4]{BaGaMa24} and \cite[Corollary 5]{simon1986compact} also the embdedding
\begin{align*}
    L^2_tH^{\beta}_x\cap C_t^{\gamma}H^{-\sigma}_x
    \Subset L^2_t\tilde{H}^{\beta-\eta}_x
\end{align*}
is compact. We omit the easy details.
\end{proof}

\subsection{Proof of \autoref{lem_Backward_convergence}: Convergence in $\mathcal{E}$ to the unique renormalized solution}\label{weak_convergence_eps}
We will show that for each sequence $\varepsilon_n\rightarrow 0$, $\rho^{\varepsilon_n}$ converges in law to $\rho$ in the topology of $\mathcal{E}$. By Gyongy-Krylov criterion \cite[Lemma 1.1]{gyongy1996existence}, one can upgrade the convergence in law to a convergence in probability, giving the desired claim. Gyongy-Krylov criterion applies here by uniqueness of the limit, since $\mathcal{E}$ is a Polish space.
For the matter of readability we split the proof in two steps.

\emph{Step 1: Tightness and passage to an auxiliary probability space.}
Recall that we are working under the assumption $\rho_0 \in L^1_x \cap L^\infty_x \cap B$.
Thanks to the embedding \eqref{second_compact_embedding} and the estimates in the proof of \autoref{prop_well_posed}, more specifically estimates \eqref{eq:bound_rho_Linfty_Lp}, \eqref{eq:bound_rho_holder}, and \eqref{eq:bound_rho_sobolev}, we have the bounds 
\begin{align}\label{estimates_original_spaces}
\sup_{n\in \mathbb{N}}
\left(\mathbb{E}\lVert\rho^{\varepsilon_n}\rVert_{C^{\gamma}_tH^{-\sigma}_x} { + \varepsilon_n^2\mathbb{E} \| \rho^{\varepsilon_n} \|_{L^2_t H^{1-\alpha-\delta}_x}^2} \right)&\lesssim 1,
\\  \label{bounds_original_0}
\sup_{n\in \N}\norm{\rho^{\varepsilon_n}}_{L^{\infty}_t(L^1_x\cap L^{\infty}_x)}
&\lesssim 1,
\\
\sup_{n\in \N}\sup_{t\in [0,T]}\norm{\rho^{\varepsilon_n}_t}_{L^2_x\cap L^p_x} &\leq \exp(\|\dvg \, b \|_{L^1_tL^\infty_x}) \quad \mathbb{P} \mbox{-almost surely}.\label{bounds_original_1}  
\end{align}
Let us denote $\mathbb{B}_M$ the closure in $\mathcal{E}$ of the intersection between $\mathcal{E}$ and the ball in $ C^{\gamma}_t\tilde{H}^{-\sigma}_x$ of radius $M < \infty$ centered in $0$. Recall that elements in $\mathcal{E}$ are bounded in $L^\infty_t (L^2_x \cap L^p_x)$.
By \autoref{Compactness_Criterion}, $\mathbb{B}_M$ is a compact set in $\mathcal{E}$. By \eqref{bounds_original_1}, Markov's inequality and \eqref{estimates_original_spaces} we have
\begin{align}\label{estimate_1_compactness_law}
    \sup_{n\in \N}\mathbb{P}(\rho^{\varepsilon_n}\notin {\mathbb{B}_M} )  \lesssim \sup_{n\in \N}\mathbb{P}(\lVert \rho^{\varepsilon_n}\rVert_{C^{\gamma}_t H^{-\sigma}_x}>M) \lesssim \frac{1}{M},
\end{align}
which can be made arbitrarily small by choosing $M$ properly. Secondly, $W=\sqrt{\mathcal{Q}}\mathcal{W}$, where $\mathcal{W}$ is a cylindrical Brownian motion on $\mathbb{H}$. The latter can be identified by a family of real independent Brownian motions $\{W^{k}\}_{k\in \N }$, i.e. as a process on $C_t\R^{\N}$. Notice that the law of $\{W^{k}\}_{k\in \N }$ is independent of $n$.

Next we want to apply Prokhorov Theorem and Jakubowski version of Skorokhod Representation Theorem \cite{jakubowski1998almost,Brzezniak_skoro}.
The arguments are relatively standard and are detailed, for example, in \cite[Chapter 2]{flandoli2023stochastic}.
Up to passing to non-relabeled subsequences, we can find an auxiliary filtered probability space $(\tilde{\Omega},\tilde{\mathcal{F}},\{\tilde{\mathcal{F}}_t\}_{t \geq 0},\tilde{\mathbb{P}})$ and processes $(\tilde{\rho}^{\varepsilon_n},\{\tilde{W}^{\varepsilon_n,k}\}_{k\in \N})$ and $(\tilde{\rho},\{\tilde{W}^{k}\}_{k\in \N})$ on it, such that:
$i$) the processes $(\tilde{\rho}^{\varepsilon_n},\{\tilde{W}^{\varepsilon_n,k}\}_{k\in \N})$ and $({\rho}^{\varepsilon_n},\{{W}^{k}\}_{k\in \N})$ have the same law on $\mathcal{E}\times C_t\R^{\N}$; and $ii$) the following convergences are valid $\tilde{\PP}$-almost surely 
\begin{align} \label{convergence_rho}
    \tilde{\rho}^{\varepsilon_n}&\rightarrow \tilde{\rho} \qquad\qquad\, \textit{in } \mathcal{E},
    \\
    \{\tilde{W}^{\varepsilon_n,k}\}_{k\in \N}&\rightarrow \{\tilde{W}^{k}\}_{k\in \N} \quad \textit{in } C_t\R^{\N}.  \nonumber
\end{align}
Moreover, since $(\tilde{\rho}^{\varepsilon_n},\{\tilde{W}^{\varepsilon_n,k}\}_{k\in \N})$ has the same law of $({\rho}^{\varepsilon_n},\{{W}^{\varepsilon_n,k}\}_{k\in \N})$, the bounds \eqref{estimates_original_spaces} and \eqref{bounds_original_1} continue to hold in the auxiliary probability space and $\tilde{\rho}^{\varepsilon_n}$ is the unique probabilistically strong, analytically weak solution on the probability space $(\tilde{\Omega},\tilde{\mathcal{F}},\{\tilde{\mathcal{F}}_t\}_{t \geq 0},\tilde{\mathbb{P}})$ of the stochastic transport equation with initial datum $\rho_0$, given by \autoref{prop_well_posed}, with noise $\tilde{W}^{\varepsilon_n}=\sqrt{\mathcal{Q}}\tilde{\mathcal{W}}^{\varepsilon_n}$. 
As a consequence of \autoref{prop_well_posed} and the fact that $\rho_0 \in L^1_x \cap L^\infty_x$, also the bound \eqref{bounds_original_0} is valid in the auxiliary probability space. Furthermore, $\tilde{\rho}$ inherits the same $L^\infty_t L^\infty_x$ bound by relatively standard arguments (see for instance \cite[Lemma 3.5]{flandoli2021scaling}), namely
\begin{align} \label{eq:bound_tilde_rho_auxiliary}
    \norm{\tilde{\rho}}_{L^{\infty}_t L^{\infty}_x}\lesssim 1,\quad  \tilde{\mathbb{P}} \mbox{-almost surely}.
\end{align}

\emph{Step 2: Identification of the limit.} 
By \eqref{convergence_rho} we have $\tilde{\rho}^{\varepsilon_n}\rightarrow \tilde{\rho}$ in $\mathcal{E}$ almost surely, and we want to show that $\tilde\rho$ is the unique renormalized solution of the transport equation \eqref{eq:transport} with initial condition $\rho_0$.
To simplify notation we drop the $\sim$ in all the auxiliary objects defined above. 
Let $\phi\in C^{\infty}_c([0,T)\times \R^d)$. Due to relation \eqref{convergence_rho}, and more specifically to the convergence in $C_t\mathcal{B}$, it is easy to show that the following limits are valid $\PP$-almost surely as $n \to \infty$:
\begin{align*}
   \int_0^T \langle \rho^{\varepsilon_n}_s,\partial_s \phi_s\rangle ds&\rightarrow  \int_0^T \langle \rho_s,\partial_s \phi_s\rangle ds,
   \\
   \int_0^T \langle \rho_s^{\varepsilon_n}, \operatorname{div}\, b \phi_s \rangle ds&\rightarrow \int_0^T \langle \rho_s, \operatorname{div}\, b\phi_s \rangle ds, 
   \\ 
   \int_0^T \langle \rho_s^{\varepsilon_n}, b\cdot\nabla\phi_s \rangle ds&\rightarrow \int_0^T \langle \rho_s, b\cdot\nabla\phi_s \rangle ds, 
   \\ 
   \varepsilon_n\int_0^T \langle \rho^{\varepsilon_n}_s,\Delta \phi_s\rangle ds&\rightarrow 0.
\end{align*}
Moreover, thanks to the regularizing properties of the covariance of Kraichnan noise (see for example \cite[equation (2.2)]{GaGrMa24+}) and Sobolev embedding, it holds  
\begin{align*}
    \varepsilon_n^2\mathbb{E}\left[\lvert\sum_{k\in \N}\int_0^T\langle \rho^{\varepsilon_n}_s,\sigma_k\cdot \nabla\phi_s\rangle dW^{\varepsilon_n,k}_s\rvert^2\right]&\lesssim \varepsilon_n^2 \mathbb{E}\left[\int_0^T \lVert \nabla \phi_s \rho^{\varepsilon_n}_s\rVert^2_{L^1_x} ds\right]\\ & \lesssim \varepsilon_n^2 \mathbb{E}\left[\int_0^T \lVert  \rho^{\varepsilon_n}_s\rVert^2_{L^2_x} ds\right]\rightarrow 0.
\end{align*}
Therefore, up to passing to further non-relabeled subsequences, also 
\begin{align*}
    \varepsilon_n \sum_{k\in \N}\int_0^T\langle \rho^{\varepsilon_n}_s,\sigma_k\cdot \nabla\phi_s\rangle dW^{\varepsilon_n,k}_s\rightarrow 0,\quad\mathbb{P} \mbox{-almost surely}.
\end{align*}
Recall also that $\rho_0^\varepsilon \to \rho_0$ in $B$. Consequently, for each $\phi\in C^{\infty}_c([0,T)\times \R^d)$ we have the almost sure identity:
 \begin{align*}
     \langle \rho_0,\phi_0\rangle
     +
     \int_0^T  \langle \rho_s, \operatorname{div}\, b\phi_s\rangle ds+\int_0^T  \langle \rho_s, b\cdot\nabla\phi_s\rangle ds + \int_0^T \langle\rho_s,\partial_s\phi_s\rangle ds = 0.
 \end{align*}
By standard density argument, there exists a $\PP$-negligible set such that the relation above holds on its complementary for all $\phi\in C^{\infty}_c([0,T)\times \R^d)$. Namely, $\rho$ is $\mathbb{P} \mbox{-almost surely}$ a weak solution of the transport equation with initial condition $\rho_0$. 
Due to the regularity \eqref{eq:bound_tilde_rho_auxiliary} of $\rho$ and our assumptions on $b$, by DiPerna-Lions theory there exists a unique weak solution of the transport equation, and it coincides with the renormalized one. 
Therefore, $\mathbb{P} \mbox{-almost surely}$ we have $\rho^{\varepsilon_n}\rightarrow \rho$ in $\mathcal{E}$, which implies the convergence in law and completes the proof.

\subsection{Proof of \autoref{lem_Backward_convergence}: Strong convergence to the unique renormalized solution}\label{strong_convergence_eps}
Now we assume $\operatorname{div}\,b=0$, and we want to show convergence in probability $\rho^\varepsilon \to \rho$ as $L^\infty_tL^2_x$-valued random variables.

We divide the proof in some steps following somehow the ideas of \cite[Lemma 3.3]{CiCrSp21}. Preliminarily we observe that \autoref{prop_well_posed} and the results of \autoref{weak_convergence_eps} continue to hold on the time interval $[0,T+1]$, by extending the drift $b$ to be identically zero on $(T,T+1]$. 
This is for us just a convenient technical construction that we exploit in order to get the convergence on the desired time interval $[0,T]$.
To keep the notation as concise as possible, we denote $L^r_t$, $C_t$, etc. function spaces defined on the time interval $[0,T]$, and $L^r_{t,*}$, $C_{t,*}$, etc. their counterparts defined on the time interval $[0,T+1]$.
We split the proof in several steps.

\emph{Step 1: Strong convergence in $L^2_{t,*}L^2_x.$}
Due to the first part of \autoref{thm:convergence}, we already know that (up to passing to subsequences) 
\begin{align*}
    \rho^{\varepsilon_n}\rightarrow \rho\quad\text{in }\mathcal{E}_* := C_{t,*}\mathcal{B}\cap C_{t,*}\tilde{H}^{-},
    \quad \mathbb{P} \mbox{-almost surely}.
\end{align*}
The latter additionally implies that on the same full-probability set 
\begin{align*}
    \rho^{\varepsilon_n}\rightharpoonup \rho, 
    \quad\text{weakly in }L^2_{t,*}L^2_x.
\end{align*}
Therefore, in order to show $\rho^{\varepsilon_n}\rightarrow\rho$ strongly in $L^2_{t,*}L^2_x$, it is sufficient to show convergence of the norms.
By the bound \eqref{eq:bound_rho_Linfty_Lp} with $\dvg\,b=0$, the fact that $\rho^{\varepsilon_n} \in C_{t,*}\mathcal
B$ $\PP$-almost surely, and lower-semicontinuity of the $L^2_x$ norm with respect to weak convergence, we know that for every $n \in \N$ 
\begin{align*}
   \sup_{t \in [0,T+1]} \lVert \rho^{\varepsilon_n}_t\rVert_{L^2_x} \leq \lVert \rho_0\rVert_{L^2_x},
    \quad \mathbb{P} \mbox{-almost surely}, 
\end{align*}
whereas the limit $\rho$ satisfies $\lVert \rho_t\rVert_{L^2_x} = \lVert \rho_0\rVert_{L^2_x}$ for every $t \in [0,T+1]$, being it renormalized.
Therefore, recalling $\rho^{\varepsilon_n} \to \rho$ in $C_{t,*}\mathcal{B}$ and exploiting again lower-semicontinuity of the $L^2_x$ norm with respect to the weak convergence, we have $\PP$-almost surely:
\begin{align*}
\norm{\rho}_{L^2_{t,*} L^2_x}^2
&\leq 
\liminf_{n\rightarrow +\infty}\norm{\rho^{\varepsilon_n}}_{L^2_{t,*} L^2_x}^2
\leq 
\limsup_{n\rightarrow +\infty}\norm{\rho^{\varepsilon_n}}_{L^2_{t,*} L^2_x}^2
\leq 
(T+1)\lVert \rho_0\rVert_{L^2_x}^2
=
\norm{\rho}_{L^2_{t,*} L^2_x}^2,
\end{align*}
that implies the claim. 

\emph{Step 2: Essentially uniform convergence of the $L^2_x$ norm on bounded sets.}
Recall the bound \eqref{eq:bound_tilde_rho_auxiliary} on $\rho^{\varepsilon_n}$. Since $\rho_0\in L^1_x\cap L^{\infty}_x$, the same holds also for the renormalized solution $\rho$ of \eqref{eq:transport}, namely
    \begin{align}\label{bounds_renormalized}
   \sup_{t \in [0,T+1]} \lVert \rho^{\varepsilon_n}_t\rVert_{L^1_x \cap L^\infty_x} 
   +
   \sup_{t \in [0,T+1]}\lVert \rho_t\rVert_{L^1_x \cap L^\infty_x} \lesssim \lVert \rho_0\rVert_{L^1_x\cap L^{\infty}_x} \quad \mathbb{P} \mbox{-almost surely}.
\end{align}
Let us consider test functions $\phi\in C^{\infty}_c(\R^d)$ such that $0\leq \phi\leq 1$, and $\ \psi\in C^{\infty}_c([0,T+1))$ such that $0\leq \psi\leq 1,\ \psi|_{[0,T+\frac{1}{2}]}\equiv 1$. 
Let us define the quantities $f^{n}_{\phi}, f_{\phi}:[0,T]\rightarrow \R,$ as follows:
\begin{align*}
    &f^n_{\phi}(t)
    :=
    \int_{\R^d}\lvert \rho^{\varepsilon_n}_t(x)\rvert^2\phi(x)dx,
    \quad 
    f_{\phi}(t)
    :=
    \int_{\R^d}\lvert \rho_t(x)\rvert^2\phi(x)dx.
\end{align*}
Notice that $f^n_{\phi}(0)=f_{\phi}(0)$.
By \autoref{lem:beta_rho} there exists a $\PP$-negligible set $N \in \mathcal{F}$, $\mathbb{P}(N)=0$, and for every $n \in \N$ a non-negative random measure $\mathcal{D}^{\varepsilon_n} : \O \to \mathcal{M}_+([0,T+1] \times \R^d)$, with the following property.
For each $\omega\in N^c$ there exists a subset $\tau_{\omega}\subset [0,T+1]$ with full Lebesgue measure such that for each $t\in \tau_{\omega}$ and $n\in \N$
\begin{align*}
    f^n_{\phi}(t)
    &=
    f^n_{\phi}(0)+\int_0^t \langle b_s\cdot\nabla\phi, \lvert \rho^{\varepsilon_n}_s\rvert^2\rangle ds
    +
    \varepsilon_n \sum_{k\in \N}\int_0^t \langle \sigma_k\cdot\nabla\phi, \lvert \rho^{\varepsilon_n}_s\rvert^2\rangle dW^k_s 
    \\ 
    &\quad+
    \varepsilon_n^2\int_0^t  \langle\lvert \rho^{\varepsilon_n}_s\rvert^2,\Delta\phi\rangle ds
    -
    \langle \langle d\mathcal{D}^{\varepsilon_n}, \mathbf{1}_{[0,t]} \phi \rangle \rangle.
\end{align*}
On the other hand, since $\rho$ is a renormalized solution of the transport equation \eqref{eq:transport}, it holds for every $t \in [0,T+1]$:
\begin{align*}
   f_{\phi}(t)= f_{\phi}(0)+\int_0^t \langle b_s\cdot\nabla\phi, \lvert \rho_s\rvert^2\rangle ds. 
\end{align*}

Next we look at the evolution of the quantities $f^n_\phi(t) \psi(t)$ and $f_\phi(t) \psi(t)$, for $t \in [0,T+1]$. Recall that $\psi(0) = 1$ and $\psi(T+1) = 0$ by construction.
Due to previous steps and the uniform bound \eqref{bounds_renormalized}, we have the convergence
\begin{align*}
     f^n_{\phi}(0)
     &+
     \int_0^{T+1}\psi_s \langle b_s\cdot\nabla\phi, \lvert \rho^{\varepsilon_n}_s\rvert^2\rangle ds
     +
     \int_0^{T+1} \partial_s \psi_s f^n_{\phi}(s)ds
     \\  
     &\rightarrow  f_{\phi}(0)
     +
     \int_0^{T+1} \psi_s \langle b_s\cdot\nabla\phi, \lvert \rho_s\rvert^2\rangle ds
     +
     \int_0^{T+1} \partial_s \psi_s f_{\phi}(s)ds.
\end{align*}
Since $\rho^{\varepsilon_n}$ is converging to the renormalized solution $\rho$ we deduce as well
\begin{align} \label{eq:epsilon_n_tozero_together}
\varepsilon_n \sum_{k\in \N}\int_0^{T+1} \psi_s \langle \sigma_k\cdot\nabla\phi, \lvert \rho^{\varepsilon_n}_s\rvert^2\rangle dW^k_s
+
\varepsilon_n^2\int_0^{T+1} \psi_s \langle\lvert \rho^{\varepsilon_n}_s\rvert^2,\Delta\phi\rangle ds
-
\langle \langle d\mathcal{D}^{\varepsilon_n} , \psi \phi  \rangle \rangle   \rightarrow 0. 
\end{align}
Up to passing to further non-relabeled subsequences, thanks to \eqref{bounds_renormalized}, it is easy to show that each of the following terms individually goes to zero $\PP$-almost surely:
\begin{align} \label{eq:each_term_individually}
    \sup_{t\in [0,T+1]}\varepsilon_n &\left\lvert\sum_{k\in \N}\int_0^t \langle \sigma_k\cdot\nabla\phi, \lvert \rho^{\varepsilon_n}_s\rvert^2\rangle dW^k_s\right\rvert
    +
    \sup_{t\in [0,T+1]}\varepsilon_n\left\lvert\sum_{k\in \N}\int_0^t  \psi_s \langle \sigma_k\cdot\nabla\phi, \lvert \rho^{\varepsilon_n}_s\rvert^2\rangle dW^k_s\right\rvert \nonumber
    \\
    &\quad+
    \varepsilon_n^2\int_0^{T+1}  \lvert\langle\lvert \rho^{\varepsilon_n}_s\rvert^2,\Delta\phi\rangle \rvert^2 ds
    +
    \varepsilon_n^2\int_0^{T+1} \psi_s^2 \lvert\langle\lvert \rho^{\varepsilon_n}_s\rvert^2,\Delta\phi\rangle \rvert^2 ds\rightarrow 0,\quad\mathbb{P} \mbox{-almost surely}.
\end{align}
From this and \eqref{eq:epsilon_n_tozero_together} we also deduce that for every $t \in [0,T+1/2]$, it holds $\PP$-almost surely as $n \to \infty$
\begin{align*}
0\leq \langle \langle d\mathcal{D}^{\varepsilon_n}, \mathbf{1}_{[0,t]}\phi \rangle\rangle 
\leq 
\langle \langle d\mathcal{D}^{\varepsilon_n} , \psi \phi  \rangle \rangle   \rightarrow 0.
\end{align*}
Now, for each $t\in \tau_\omega \cap [0,T+\frac{1}{4}]$ we can control
\begin{align*}
    \lvert f^n_{\phi}(t)-f_{\phi}(t)\rvert & \leq \int_0^{T+1} \lvert\langle b_s\cdot\nabla \phi, \lvert \rho^{\varepsilon_n}_s \rvert^2-\lvert \rho_s \rvert^2\rangle\rvert ds+ \sup_{t\in [0,T+1]}\varepsilon_n \left\lvert\sum_{k\in \N}\int_0^t \langle \sigma_k\cdot\nabla\phi, \lvert \rho^{\varepsilon_n}_s\rvert^2\rangle dW^k_s\right\rvert\\ & +\varepsilon_n^2\int_0^{T+1}  \lvert\langle\lvert \rho^{\varepsilon_n}_s\rvert^2,\Delta\phi\rangle \rvert^2 ds
    +
    \langle \langle d\mathcal{D}^{\varepsilon_n}, \mathbf{1}_{[0,T+1/2]} \phi \rangle  \rangle 
    \rightarrow 0.
\end{align*}
Therefore, recalling the definition of $f^n_\phi$ and $f_\phi$, we get 
\begin{align*}
    \sup_{t\in \tau_\omega \cap [0,T+\frac{1}{4}]}
    \left| \int_{\R^d}\left(\lvert \rho^{\varepsilon_n}_t(x)\rvert^2-\lvert \rho_t(x)\rvert^2\right)\phi(x)dx \right| \rightarrow 0.
\end{align*}
Let $O_R$ denote the ball in $\R^d$ of radius $R<\infty$ and center zero.
By approximating $\mathbf{1}_{O_R}$ by a sequence of smooth $\phi$'s, we deduce
\begin{align} \label{eq:convergence_norms_onboundedsets}
\sup_{t\in \tau_\omega \cap [0,T+\frac{1}{4}]}
    \left| \int_{O_R} \lvert \rho^{\varepsilon_n}_t(x)\rvert^2 dx
    -
    \int_{O_R}\lvert \rho_t(x)\rvert^2
    dx \right| \rightarrow 0.
\end{align}

\emph{Step 3: Essentially uniform convergence in $L^2_{loc}$.}
Let $t\in \tau_\omega\cap [0,T+\frac{1}{4}]$. Then we have for every $R < \infty$
\begin{align*}
    \int_{O_R}\lvert\rho^{\varepsilon_n}_t(x)- \rho_t(x)\rvert^2 dx
    &\leq 
    \sup_{t\in \tau_\omega \cap [0,T+\frac{1}{4}]}
    \left| \int_{O_R} \lvert \rho^{\varepsilon_n}_t(x)\rvert^2 dx
    -
    \int_{O_R}\lvert \rho_t(x)\rvert^2
    dx \right| 
    \\
    &\quad+
    2\sup_{t\in \tau_\omega\cap [0,T+\frac{1}{4}]}\langle \one_{O_R}\rho_t,\rho^{\varepsilon_n}_t-\rho_t\rangle.
\end{align*}
Since $\rho_t\in C_{t,*}L^2_x$ (cf. \cite[Theorem II.3]{diperna1989ordinary}) and $\rho^{\varepsilon_n}\rightarrow \rho $ in $\mathcal{E}_*$,
we have 
\begin{align*}
    \limsup_{n\rightarrow+\infty}\sup_{t\in \tau_\omega\cap [0,T+\frac{1}{4}]}\langle \one_{O_R}\rho_t,\rho^{\varepsilon_n}_t-\rho_t\rangle=0.
\end{align*}
Therefore, putting this together with \eqref{eq:convergence_norms_onboundedsets}, we arrive to
\begin{align} \label{eq:L2_loc}
  \limsup_{n\rightarrow +\infty}\sup_{t\in \tau_\omega\cap [0,T+\frac{1}{4}]}\int_{O_R}\lvert\rho^{\varepsilon_n}_t(x)- \rho_t(x)\rvert^2 dx=0 .
\end{align}

\emph{Step 4: Uniform concentration on a large ball.}
Let us fix $\eta>0$ and let us introduce, for every $R'>R$, the functions $\Psi_{R,R'}\in C^{\infty}_c(\R^d)$ and $\Psi_R\in C_b^{\infty}(\R^d)$, being such that $ 0\leq \Psi_{R,R'}\leq \Psi_{R}\leq 1$, and
\begin{align*}
    \Psi_{R,R'}(x)=\begin{cases}
        0\quad \text{if }\lvert x\rvert <R/2,\\
        1\quad \text{if }R< \lvert x\rvert <R',\\
        0\quad \text{if }2R'<\lvert x\rvert,
    \end{cases}
    \qquad
        \Psi_{R}(x)=\begin{cases}
        \Psi_{R,R'}\quad \text{if }\lvert x\rvert <R,\\
        1 \qquad \,\,\,\text{ if }R\leq \lvert x\rvert,
    \end{cases}
\end{align*}
and such that for some constant $C$ independent of $x,R,R'$ we can control
\begin{align*}
\lvert \nabla\Psi_{R,R'}(x)\rvert+\lvert \nabla\Psi_{R}(x)\rvert \leq \frac{C}{R},
\qquad
\lvert \nabla^2\Psi_{R,R'}(x)\rvert+\lvert \nabla^2\Psi_{R}(x)\rvert \leq \frac{C}{R^2}.
\end{align*}
Using $\Psi_{R,R'}$ as test function in \autoref{lem:beta_rho}, and recalling that $\mathcal{D}^{\varepsilon_n} \geq 0$, we obtain for each $t\in \tau_\omega\cap [0,T+\frac{1}{4}]$
\begin{align} \label{estimates_out_of_ball}
    \int_{\R^d} \lvert \rho^{\varepsilon_n}_t(x)\rvert^2\Psi_{R,R'}(x)dx
    &\leq 
    \int_{\R^d\setminus O_{R/2}} \lvert \rho_0(x)\rvert^2dx+\frac{C}{R}\int_0^{T+1} \int_{\R^d}\lvert b_s(x)\rvert \lvert \rho^{\varepsilon_n}_s(x)\rvert^2 dx ds
    \\ \notag &+ 
    \frac{C\varepsilon_n^2}{R^2}\int_0^{T+1} \int_{\R^d}\lvert \rho^{\varepsilon_n}_s(x)\rvert^2 dx ds
    \\ \notag &+
    \varepsilon_n\sup_{t\in [0,T+1]} \left|\sum_{k\in\N}\int_0^t \langle\sigma_k\cdot(\nabla \Psi_{R,R'}-\nabla \Psi_{R}),\lvert \rho^{\varepsilon_n}_s \rvert^2\rangle dW^{k}_s \right|
    \\ \notag &+
    \varepsilon_n\sup_{t\in [0,T+1]} \left|\sum_{k\in\N}\int_0^t \langle\sigma_k\cdot\nabla \Psi_{R},\lvert \rho^{\varepsilon_n}_s \rvert^2\rangle dW^{k}_s \right|.
\end{align}
By \eqref{bounds_renormalized} we have
\begin{align*}
    \mathbb{E}&\left[\sup_{t\in [0,T+1]} \left| \sum_{k\in\N}\int_0^t \langle\sigma_k\cdot(\nabla \Psi_{R,R'}-\nabla \Psi_{R}),\lvert \rho^{\varepsilon_n}_s \rvert^2\rangle dW^{k}_s\right|^2\right]
    \\
    &\quad\lesssim 
    \int_0^{T+1}\mathbb{E}\left[\lVert (\nabla \Psi_{R,R'}-\nabla \Psi_{R})\lvert \rho^{\varepsilon_n}_t \rvert^2\rVert^2_{L^1_x}\right]dt
    \\ 
    &\quad\lesssim 
    \lVert \rho_0\rVert_{L^1_x\cap L^{\infty}_x}\lVert \nabla \Psi_{R,R'}-\nabla \Psi_{R}\rVert_{L^1_x}^2
     \rightarrow 0,
     \qquad\text{as }R'\rightarrow+\infty.
\end{align*}
Hence there exists a sequence of radii $R_m'\rightarrow \infty$ such that as $m \to \infty$
\begin{align*}
    \varepsilon_n\sup_{t\in [0,T+1]} \left|\sum_{k\in\N}\int_0^t \langle\sigma_k\cdot(\nabla \Psi_{R,R'_m}-\nabla \Psi_{R}),\lvert \rho^{\varepsilon_n}_s \rvert^2\rangle dW^{k}_s \right|\rightarrow 0,\quad\mathbb{P} \mbox{-almost surely}.
\end{align*}
Therefore, taking the limit $m \to \infty$ in \eqref{estimates_out_of_ball} above for such sequence of $R'_m$ and recalling \eqref{bounds_renormalized}, we obtain for every given $n \in \N$  
\begin{align*}
\int_{\R^d\setminus O_R} \lvert \rho^{\varepsilon_n}_t(x)\rvert^2dx&\leq \int_{\R^d\setminus O_{R/2}} \lvert \rho_0(x)\rvert^2dx+\frac{C}{R}\lVert b\rVert_{L^1_tW^{1,q}_x}\lVert \rho_0\rVert_{L^1_x\cap L^{\infty}_x}^2\notag\\ & + \frac{C \varepsilon_n^2}{R^2}\lVert \rho_0\rVert_{L^1_x\cap L^{\infty}_x}^2\notag
\\ 
&+
\varepsilon_n\sup_{t\in [0,T+1]}\left|\sum_{k\in\N}\int_0^t \langle\sigma_k\cdot\nabla \Psi_{R},\lvert \rho^{\varepsilon_n}_s\rvert^2\rangle dW^{k}_s\right|.
\end{align*}
Let us now choose $R=R(\eta)$ large enough such that \begin{align*}
  \int_{\R^d\setminus O_{R/2}} \lvert \rho_0(x)\rvert^2dx+\frac{C}{R}\lVert b\rVert_{L^1_tW^{1,q}_x}\lVert \rho_0\rVert_{L^1_x\cap L^{\infty}_x}^2 
  \leq \frac{\eta}{2}.  
\end{align*}
For such choice of $R$ let us observe that 
\begin{align*}
    \varepsilon^2_n\mathbb{E}\left[\sup_{t\in [0,T+1]}
    \left|\sum_{k\in\N}\int_0^t \langle\sigma_k\cdot\nabla \Psi_{R},\lvert \rho^{\varepsilon_n}_s \rvert^2\rangle dW^{k}_s \right|^2\right]\rightarrow 0\quad\text{as }n\rightarrow+\infty.
\end{align*}
Therefore, there exist a non-relabeled subsequence $\varepsilon_n=\varepsilon_n(\eta)$ such that 
\begin{align*}
    \varepsilon_n\sup_{t\in [0,T+1]}\left|\sum_{k\in\N}\int_0^t \langle\sigma_k\cdot\nabla \Psi_{R},\lvert \rho^{\varepsilon_n}_s \rvert^2\rangle dW^{k} _s \right|\rightarrow 0, \quad \mathbb{P} \mbox{-almost surely}.
\end{align*}
In conclusion, up to choosing $\varepsilon_n$ in such a subsequence, 
\begin{align} \label{eq:uniform_control_ball_eps}
    \limsup_{n\rightarrow +\infty}\sup_{t\in \tau_\omega\cap [0,T+\frac{1}{4}]}\int_{\R^d\setminus O_R} \lvert \rho^{\varepsilon_n}_t(x)\rvert^2dx\leq \eta.
\end{align}
The same argument implies that also
\begin{align} \label{eq:uniform_control_ball_ren}
    \sup_{t\in \tau_\omega\cap [0,T+\frac{1}{4}]}\int_{\R^d\setminus O_R} \lvert \rho_t(x)\rvert^2dx\leq \eta.
\end{align}

\emph{Step 5: Uniform convergence in $L^2_x.$}
For each $t\in \tau_\omega\cap [0,T+\frac{1}{4}]$ and $R \in (1,\infty) $ we obviously have
\begin{align*} 
   \norm{\rho^{\varepsilon_n}_t-\rho_t}_{L^2_x}^2& \leq \int_{O_R}\lvert\rho^{\varepsilon_n}_t(x)- \rho_t(x)\rvert^2 dx+2\int_{\R^d\setminus O_R}\lvert \rho^{\varepsilon_n}_t(x)\rvert^2 dx+  2\int_{\R^d\setminus O_R}\lvert \rho_t(x)\rvert^2 dx.
\end{align*}
By \eqref{eq:uniform_control_ball_eps} and \eqref{eq:uniform_control_ball_ren} from previous step, for every $\eta>0$ there exist $R < \infty$ and a non-relabeled subsequence $\varepsilon_n$ such that
\begin{align*}
    \limsup_{n\rightarrow +\infty}\sup_{t\in \tau_\omega\cap [0,T+\frac{1}{4}]}\norm{\rho^{\varepsilon_n}_t-\rho_t}_{L^2_x}^2& \leq \limsup_{n\rightarrow +\infty}\sup_{t\in \tau_\omega\cap [0,T+\frac{1}{4}]}\int_{O_R}\lvert\rho^{\varepsilon_n}_t(x)- \rho_t(x)\rvert^2 dx+4\eta  
    \leq 
    4\eta,
\end{align*}
where the last passage is justified by \eqref{eq:L2_loc} in Step 3 of the proof. 
Due to the arbitrariness of $\eta$, by a diagonal argument we can find a subsequence such that
\begin{align*}
\lim_{n\rightarrow+\infty}\norm{\rho^{\varepsilon_n}-\rho}_{L^{\infty}_tL^2_x}^2& =0, \quad\mathbb{P} \mbox{-almost surely}.
\end{align*}
Since $L^{\infty}_t L^2_x$ is a metric space, the latter implies the convergence in probability of the full sequence since ${\rho}$ is deterministic. The proof is complete.

\begin{rmk}
Since $\rho^{\varepsilon_n}, \rho\in \mathcal{E}$, the map $t \mapsto \norm{\rho^{\varepsilon_n}_t-\rho_t}_{L^2_x}^2$ is lower-semicontinuous.
In particular, the line above also implies 
\begin{align*}
\lim_{n\rightarrow+\infty} \sup_{t \in [0,T] }\norm{\rho^{\varepsilon_n}_t-\rho_t}_{L^2_x}^2& =0, \quad\mathbb{P} \mbox{-almost surely}.
\end{align*}
\end{rmk}

\section{Measurability and Stability Results}\label{Sec_Meas_Stab}
In this section we collect some results we will employ drastically in the proof of \autoref{thm:LDP} and \autoref{thm:strong_LDP}.
\subsection{Some measurability results}\label{subsec:meas_results}
Recall that, given a topological space $(X,\tau)$ and a subset $Y \subset X$, the subspace topology on $Y \subset X$ is defined as the coarsest topology on $Y$ such that the inclusion map $Y \subset X$ is continuous. Open sets in the subspace topology $Y \subset X$ are of the form $A \cap Y$ for some $A \in \tau$. Note that in our case, since $\mathcal{E}$ and $\mathscr{E}$ coincide as sets, $\mathscr{E}$ endowed with the subspace topology $\mathscr{E} \subset \mathcal{E}$ coincides with $(\mathcal{E}, \tau_{d_{\mathcal{E}}})$ as topological space. Similar reasonings apply also to $\mathscr{E}\times \mathscr{E}$.

\begin{lem} \label{lem:|rho-f|_closed}
For every $\delta \geq 0$ the set
\begin{align*}
   D_\delta := \{ (\rho,f) \in \mathscr{E} \times \mathscr{E} :  \| \rho_t - f_t\|_{\mathscr{E}} \leq \delta \}
\end{align*}
is closed with respect to the subspace topology on $\mathscr{E} \times \mathscr{E} \subset \mathcal{E} \times \mathcal{E}$.
\end{lem}
\begin{proof}
Let $\{t_n\}_{n \in \N}$ be a countable dense in $[0,T]$ containing the right endpoint (without loss of generality $t_1 = T$). We notice that
\begin{align*}
D_{\delta}&=\{ (\rho,f) \in \mathscr{E} \times \mathscr{E} :\sup_{t \in [0,T]}\| \rho_t - f_t\|_{L^2_x}+\sup_{t \in [0,T]}\| \rho_t - f_t\|_{L^p_x}  \leq \delta \}\\
&=
\bigcap_{n,m \in \N}
\{ (\rho,f) \in \mathscr{E} \times \mathscr{E} :\| \rho_{t_n} - f_{t_n}\|_{L^2_x}+\| \rho_{t_m} - f_{t_m}\|_{L^p_x} \leq \delta \}
,
\end{align*}
since elements in $\mathscr{E}$ are weakly continuous in $L^2_x\cap L^p_x$. Since both the $L^2_x$ and the $L^p_x$ norm are lower-semicontinuous with respect to the convergence in $\mathcal{B}$, each set in the intersection above is closed with respect to the subspace topology on $\mathscr{E} \times \mathscr{E} \subset \mathcal{E} \times \mathcal{E}$, implying $D_\delta$ closed as well since it is a countable intersection of closed sets.
\end{proof}

\begin{lem} \label{lem:dist(rho,F)_meas}
Let $F$ be a separable closed set in $\mathscr{E}$. Then for every $\delta \geq 0$ the set
\begin{align*}
F_\delta := \{ \rho \in \mathscr{E} : \mbox{d}_{\mathscr{E}}(\rho,F) \leq \delta \}
\end{align*}
is measurable with respect to the Borel $\sigma$-field on $\mathscr{E}$ generated by the subspace topology on $\mathscr{E} \subset \mathcal{E}$. Moreover it holds
\begin{align} \label{eq:F_delta_union}
F_\delta = \bigcap_{N\in \N}F^N_{\delta}
\end{align}
for a sequence of subsets $F^{N}_\delta$ satisfying $F^N_{\delta}=\bigcap_{n=1}^N F^n_{\delta}$ for each $N\in \N$ and $\ F^N_{\delta}=\bigcup_{M\in \N} F^{M,N}_{\delta}$ for some $F^{M,N}_{\delta}$ that are closed with respect to the subspace topology on $\mathscr{E} \subset \mathcal{E}$
and $F^{M,N}=\bigcup_{m=1}^M F^{m,N} $ for each $M,N\in \N.$
\end{lem}
\begin{proof}
If $F$ is empty then $F_\delta = \emptyset$ and the thesis holds, therefore we can assume $F$ non-empty hereafter. Let $\{f^m\}_{m \in \N} \subset F$ be a dense subset in $F$ with respect to the $\mathscr{E}$ topology. We have for every $\delta\geq 0$  
\begin{align} \label{eq:F_delta}
    F_\delta = \{ \rho \in \mathscr{E} : \inf_{m \in \N}  d_{\mathscr{E}}(\rho,f^m) \leq \delta \}.
\end{align}
Indeed, as $\{f^m\}_{m \in \N}$ is dense in $F$, given any sequence $f_k \in F$ such that $ d_{\mathscr{E}}(\rho, f_k) \leq \mbox{d}_{\mathscr{E}}(\rho,F) + 1/k$, $k \in \N$, there exists $m = m_k \in \N$ with $d_{\mathscr{E}}( f^m, f_k) \leq 1/k$; we deduce $d_{\mathscr{E}}( \rho, f^m) \leq d_{\mathscr{E}}(\rho, f_k)+ d_{\mathscr{E}}(f_k, f^m ) \leq \mbox{dist}_{\mathscr{E}}(\rho,F) + 2/k$, implying $\inf_{m \in \N} d_{\mathscr{E}}(\rho,f^m) \leq d_{\mathscr{E}}(\rho,F)$ since $k$ is arbitrary. On the other hand, the converse inequality holds true trivially since $f^m \in F$ for every $m$, and \eqref{eq:F_delta} follows.
For fixed $M,N \in \N$ let us denote
\begin{align*}
    F^{M,N}_\delta 
    := 
    \bigcup_{m =1}^M \{ \rho \in \mathscr{E} : d_{\mathscr{E}}(\rho, f^m ) \leq \delta+\frac{1}{N} \}.
\end{align*}
By definition the sets $F^{M,N}_\delta$ are nested: $F^{M,N}_\delta \subset F^{M+1,N}_\delta$.
Moreover, each $F^{M,N}_\delta$ is closed with respect to the subspace topology on $\mathscr{E} \subset \mathcal{E}$, by the same arguments in the proof of \autoref{lem:|rho-f|_closed} and $F_{\delta}=\bigcap_{N\in \N}\bigcup_{M\in \N}F^{M,N}_{\delta}$.
By letting $F^N_{\delta}:=\bigcup_{M\in \N}F^{M,N}_{\delta}$ and observing that $F^N_{\delta}=\bigcap_{n=1}^N F^{n}_{\delta} $, we obtain the desired \eqref{eq:F_delta_union}. 
\end{proof}

The following lemma states that the probability of events that a random variables belongs to $F_{\delta}$ only depends on its law on $\mathcal{E}$.
\begin{lem} \label{lem:same_law}
Let $(\Omega,\mathcal{F},\PP)$ and $(\tilde{\Omega},\tilde{\mathcal{F}},\tilde{\PP})$ be two probability spaces and let $\rho : \Omega \to \mathcal{E}$ and $\tilde \rho : \tilde{ \Omega} \to \mathcal{E}$ two random variables with the same law.  
Then for every closed separable $F \subset \mathscr{E}$ and $\delta \geq 0$ we have $\{d_{\mathscr{E}}(\rho,F) \leq \delta \} \in \mathcal{F}$, $\{d_{\mathscr{E}}(\tilde\rho,F) \leq \delta \} \in \tilde{\mathcal{F}}$, and 
\begin{align*}
\PP \{d_{\mathscr{E}}(\rho,F) \leq \delta \}
=
\tilde\PP \{d_{\mathscr{E}}(\tilde\rho,F) \leq \delta \}.
\end{align*}
\end{lem}
\begin{proof}
Noticing that $\mathcal{E}$ and $\mathscr{E}$ coincide as sets, by previous lemma we can rewrite 
\begin{align*}
\{d_{\mathscr{E}}(\rho,F) \leq \delta \}
=
\{\rho \in F_\delta\}
= \bigcap_{N\in \N}\bigcup_{M \in \N} \{\rho \in F_\delta^{M,N}\},
\end{align*}
where each $F^{M,N}_\delta$ is closed with respect to the subspace topology on $\mathscr{E} \subset \mathcal{E}$.
In particular, $F^{M,N}_\delta$ is closed in $\mathcal{E}$ for every $M,N$.
This implies $\{\rho \in F_\delta\} \in \mathcal{F}$ since $\rho : \Omega \to \mathcal{E}$ is a random variable, in particular $\rho$ is measurable with respect to the Borel $\sigma$-field on $\mathcal{E}$. 
The same holds for $\tilde{\rho}$.

By definition we have
\begin{align*}
    \mathbb{P}(\rho \in F_{\delta})=\lim_{N\rightarrow +\infty} \mathbb{P}(\rho \in F^N_{\delta}),\quad \tilde{\mathbb{P}}(\tilde{\rho} \in F_{\delta})=\lim_{N\rightarrow +\infty} \tilde{\mathbb{P}}(\tilde{\rho} \in F^N_{\delta}).
\end{align*}
Therefore it is enough to show the validity of
\begin{align*}
    \mathbb{P}(\rho \in F^N_{\delta})=\tilde{\mathbb{P}}(\tilde{\rho} \in F^N_{\delta}) \quad \forall N\in \N.
\end{align*}
By monotone convergence we have
\begin{align*}
\PP \{\rho\in F^N_{\delta}\}
&=
\lim_{M \to \infty} \PP \{\rho \in F^{M,N}_\delta \}  =\lim_{M \to \infty} \tilde{\PP} \{\tilde{\rho} \in F^{M,N}_\delta \} =\tilde{\PP}\{\tilde{\rho}\in F^N_{\delta}\},
\end{align*}
where the second identity is valid because $Law(\rho) = Law(\tilde{\rho})$, and therefore they coincide on all closed sets in $\mathcal{E}$. 
\end{proof}

\begin{cor}\label{corollary_integrals}
For every $\rho$, $\tilde{\rho}$, and $F$ as above, the functions 
\begin{align*}
d_{\mathscr{E}}(\rho,F) : \Omega \to \R_+,
\quad
\mbox{and}
\quad
d_{\mathscr{E}}(\tilde{\rho},F) : \Omega \to \R_+,
\end{align*}
are Borel measurable random variables, and they have the same law.
In particular, for every continuous and bounded $f:\R_+ \to \R$
\begin{align*}
\mathbb{E} \left[ f(d_{\mathscr{E}}(\rho,F) )\right]
=
\tilde{\mathbb{E}} \left[ f(d_{\mathscr{E}}(\tilde \rho,F)) \right].
\end{align*}
\end{cor}
\begin{proof}
By previous \autoref{lem:same_law}, the maps defined above are Borel measurable as the closed intervals of the form $\{[0,\delta]\}_{\delta \geq 0}$ generate the Borel $\sigma$-field on $\R_+$. In particular, they are real-valued random variables. Again by \autoref{lem:same_law}, they have the same law as their cumulative distribution functions coincide. 
\end{proof}

\subsection{Stability Results}\label{subsec_stab_res}
This subsection is devoted to the proof of \autoref{stability_theorem} below, giving stability (in law) of solutions to the stochastic transport equation \eqref{eq:rho_general} with respect to the drift $g$.

For $N \in \R_+$ let us define 
\begin{align*}
    S^N:=\{g\in L^2_t \mathcal{H}_0 \,:\, \|g\|_{L^2_t\mathcal{H}_0}^2 \leq N\},
\end{align*}
which is a Polish space when endowed with the weak $L^2_t \mathcal{H}_0$ topology. 
Next, let $(\Omega^{\varepsilon},\mathcal{F}^{\varepsilon},\{\mathcal{F}^{\varepsilon}_t\}_{t \geq 0},\mathbb{P}^{\varepsilon})$ be a family, indexed by $\varepsilon\geq 0$, of filtered probability space satisfying the usual assumptions. We denote
\begin{align*}
 \mathcal{P}_{prog}^N :=
 \{g: \Omega^\varepsilon \to S^N, \,\, g \mbox{ is } \{\mathcal{F}^\varepsilon_t\}_{t \geq 0} \mbox{ is progressively measurable} \},   
\end{align*}
without explicitly specifying the dependence of $\varepsilon$. 

Let $g^{\varepsilon}, g\in \mathcal{P}_{prog}^N $ for some $N<+\infty$.
In the following we denote $\rho^{\varepsilon,\rho_0^\varepsilon,g^\varepsilon}$ the unique probabilistically strong, analytically weak solution of the stochastic transport equation \eqref{eq:rho_general} with initial condition $\rho_0^\varepsilon$ and advecting velocity field $b+g^\varepsilon$, given by \autoref{prop_well_posed}.
Similarly, we denote $\rho^{\rho_0,g}$ the unique renormalized solution of the transport equation \eqref{eq:transport} with initial condition $\rho_0$ and advecting velocity field $b+g$, given by \cite[Theorem II.3]{diperna1989ordinary}. Notice that $\rho^{\rho_0,g} : \Omega^0 \to C_tL^2_x$ is a  random variable as $g$ is a random variable taking values in $S^N$ and the map $g \mapsto \rho^{\rho_0,g}$ is continuous by \cite[Theorem II.7]{diperna1989ordinary}.
We have the following:

\begin{prop}\label{stability_theorem}
Under the same assumptions of the first part of \autoref{thm:convergence}, suppose that $g^{\varepsilon}, g\in \mathcal{P}_{prog}^N $ for some $N<+\infty$ and $g^{\varepsilon}$ converges in law to $g$ as random variables in $S^N$. 
Then $\rho^{\varepsilon,\rho_0^\varepsilon,g^\varepsilon}$ converges in law to $\rho^{\rho_0,g}$, as random variables in $\mathcal{E}$. 

Moreover, if $\dvg \, b = 0$ and $F$ is a closed separable set in $\mathscr{E}$, it holds
\begin{align*}
\lim_{\varepsilon\rightarrow 0}
    \mathbb{E}^{\varepsilon}[ d_{\mathscr{E}}(\rho^{{\varepsilon},\rho_0^\varepsilon,g^{\varepsilon}},F) \wedge 1]=
    {\mathbb{E}}^0[ d_{\mathscr{E}}(\rho^{\rho_0,g},F) \wedge 1].    
\end{align*}
\end{prop}
\begin{proof}
As previously done for \autoref{thm:convergence} and \autoref{lem_Backward_convergence}, we split the argument in several steps. Moreover, we will consider extensions of $g^{\varepsilon}, g, b$ identically null on $(T,T+1]$. We call $g^{\varepsilon,*},g^*,b^*$ these extensions. Obviously also $g^{\varepsilon,*}$ converges in law to $g^{*}$. Similarly we will denote by $\mathcal{E}_*:= C_{t,*}\mathcal{B} \cap C_{t,*}\tilde{H}^-$, $ L^2_{t,*}\mathcal{H}_0 := L^2_{t,*}\mathcal{H}_0  
$ endowed with the weak topology, and analogously for other Bochner spaces.

\emph{Step 1: Convergence in law.}
As in the proof of \autoref{thm:convergence}, in order to prove the first part of the statement it is enough to show that for each sequence $\varepsilon_n\rightarrow 0$ there exists a (non-relabeled) subsequence such that $\rho^{\varepsilon_n,\rho_0^{\varepsilon_n},g^{\varepsilon_n}}$ converges in law to $\rho^{\rho_0,g}$ in the topology of $\mathcal{E}$. This implies the validity of the claim due to the uniqueness of the limit.  

Let us fix a small parameter $0<\delta \ll 1$ (not the same as in \autoref{prop_well_posed}). Due to the stability of renormalized solutions given by \cite[Theorem II.7, Theorem III.1]{diperna1989ordinary}  we can find $0<\theta \ll 1$ small enough and $\rho_0^1\in L^1_x\cap L^{\infty}_x,\ \rho_0^2\in L^2_x\cap L^p_x$ such that 
\begin{align*}
    \rho_0=\rho_0^1+\rho_0^2,\quad \norm{\rho_0^1}_{L^2_x\cap L^p_x}\leq 1,\quad \norm{\rho_0^2}_{L^2_x\cap L^p_x}\leq \theta,
\end{align*}
and moreover, denoting by ${\rho}^{1,g}$ the unique renormalized solution of the transport equation on $[0,T+1]$ with initial condition ${\rho}_0^1$ and advection $b^*+g^*$:
\begin{align} \label{eq:distance_renormali}
d_{\mathscr{E}_*}(\rho^g,\rho^{1,g})
=
 d_{\mathscr{E}}({\rho}^g,{\rho}^{1,g})\leq \exp(\|\dvg \, b \|_{L^1_tL^\infty_x})\theta \leq \frac{\delta}{8}.
\end{align}
Secondly, let us choose a sequence $\varepsilon_n$ such that $\varepsilon_n<\overline{\varepsilon}$ for some $\overline{\varepsilon}$ sufficiently small such that $\lVert \rho_0-\rho_0^{\varepsilon_n}\rVert_{L^2_x \cap L^p_x} \leq \theta$ and let us introduce $\rho_0^{\varepsilon_n,1} := \rho_0^1\in L^1_x\cap L^{\infty}_x$ and $\rho_0^{\varepsilon_n,2}:=\rho_0^{\varepsilon_n}-\rho_0+\rho_0^2\in L^2_x\cap L^p_x$. 
Let us denote by $\rho^{\varepsilon_n,g^{\varepsilon_n,*},1}$ (resp. $\rho^{\varepsilon_n,g^{\varepsilon_n,*},2}$) the unique weak solution of the stochastic transport equation on $[0,T+1]$ with initial condition $\rho_{0}^{\varepsilon_n,1}$ (resp. $\rho_0^{\varepsilon_n,2}$) and drift $b^*+g^{\varepsilon_n,*}$,  given by \autoref{prop_well_posed}. 
Bounds analogous to 
\eqref{estimate_stochastc_splitting}, \eqref{estimates_original_spaces}, \eqref{bounds_original_0}, and \eqref{bounds_original_1} hold true $\mathbb{P}^{\varepsilon_n}$-almost surely for $\rho^{\varepsilon_n,g^{\varepsilon_n,*},1}$ and $\rho^{\varepsilon_n,g^{\varepsilon_n,*},2}$.
Furthermore, by linearity and uniqueness of \eqref{eq:rho_general}, we have almost surely
\begin{align} \label{eq:decomposition_ic_aux}
\rho^{\varepsilon_n,\rho_0^{\varepsilon_n},g^{\varepsilon_n}}=\rho^{\varepsilon_n,g^{\varepsilon_n,*},1}|_{[0,T]}+\rho^{\varepsilon_n,g^{\varepsilon_n,*},2}|_{[0,T]}.
\end{align}

Since $g^{\varepsilon_n} \to g$ in law by assumption, arguing as in Step 2 of \autoref{weak_convergence_eps} the family of laws of the processes
\begin{align*}
X^n := (\rho^{\varepsilon_n,\rho_0^{\varepsilon_n},g^{\varepsilon_n}},\rho^{\varepsilon_n,g^{\varepsilon_n,*},1},\rho^{\varepsilon_n,g^{\varepsilon_n,*},2}, g^{\varepsilon_n},g^{\varepsilon_n,*},\{W^{\varepsilon_n,k}\}_{k\in \N})
\end{align*}
is tight in $\mathbb{X} := \mathcal{E}\times \mathcal{E}_* \times \mathcal{E}_* \times L^2_t\mathcal{H}_0\times L^2_{t,*}\mathcal{H}_0\times C_t\R^{\N}$.
By Jakubowski-Skorokhod Representation Theorem, up to choosing a (non-relabeled) subsequence $\varepsilon_n=\varepsilon_n(\delta)$ possibly depending on the parameter $\delta>0$, we can find an auxiliary filtered probability space $(\Omega^{\delta},\mathcal{F}^{\delta},\{\mathcal{F}_t^{\delta}\}_{t \geq 0},\mathbb{P}^{\delta})$ and random variables $\tilde{X}^n$ and $\tilde{X}$ such that $X^n$ and $\tilde{X}^n$ have the same law on $\mathbb{X}$ and $\tilde{X}^n \to \tilde{X}$ $\PP^\delta$-almost surely in $\mathbb{X}$.
We point out that, a priori, the space $(\Omega^{\delta},\mathcal{F}^{\delta},\mathcal{F}_t^{\delta},\mathbb{P}^{\delta})$ depends on the random variables $\rho^{\varepsilon_n,g^{\varepsilon_n,*},1}$ and $ \rho^{\varepsilon_n,g^{\varepsilon_n,*},2}$, defined in the  original probability spaces, and in particular by their $\delta$-dependent initial conditions $\rho_0^{\varepsilon_n,1}$ and $\rho_0^{\varepsilon_n,2}$.

Let us rewrite the processes $\tilde{X}^n$, $\tilde{X}$ in coordinates with respect to $\mathbb{X}$ as follows:
\begin{align*}
\tilde{X}^n 
&=: 
(\tilde{\rho}^{\varepsilon_n},\tilde{\rho}^{\varepsilon_n,1},\tilde{\rho}^{\varepsilon_n,2}, \tilde{g}^{\varepsilon_n},\tilde{g}^{\varepsilon_n,*},\{\tilde{W}^{\varepsilon_n,k}\}_{k\in \N}),
\\
\tilde{X} 
&=:
(\tilde{\rho},\tilde{\rho}^{1},\tilde{\rho}^{2}, \tilde{g},\tilde{g}^*,\{\tilde{W}^{k}\}_{k\in \N}),
\end{align*}
with $\PP^\delta$-almost sure convergence in $\mathbb{X}$ meaning convergence of each component in the respective space. In addition, by the definition of $g^{\varepsilon_n,*}$ and \eqref{eq:decomposition_ic_aux} we have $\PP^\delta$-almost surely 
\begin{align*}
\tilde{g}^{\varepsilon_n,*}|_{[0,T]}=\tilde{g}^{\varepsilon_n},
\qquad
\tilde{\rho}^{\varepsilon_n}
=
\tilde{\rho}^{\varepsilon_n,1}|_{[0,T]}
+
\tilde{\rho}^{\varepsilon_n,2}|_{[0,T]}. 
\end{align*}

By the same arguments presented in Step 1 in \autoref{weak_convergence_eps}, $\tilde{\rho}^{\varepsilon_n}$ is the unique probabilistically strong, analytically weak solution of the stochastic transport equation on the time interval $[0,T]$, with initial condition $\rho_0^{\varepsilon_n}$ and advecting velocity field $b+\tilde{g}^{\varepsilon_n}$ and noise $\tilde{W}^{\varepsilon_n}:=\sqrt{\mathcal{Q}}\tilde{\mathcal{W}}^{\varepsilon_n}$, where $\tilde{\mathcal{W}}^{\varepsilon_n} = \{ \tilde{W}^{\varepsilon_n,k} \}_{k \in \N}$ is a cylindrical Wiener process on $\mathbb{H}$.
Similarly, $\tilde{\rho}^{\varepsilon_n,1}$ (resp. $\tilde{\rho}^{\varepsilon_n,2}$) is the unique solution of the stochastic transport equation on the time interval $[0,T+1]$, with initial condition $\rho_0^{\varepsilon_n,1}$ (resp. $ \rho_0^{\varepsilon_n,2}$) and advecting velocity field $ b^*+\tilde{g}^{\varepsilon_n,*}$ and noise $\tilde{W}^{\varepsilon_n}_t$.

Moreover, following Step 2 of \autoref{weak_convergence_eps} we can show that
$\tilde{\rho}$ is a weak solution of the transport equation on the time interval $[0,T]$, with initial condition $\rho_0$ and advecting velocity field $b+\tilde{g}$.
The only difference with respect to the aforementioned argument consists in the identification of the limit 
\begin{align*}
 \lim_{n \to \infty}   \int_0^T \langle \tilde{\rho}^{\varepsilon_n}_s,\tilde{g}^{\varepsilon_n}_s\cdot\nabla\phi\rangle ds
 =
 \int_0^T \langle \tilde{\rho}_s,\tilde{g}_s\cdot\nabla\phi\rangle ds,
\end{align*}
which can be checked as follows. Rewrite
\begin{align*}
    \left| \int_0^T \langle \tilde{\rho}^{\varepsilon_n}_s,\tilde{g}^{\varepsilon_n}_s\cdot\nabla\phi\rangle ds
    -
    \int_0^T \langle \tilde{\rho}_s,\tilde{g}_s\cdot\nabla\phi\rangle ds \right| 
    &\leq 
     \left|\int_0^T \langle \tilde{\rho}_s,(\tilde{g}_s-\tilde{g}^{\varepsilon_n}_s)\cdot\nabla\phi\rangle ds \right| 
     \\
     &+
      \left| \int_0^T \langle \tilde{\rho}^{\varepsilon_n}_s-\tilde{\rho}_s,\tilde{g}^{\varepsilon_n}_s\cdot\nabla\phi\rangle ds \right| ,
\end{align*}
and notice that both terms at the right-hand-side above approach $0$ almost surely with respect to the measure $\mathbb{P}^{\delta}$, the first one due to the weak convergence of $\tilde{g}^{\varepsilon_n} \rightharpoonup \tilde{g}$ in $L^2_t\mathcal{H}_0$, the second one due to the strong convergence of $\tilde{\rho}^{\varepsilon_n}$ to $\tilde{\rho}$ in $C_t\tilde{H}^{-}_x$ and the uniform boundedness of 
$\frac{{\nabla}\phi}{w} \cdot \tilde{g}^{\varepsilon_n}$ in $L^2_tH^{\eta}_x$ for each $\eta \leq d/2+\alpha$ (here $w$ is the weight appearing in the definition of the weighted Sobolev space $\tilde{H}^-_x$).

Similarly, $\tilde{\rho}^{1}$ (resp. $\tilde{\rho}^{2}$) is a weak solution of the transport equation on the time interval $[0,T+1]$, with initial condition $\rho_0^1$ (resp. $\rho_0^2$) and advecting velocity field $b^*+\tilde{g}^*$.

Due to the regularity of $\tilde{\rho}^1, b^*, \tilde{g}^*$, by DiPerna-Lions theory $\tilde{\rho}^{1}$ is $\PP^\delta$-almost surely the unique renormalized solution of the transport equation with initial condition $\rho_0^1$ and (random) drift $b^*+\tilde{g}^*$. 

Take any Lipschitz continuous $h\in C_b(\mathcal{E})$ of Lipschitz constant less or equal than $L \in (0,\infty)$. Then
\begin{align}\label{estimate_convergence_law}
\left|\mathbb{E}^{\varepsilon_n}
\left[h(\rho^{\varepsilon_n,\rho_0^{\varepsilon_n},g^{\varepsilon_n}})\right]-\mathbb{E}^{0}\left[h(\rho^{\rho_0,g})\right] \right|
&=
\left\lvert\mathbb{E}^{\delta}
\left[h(\tilde{\rho}^{\varepsilon_n})-h(\rho^{\rho_0,\tilde{g}})\right]\right\rvert 
\\
& \leq  
\left\lvert\mathbb{E}^{\delta}\left[h(\tilde{\rho}^{\varepsilon_n})-h(\tilde{\rho}^{\varepsilon_n,1}|_{[0,T]})\right]\right\rvert\notag
\\
& \quad+
\left\lvert\mathbb{E}^{\delta}\left[h(\tilde{\rho}^{\varepsilon_n,1}|_{[0,T]})-h(\tilde{\rho}^{1}|_{[0,T]})\right]\right\rvert\notag
\\ 
& \quad +
\left\lvert\mathbb{E}^{\delta}\left[h(\tilde{\rho}^{1}|_{[0,T]})-h(\rho^{\rho_0,\tilde{g}})\right]\right\rvert\notag
\\ 
& \leq  \notag
L\delta
+
\left\lvert\mathbb{E}^{\delta}\left[h(\tilde{\rho}^{\varepsilon_n,1}|_{[0,T]})-h(\tilde{\rho}^{1}|_{[0,T]})\right]\right\rvert,
\end{align}
We can control two out of the three summand above by a multiple of $\delta$, because we have chosen the initial conditions $\rho_0^{\varepsilon_n,2}$ and $\rho_0^2$ properly small, cf. equations \eqref{estimate_stochastc_splitting} and \eqref{eq:distance_renormali}. Here we have used that $(\rho^{\rho_0,g},\rho^{\rho_0^1,g^*})$ has the same law of $(\rho^{\rho_0,\tilde{g}},\tilde{\rho}^1)$ since $(g,g^*)$ has the same law of $(\tilde{g},\tilde{g}^*)$.
The remaining term at the right-hand-side of \eqref{estimate_convergence_law} is infinitesimal as $n \to \infty$ by the $\PP^\delta$-almost sure convergence $\tilde{\rho}^{\varepsilon_n,1} \to \tilde{\rho}$ in $\mathcal{E}_*$ given by Jakubowski-Skorokhod. 
Therefore, taking the supremum limit as $n \to \infty$ of both sides, by the discussion above we deduce the following. For each $\delta>0$ there exists a sub-subsequence $\varepsilon_n=\varepsilon_n(\delta)$ such that for each $h\in C_b(\mathcal{E})$ with Lipschitz constant less or equal than $L$
\begin{align*}
    \limsup_{n\rightarrow +\infty} \left\lvert\mathbb{E}^{\varepsilon_n}
\left[h(\rho^{\varepsilon_n,\rho_0^{\varepsilon_n},g^{\varepsilon_n}})\right]-\mathbb{E}^{0}\left[h(\rho^{\rho_0,g})\right]\right\rvert\leq L\delta.
\end{align*}
As $\delta$ is arbitrary, by a diagonal argument we can find a further sub-subsequence such that
\begin{align*}
    \lim_{n\rightarrow +\infty}\left\lvert\mathbb{E}^{\varepsilon_n}
\left[h(\rho^{\varepsilon_n,\rho_0^{\varepsilon_n},g^{\varepsilon_n}})\right]-\mathbb{E}^{0}\left[h(\rho^{\rho_0,g})\right] \right\rvert=0.
\end{align*}
The latter implies the first part of the statement of \autoref{stability_theorem}.

\emph{Step 2: Strong convergence in the auxiliary probability space.} Now we assume $\operatorname{div}\,b=0$. We want to show that up to non-relabeled subsequences, $\tilde{\rho}^{\varepsilon_n,1}\rightarrow \tilde{\rho}^{1}$ in $\mathscr{E},\ \mathbb{P}^{\delta}-a.s.$ This step follows some ideas of \autoref{strong_convergence_eps}, however some changes are required in order to treat the poor analytical convergence of $\tilde{g}^{\varepsilon_n,*}$ to $\tilde{g}^*$. 

Step 1 of the mentioned argument proceeds without any change. 
Secondly, we fix $\eta>0$. Arguing exactly as in Step 4 of the mentioned proof, for each $\eta>0$ we can find $R=R(\eta)$ and a subsequence $\varepsilon_n=\varepsilon_n(\eta)$ such that 
\begin{align*}
    \limsup_{n\rightarrow+\infty }\sup_{t\in \tau_\omega\cap [0,T+\frac{1}{4}]}\int_{\R^d\setminus O_R} \lvert \tilde{\rho}^{\varepsilon_n,1}_t(x)\rvert^2dx\leq \eta,
    \quad  \sup_{t\in \tau_\omega\cap [0,T+\frac{1}{4}]}\int_{\R^d\setminus O_R} \lvert \tilde{\rho}^{1}_t(x)\rvert^2dx\leq \eta .
\end{align*}
Now we continue with {Step 2} of \autoref{strong_convergence_eps}.
Defining the functions $f^n_{\phi},\ f_{\phi}$ in the same way, since $\tilde{\rho}^{\varepsilon_n,1}$ is approaching a renormalized solution, due to the strong convergence in $L^2_{t,*}L^2_x$ and \eqref{bounds_renormalized},
\begin{align*}
     f^n_{\phi}(0)&+\int_0^{T+1}\psi_s \langle (b^*_s+\tilde{g}^{\varepsilon_n,*})\cdot\nabla\phi, \lvert \tilde{\rho}^{\varepsilon_n,1}_s\rvert^2\rangle ds+\int_0^{T+1} \partial_s \psi_s f^n_{\phi}(s)ds
     \\  
    & \rightarrow  
     f_{\phi}(0)+\int_0^{T+1} \psi_s \langle (b^*_s+\tilde{g}^*)\cdot\nabla\phi, \lvert \tilde{\rho}^{1}_s\rvert^2\rangle ds+\int_0^{T+1} \partial_s \psi_s f_{\phi}(s)ds.
\end{align*}
Since $\tilde{\rho}^{\varepsilon_n,1}$ is converging to a renormalized solution, \autoref{lem:beta_rho} yields 
\begin{align*}
\varepsilon_n \sum_{k\in \N}\int_0^{T+1} \psi_s \langle \sigma_k\cdot\nabla\phi, \lvert \tilde{\rho}^{\varepsilon_n,1}_s\rvert^2\rangle dW^{\varepsilon_n,k}_s+\varepsilon_n^2\int_0^{T+1} \psi_s \langle\lvert \tilde{\rho}^{\varepsilon_n,1}_s\rvert^2,\Delta\phi\rangle ds-\langle \langle d\mathcal{D}^{\varepsilon_n,*}, \psi \phi\rangle\rangle   \rightarrow 0. 
\end{align*}
Up to passing to further non-relabeled subsequences, the analogue of \eqref{eq:each_term_individually} holds for $\tilde{\rho}^{\varepsilon_n,1}$ on the time interval $[0,T+1]$.
Therefore, for each $t\in [0,T+\frac{1}{2}],$ also the signed quantity 
\begin{align}\label{convergence_timet_2}
    0\leq \langle \langle d\mathcal{D}^{\varepsilon_n,*}, \one_{[0,t]}\phi\rangle\rangle 
    \leq 
    \langle \langle d\mathcal{D}^{\varepsilon_n,*}, \psi\phi\rangle\rangle   \rightarrow 0.
\end{align}
By similar arguments as those in the proof of the second part of \autoref{lem_Backward_convergence}, there exists $t'\in \tau_\omega\cap (T,T+\frac{1}{4}]$ such that
\begin{align*}
    \lvert f^n_{\phi}(t')-f_{\phi}(t')\rvert 
    &\leq 
    \int_0^{T+1} \lvert\langle b_s\cdot\nabla \phi, \lvert \tilde{\rho}^{\varepsilon_n,1}_s \rvert^2-\lvert \tilde{\rho}^{1}_s \rvert^2\rangle\rvert ds
    + 
    \sup_{t\in [0,T+1]}\varepsilon_n \left\lvert\sum_{k\in \N}\int_0^t \langle \sigma_k\cdot\nabla\phi, \lvert \tilde{\rho}^{\varepsilon_n,1}_s\rvert^2\rangle dW^k_s\right\rvert
    \\
    &+
    \varepsilon_n^2\int_0^{T+1}  \lvert\langle\lvert \tilde{\rho}^{\varepsilon_n,1}_s\rvert^2,\Delta\phi\rangle \rvert^2 ds+\langle \langle d\mathcal{D}^{\varepsilon_n}, \psi\phi\rangle\rangle 
    \\ 
    &+
    \left\lvert \int_0^{T+1} \one_{[0,t']}(s)\left(\langle \tilde{g}^{\varepsilon_n,*}_s\cdot\nabla \phi, \lvert \tilde{\rho}^{\varepsilon_n,1}_s \rvert^2\rangle-\langle \tilde{g}^{*}_s\cdot\nabla \phi, \lvert \tilde{\rho}^{1}_s \rvert^2\rangle\right) ds  \right\rvert\rightarrow 0.
\end{align*}
The latter is a consequence of uniform bounds in $L^\infty_t L^\infty_x$ for $\tilde{\rho}^{\varepsilon_n,1}$ and $\tilde{\rho}^1$, almost sure convergence $\tilde{X}^n \to \tilde{X}$ in the space $\mathbb{X}$ given by Jakubowski-Skorokhod Representation Theorem, the convergence \eqref{eq:each_term_individually}  for $\tilde{\rho}^{\varepsilon_n,1}$ on the time interval $[0,T+1]$, the previous bound \eqref{convergence_timet_2}, and the following computation
\begin{align*}
    &\left\lvert \int_0^{T+1} \one_{[0,t']}(s)\left(\langle \tilde{g}^{\varepsilon_n,*}_s\cdot\nabla \phi, \lvert \tilde{\rho}^{\varepsilon_n,1}_s \rvert^2\rangle-\langle \tilde{g}^{*}_s\cdot\nabla \phi, \lvert \tilde{\rho}^{1}_s \rvert^2\rangle\right) ds  \right\rvert\\ & \leq \int_0^{T+1}\lvert \langle \tilde{g}^{\varepsilon_n,*}_s\cdot\nabla \phi, \lvert \tilde{\rho}^{\varepsilon_n,1}_s \rvert^2-\lvert \tilde{\rho}^{1}_s \rvert^2\rangle\rvert ds+ \left\lvert \int_0^{T+1} \one_{[0,t']}(s)\langle (\tilde{g}^{\varepsilon_n,*}_s-\tilde{g}^{*}_s)\cdot\nabla \phi, \lvert \tilde{\rho}^{1}_s \rvert^2\rangle ds\right\rvert\\ & \leq\lVert \tilde{g}^{\varepsilon_n,*}\cdot\nabla \phi ( \tilde{\rho}^{\varepsilon_n,1}+\tilde{\rho}^{1})  \rVert_{L^2_{t,*}L^2_x}\lVert \tilde{\rho}^{\varepsilon_n,1}-\tilde{\rho}^{1}  \rVert_{L^2_{t,*}L^2_x}\\ & +\left\lvert \int_0^{T+1} \one_{[0,t']}(s)\langle (\tilde{g}^{\varepsilon_n,*}_s-\tilde{g}^{*}_s)\cdot\nabla \phi, \lvert \tilde{\rho}^{1}_s \rvert^2\rangle ds\right\rvert\rightarrow 0.
\end{align*}
Hence, by approximating the indicator function of the ball of radius $R$ we get that for such $t'>T$
\begin{align*}
\left|\int_{O_R}\lvert\tilde{\rho}^{\varepsilon_n,1}_{t'}(x)\rvert^2-\lvert \tilde{\rho}^{1}_{t'}(x)\rvert^2dx \right|\rightarrow 0.
\end{align*}
Let us now consider a generic $t\in [0,t']\cap \tau_\omega$.
In the following we use that the $L^2_x$ norm of $\tilde{\rho}^{\varepsilon_n,1}$ is non-increasing in time by \autoref{lem:beta_rho}, up to taking $t \in \tau_\omega$, whereas the $L^2_x$ norm of renormalized solutions is constant with respect to time.
On the one hand, since at most $\eta$ norm is concentrated outside the ball of radius $R$ and $\tilde{\rho}^{1}$ is renormalized we have
\begin{align*}
    \int_{O_R}\lvert\tilde{\rho}^{\varepsilon_n,1}_{t}(x)\rvert^2 dx&\geq \int_{\R^d}\lvert\tilde{\rho}^{\varepsilon_n,1}_{t}(x)\rvert^2 dx-\int_{\R^d\setminus O_R}\lvert\tilde{\rho}^{\varepsilon_n,1}_{t}(x)\rvert^2 dx\\ & \geq \int_{\R^d}\lvert\tilde{\rho}^{\varepsilon_n,1}_{t'}(x)\rvert^2 dx-\int_{\R^d\setminus O_R}\lvert\tilde{\rho}^{\varepsilon_n,1}_{t}(x)\rvert^2 dx\\ & \geq \int_{O_R}\lvert \tilde{\rho}^{1}_{t'}(x)\rvert^2 dx+ \int_{O_R}\left(\lvert\tilde{\rho}^{\varepsilon_n,1}_{t'}(x)\rvert^2-\lvert \tilde{\rho}^{1}_{t'}(x)\rvert^2\right) dx-\int_{\R^d\setminus O_R}\lvert\tilde{\rho}^{\varepsilon_n,1}_{t}(x)\rvert^2 dx\\ & \geq \int_{O_R}\lvert \tilde{\rho}^{1}_{t}(x)\rvert^2 dx+ \int_{O_R}\left(\lvert\tilde{\rho}^{\varepsilon_n,1}_{t'}(x)\rvert^2-\lvert \tilde{\rho}^{1}_{t'}(x)\rvert^2\right) dx-\int_{\R^d\setminus O_R}\lvert\tilde{\rho}^{\varepsilon_n,1}_{t}(x)\rvert^2 dx-\eta.
\end{align*}
On the other hand 
\begin{align*}
    \int_{O_R}\lvert\tilde{\rho}^{\varepsilon_n,1}_{t}(x)\rvert^2 dx& \leq \int_{\R^d}\lvert\tilde{\rho}^{1}_{t}(x)\rvert^2 dx  \leq \int_{O_R}\lvert\tilde{\rho}^{1}_{t}(x)\rvert^2 dx +\eta. 
\end{align*}
Combining the two above we obtain 
\begin{align*}
\limsup_{n\rightarrow +\infty}\sup_{t\in [0,t']\cap \tau_\omega}\left| \int_{O_R}\lvert\tilde{\rho}^{\varepsilon_n,1}_{t}(x)\rvert^2-\lvert \tilde{\rho}^{1}_{t}(x)\rvert^2dx \right| 
& \leq 
2\eta,
\end{align*}
and consequently, arguing as in {Step 3} of \autoref{strong_convergence_eps},
\begin{align*}
    \limsup_{n\rightarrow +\infty}\sup_{t\in [0,t']\cap \tau_\omega}\int_{O_R}\lvert\tilde{\rho}^{\varepsilon_n,1}_{t}(x)- \tilde{\rho}^{1}_{t}(x)\rvert^2dx& \leq 2\eta.
\end{align*}
Lastly, as in {Step 5} of we obtain \autoref{strong_convergence_eps}
\begin{align*}
     \limsup_{n\rightarrow +\infty}\sup_{t\in [0,T]}\lVert \tilde{\rho}^{\varepsilon_n,1}_{t}- \tilde{\rho}^{1}_{t}\rVert_{L^2_x}^2\leq 6\eta, \quad \mathbb{P}^{\delta} \mbox{-almost surely.}
     \end{align*}
Due to the arbitrariness of $\eta$, by a diagonal argument we can find a subsequence converging almost surely to $\tilde{\rho}^{1}_{t}$ in $L^\infty_tL^2_x$ (actually with the supremum in time on the full $[0,T])$. The convergence takes place also in $L^{\infty}_tL^p_x$, following verbatim the last part of the proof of \autoref{thm:convergence}.

\emph{Step 3: End of the proof.} Concerning the second part of the statement, again it is enough to show that for each subsequence $\varepsilon_n$ there exists a non-relabeled sub-subsequence such that the claim holds. We start observing that due to previous step and by triangle inequality, arguing as in \eqref{estimate_convergence_law}, there exists for each $\delta$ a sub-subsequence, depending on $\delta$, such that
\begin{align}\label{pointwise_inequality_delta}
\limsup_{n\rightarrow+\infty}\norm{\tilde{\rho}^{\varepsilon_n}-\rho^{\rho_0,\tilde{g}}}_{\mathscr{E}}\leq \delta,\quad \mathbb{P}^{\delta}-a.s.
\end{align}
Secondly, due to \autoref{corollary_integrals}
\begin{align*}
\mathbb{E}^{\varepsilon_n}[ d_{\mathscr{E}}(\rho^{{\varepsilon_n},g^{\varepsilon_n}},F) \wedge 1]
&=
{\mathbb{E}}^{\delta}[ d_{\mathscr{E}}(\tilde{\rho}^{\varepsilon_n},F) \wedge 1]
\\
&=
\int_0^1 
{\PP^{\delta}}\{d_{\mathscr{E}}(\tilde{\rho}^{\varepsilon_n},F) > \lambda\} d\lambda
\\&=
\int_0^1  \left( 1-  
{\PP}^{\delta}\{d_{\mathscr{E}}(\tilde{\rho}^{\varepsilon_n},F) \leq \lambda\} \right)  d\lambda.
\end{align*}
Obviously it holds 
\begin{align*}
  \int_{\sqrt{\delta}}^1 \left( 1-  
{\PP}^{\delta}\{d_{\mathscr{E}}(\tilde{\rho}^{\varepsilon_n},F) \leq \lambda\} \right) d\lambda 
&\leq  
\int_0^1  \left( 1-  
{\PP}^{\delta}\{d_{\mathscr{E}}(\tilde{\rho}^{\varepsilon_n},F) \leq \lambda\}\right)  d\lambda
\\
&\leq 
\sqrt{\delta}+\int_{\sqrt{\delta}}^1 \left( 1-  
{\PP}^{\delta}\{d_{\mathscr{E}}(\tilde{\rho}^{\varepsilon_n},F) \leq \lambda\} \right) d\lambda. 
\end{align*}

In addition, it is elementary to check to following set inclusions: 
\begin{align*}
\{d_{\mathscr{E}}({\rho}^{\rho_0,\tilde{g}},\tilde{\rho}^{\varepsilon_n}) \leq \sqrt{\delta}\lambda\}
\cap
\{d_{\mathscr{E}}(\tilde{\rho}^{\varepsilon_n},F) \leq \lambda\}
&\subset 
\{d_{\mathscr{E}}({\rho}^{\rho_0,\tilde{g}},F) \leq (1+\sqrt{\delta})\lambda\},
\\
\{d_{\mathscr{E}}({\rho}^{\rho_0,\tilde{g}},\tilde{\rho}^{\varepsilon_n}) \leq \sqrt{\delta}\lambda\}
\cap
\{d_{\mathscr{E}}({\rho}^{\rho_0,\tilde{g}},F) \leq (1-\sqrt{\delta}) \lambda\}
&\subset 
\{d_{\mathscr{E}}(\tilde{\rho}^{\varepsilon_n},F) \leq \lambda\}.
\end{align*}
Let us observe that for each $\lambda\geq \sqrt{\delta}$, the event $\{d_{\mathscr{E}}({\rho}^{\rho_0,\tilde{g}},\tilde{\rho}^{\varepsilon_n}) > \sqrt{\delta}\lambda\}$ is $\mathcal{F}^{\delta}$-measurable and its probability goes to zero as $n \to \infty$ by Dominated Convergence Theorem and \eqref{pointwise_inequality_delta}.
The measurability of the event follows arguing similarly to \autoref{lem:|rho-f|_closed} and \autoref{lem:same_law}.  
Indeed, since $\mathcal{E}$ is Polish and $\mathcal{E}=\mathscr{E}$ as sets, by those arguments there exists a closed set in $G_{\delta}\subset \mathscr{B}(\mathcal{E}\times \mathcal{E})=\mathscr{B}(\mathcal{E})\otimes \mathscr{B}(\mathcal{E})$ such that 
\begin{align*}
 \{({\rho}^{\rho_0,\tilde{g}},\tilde{\rho}^{\varepsilon_n} )\in \mathscr{E}\times \mathscr{E}:d_{\mathscr{E}}({\rho}^{\rho_0,\tilde{g}},\tilde{\rho}^{\varepsilon_n} ) \leq  \sqrt{\delta}\lambda\}=\{({\rho}^{\rho_0,\tilde{g}},\tilde{\rho}^{\varepsilon_n} )\in G_{\delta}\}   
\end{align*}
which is $\mathcal{F}^{\delta}$-measurable since both $\rho^{\rho_0,\tilde{g}}$ and $ \tilde{\rho}^{\varepsilon_n}:\Omega^{\delta}\rightarrow \mathcal{E}$ are random variables.
Therefore we have
\begin{align*}
{\PP}^{\delta}\{d_{\mathscr{E}}(\tilde{\rho}^{\varepsilon_n},F) \leq \lambda\}
\leq
{\PP}^{\delta}\{d_{\mathscr{E}}(\rho^{\rho_0,\tilde{g}},F) \leq (1+\sqrt{\delta})\lambda\}
+
{\PP}^{\delta}\{d_{\mathscr{E}}({\rho}^{\rho_0,\tilde{g}},\tilde{\rho}^{\varepsilon_n} ) > \sqrt{\delta}\lambda\},\\
{\PP}^{\delta}\{d_{\mathscr{E}}(\tilde{\rho}^{\varepsilon_n},F) \leq \lambda\}\geq {\PP}^{\delta}\{d_{\mathscr{E}}(\rho^{\rho_0,\tilde{g}},F) \leq (1-\sqrt{\delta})\lambda\}-{\PP}^{\delta}\{d_{\mathscr{E}}({\rho}^{\rho_0,\tilde{g}},\tilde{\rho}^{\varepsilon_n} ) > \sqrt{\delta}\lambda\}.
\end{align*}
Consequently, we have
\begin{align*}
\limsup_{ n \to \infty}
{\PP^{\delta}}\{d_{\mathscr{E}}(\tilde{\rho}^{\varepsilon_n},F) \leq \lambda\}
\leq
{\PP}^{\delta}\{d_{\mathscr{E}}({\rho}^{\rho_0,\tilde{g}},F) \leq (1+\sqrt{\delta})\lambda\}.
\end{align*}
On the other hand, by Fatou's lemma and change of variables:
\begin{align*}
\liminf_{n \to \infty}
\mathbb{E}^{\varepsilon_n}[ d_{\mathscr{E}}(\rho^{{\varepsilon_n},g^{\varepsilon_n}},F) \wedge 1]
&=
\liminf_{n \to \infty}
\int_{0}^1 \left( 1-  
{\PP}^{\delta}\{d_{\mathscr{E}}(\tilde{\rho}^{\varepsilon_n},F)) \leq \lambda\}  \right) d\lambda
\\
&\geq
\int_{\sqrt{\delta}}^1 \left(  1- \limsup_{n \to \infty}  
{\PP}^{\delta}\{d_{\mathscr{E}}(\tilde{\rho}^{\varepsilon_n},F)) \leq \lambda\}  \right) d\lambda
\\
&\geq
\int_{\sqrt{\delta}}^1 \left( 1-   
{\PP}^{\delta}\{d_{\mathscr{E}}({\rho}^{\rho_0,\tilde{g}},F) \leq (1+\sqrt{\delta})\lambda\} \right) d\lambda
\\
&=
\int_{\sqrt{\delta}}^1 {\PP}^{\delta}\{d_{\mathscr{E}}({\rho}^{\rho_0,\tilde{g}},F) > (1+\sqrt{\delta})\lambda\}  d\lambda
\\
&=
\frac{1}{1+\sqrt{\delta}}
\int_{\sqrt{\delta}(1+\sqrt{\delta})}^{1+\sqrt{\delta}} {\PP}^{\delta}\{d_{\mathscr{E}}({\rho}^{\rho_0,\tilde{g}},F) > \lambda\}  d\lambda
\\
&\geq
\int_0^{1} {\PP}^{\delta}\{d_{\mathscr{E}}({\rho}^{\rho_0,\tilde{g}},F) > \lambda\}  d\lambda-\frac{\sqrt{\delta}}{1+\sqrt{\delta}}\\ & =\mathbb{E}^{0}[ d_{\mathscr{E}}(\rho^{\rho_0,g},F) \wedge 1]-\frac{\sqrt{\delta}}{1+\sqrt{\delta}}.
\end{align*}

Concerning the other inequality, we also have for each $n\in \N$
\begin{align*}
\mathbb{E}^{\varepsilon_n}[ d_{\mathscr{E}}(\rho^{{\varepsilon_n},g^{\varepsilon_n}},F) \wedge 1]& \leq 
\sqrt{\delta}+\int_{\sqrt{\delta}}^1 \left( 1-  
{\PP}^{\delta}\{d_{\mathscr{E}}(\tilde{\rho}^{\varepsilon_n},F) \leq \lambda\} \right) d\lambda 
\\ 
& \leq  
\sqrt{\delta}+\int_{\sqrt{\delta}}^1 {\PP}^{\delta}\{d_{\mathscr{E}}({\rho}^{\rho_0,\tilde{g}},\tilde{\rho}^{\varepsilon_n}) >\sqrt{\delta}\lambda\}d\lambda
\\ 
&+\int_{\sqrt{\delta}}^1 \left( 1-  
{\PP}^{\delta}\{d_{\mathscr{E}}({\rho}^{\rho_0,\tilde{g}},F) \leq (1-\sqrt{\delta})\lambda\} \right) d\lambda
\\ 
& \leq 
\sqrt{\delta}+\int_{\sqrt{\delta}}^1 {\PP}^{\delta}\{d_{\mathscr{E}}({\rho}^{\rho_0,\tilde{g}},\tilde{\rho}^{\varepsilon_n}) >\sqrt{\delta}\lambda\}d\lambda+{\mathbb{E}}^0[ d_{\mathscr{E}}(\rho^{\rho_0,g},F) \wedge 1].
\end{align*}
Taking the limsup of both sides, by Dominated Convergence  
\begin{align*}
    \limsup_{n\rightarrow+\infty} \mathbb{E}^{\varepsilon_n}[ d_{\mathscr{E}}(\rho^{{\varepsilon_n},g^{\varepsilon_n}},F) \wedge 1] \leq \lambda\}  d\lambda\leq \sqrt{\delta}+{\mathbb{E}}^0[ d_{\mathscr{E}}(\rho^{\rho_0,g},F) \wedge 1].
\end{align*}
Due to the arbitrariness of $\delta$, by a diagonal argument we can find a sub-subsequence such that both of the following are valid:
\begin{align*}
\liminf_{n \to \infty}
\mathbb{E}^{\varepsilon_n}[ d_{\mathscr{E}}(\rho^{{\varepsilon_n},g^{\varepsilon_n}},F) \wedge 1]&\geq \mathbb{E}^{0}[ d_{\mathscr{E}}(\rho^{\rho_0,g},F) \wedge 1],
\\
   \limsup_{n\rightarrow+\infty} \mathbb{E}^{\varepsilon_n}[ d_{\mathscr{E}}(\rho^{{\varepsilon_n},g^{\varepsilon_n}},F) \wedge 1] &\leq {\mathbb{E}}^0[ d_{\mathscr{E}}(\rho^{\rho_0,g},F) \wedge 1].
\end{align*}
Combining the two above, the last claim of \autoref{stability_theorem} follows.
\end{proof}

\begin{prop}\label{compactness_transport_map}
Let $K$ a compact subset of $B$ and let $N\in (0,\infty)$. Denote 
\begin{align*}
\Gamma_{K,N} := \{ \rho=\rho^{\rho_0,g} \,: \, \rho_0 \in K, \, g \in S^N \}.
\end{align*}
Then $\Gamma_{K,M}$ is compact both in $\mathcal{E}$ and $\mathscr{E}$.
\end{prop}
\begin{proof}
    Since $\mathcal{E}$ and $\mathscr{E}$ are metric spaces it is enough to show sequential compactness. Let $\rho^n\in \Gamma_{K,N}$ and $\rho_0^n,\ b+g^n$ the corresponding initial condition and advecting velocity fields. By compactness of $K$ and $S^N$, up to passing to non-relabeled subsequences, we can assume that \begin{align}\label{convergence_data_compactness_lemma}
        \rho_0^n\rightarrow \rho_0\in K,
        \quad
        g^n\rightharpoonup g\in S^N.  
    \end{align}
    Let $\rho\in \mathscr{E}\cap C_t(L^2_x\cap L^p_x)$ be the unique renormalized solution of \eqref{eq:transport} with initial condition $\rho_0$ and advecting velocity fields $b+g$. Then, by \cite[Theorem II.7]{diperna1989ordinary}, relation \eqref{convergence_data_compactness_lemma} implies
    \begin{align*}
        \rho^n\rightarrow \rho\quad \text{in }C_t(L^2_x\cap L^p_x).
    \end{align*}
    The latter readily implies the convergence both in either $\mathcal{E}$ and $\mathscr{E}$ and the claim.
\end{proof}
\section{Large Deviations Principles}\label{Sec_LDP}
\subsection{The Uniform Laplace Principle in $\mathcal{E}$ and its Consequences}\label{Laplace_mathcalE}

We are ready to give the proof of \autoref{thm:LDP}. As already discussed in the Introduction, we apply the weak convergence approach by \cite{Bud_Dup}.

\begin{proof}[Proof of \autoref{thm:LDP}]
    The two conditions of \cite[Theorem 5, Assumption 1]{Bud_Dup} are provided by \autoref{compactness_transport_map} and the first part of \autoref{stability_theorem}. It remains to show that for each $v\in \mathcal{E}$ the map $\rho_0\rightarrow I_{\rho_0}(v)$ is a lower continuous map from $B$ to $[0,+\infty]$. 
    The arguments goes in this way. 
    
    First, let us observe that the infimum appearing in \eqref{def_functional} is attained, if finite. 
    Indeed, let $I_{\rho_0}(v)=M<+\infty$ and $g^n$ be a minimizing sequence. 
    Without loss of generality, we can assume $2M< \lVert g^n \rVert_{L^2_t\mathcal{H}_0}^2\leq 2M+1$ and $\lVert g^n \rVert_{L^2_t \mathcal{H}_0}^2\rightarrow 2M$. Then, up to passing to non-relabeled subsequences, there exists $ g\in L^2_t \mathcal{H}_0$ such that 
    \begin{align*}
        g^n\rightharpoonup g\quad \text{in } L^2_t\mathcal{H}_0, \quad \lVert g \rVert_{L^2_t \mathcal{H}_0}^2\leq \liminf_{n\rightarrow +\infty} \lVert g^n \rVert_{L^2_t \mathcal{H}_0}^2=2M.
    \end{align*}
    Due to the convergence of $g^n$ to $g$, by \cite[Theorem II.7]{diperna1989ordinary}, $v$ is the unique renormalized solution of the transport equation \eqref{eq:transport} with initial condition $\rho_0$ and advecting velocity field $b+g$.
    Therefore $g$ realizes the infimum in \eqref{def_functional}.
    
    Now we are ready to prove the lower-semicontinuity of the map $\rho_0\rightarrow I_{\rho_0}(v)$. 
    Fix $\rho_0\in B$ and a family $\{\rho_0^n\}_{n\in \mathbb{N}}\subset B $ converging to $\rho_0$. Without loss of generality we assume $\liminf_{n\rightarrow +\infty}I_{\rho_0^n}(v)=M<+\infty$, otherwise we have nothing to prove. Therefore, since the infimum in \eqref{def_functional} is attained, there exists a subsequence $n_k$ and family $\{g^{n_k}\}_{n_k\in\mathbb{N}}\subset S^{2M+1}$ such that, for each $k$, $v$ is the unique renormalized solution of \eqref{eq:transport} with initial condition $\rho_0^{n_k}$ and advecting velocity field $b+g^{n_k}$.
Up to passing to a further subsequence, which we continue to denote by $g^{n_k}$ for simplicity of notation, there exists $g\in S^{2M+1}$ such that $g^{n_k}\rightharpoonup g$ in $L^2_t \mathcal{H}_0$. 

Due to the convergence of $g^n \rightharpoonup g$ and $\rho_0^{n_k} \to \rho_0$, by \cite[Theorem II.7]{diperna1989ordinary}, $v$ is the unique renormalized solution of\eqref{eq:transport} with initial condition $\rho_0$ and advecting velocity field $b+g$.
Hence, from the lower-semicontinuity of the norm with respect to the weak convergence: 
\begin{align*}
    I_{\rho_0}(v)\leq \frac{1}{2} \lVert g\rVert^2_{L^2_t \mathcal{H}_0}  
    \leq 
    \liminf_{k\rightarrow +\infty}
    \frac{1}{2}\int_0^T \lVert g^{n_k}\rVert^2_{L^2_t \mathcal{H}_0} \leq M=\liminf_{n\rightarrow +\infty} I_{\rho_0^n}(v).
\end{align*}
The proof is complete.
\end{proof}

\begin{rmk}
A close inspection of the proof of \cite{Bud_Dup} and \cite[Chapter 3, Theorem 3.3]{FrWe_book} shows that the family of processes $\{\rho^{\varepsilon,\rho_0}\}$ also satisfies the pointwise LDP on $\mathcal{E}$ with speed $\varepsilon^2$ and rate function $I_{\rho_0}$.
Moreover, also the alternative versions of pointwise LDP  hold true:
\begin{itemize}
\item 
For every $\rho_0^\varepsilon \to \rho_0$ in $B$ and bounded and continuous functions $h: \mathcal{E} \to \R$ it holds
\begin{align*} 
\lim_{\varepsilon \downarrow 0}
\varepsilon^2 \log \mathbb{E} \left[ \exp(-\varepsilon^{-2}h(\rho^{\varepsilon,\rho_0^\varepsilon}))  \right]
=
-
\inf_{\rho \in \mathscr{X}} \left\{ h(\rho) + I_{\rho_0}(\rho)\right\}.
\end{align*}
\item 
For every $\rho_0^\varepsilon \to \rho_0$ in $B$, $\delta>0$, and $M \in (0,\infty)$ it holds
\begin{align*} 
\limsup_{\varepsilon \downarrow 0}
\sup_{m \leq M}
\left( \varepsilon^2 \log \PP \{ d_\mathcal{E} (\rho^{\varepsilon,\rho_0^\varepsilon}, \Phi_{\rho_0}^m) \geq \delta\} + m \right) 
\leq 0,
\end{align*}
\begin{align*} 
\liminf_{\varepsilon \downarrow 0}
\inf_{\rho \in \Phi_{\rho_0}^{M}}
\left( \varepsilon^2 \log \PP (d_\mathcal{E}(\rho^{\varepsilon,\rho_0^\varepsilon},\rho)<\delta) 
+ I_{\rho_0}(\rho)
\right)
\geq 0.
\end{align*}
\end{itemize}
\end{rmk}

\begin{rmk} \label{rmk:LDP}
Let the family $\{\rho^{\varepsilon,\rho_0}\}$ satisfies the LDP on $\mathscr{X}$ uniformly on compacts with speed $\varepsilon^2$ and rate function $I_{\rho_0}$.
Then for every $F \subset \mathscr{X}$ closed and $G \subset \mathscr{X}$ open it holds
\begin{align*} 
\limsup_{\varepsilon \downarrow 0}
\sup_{\rho_0 \in K} \left( \varepsilon^2 \log \PP \{ \rho^{\varepsilon,\rho_0} \in F \}  + \inf_{\rho \in F} I_{\rho_0}(\rho) \right) \leq 0,
\end{align*}
\begin{align*} 
\liminf_{\varepsilon \downarrow 0}
\inf_{\rho_0 \in K} \left( \varepsilon^2 \log \PP \{ \rho^{\varepsilon,\rho_0} \in G \}  + \inf_{\rho \in G} I_{\rho_0}(\rho) \right) \geq 0.
\end{align*}

To see this, one can adapt the argument of \cite[Chapter 3, Theorem 3.3]{FrWe_book}. We omit the proof.

\end{rmk}

As a direct consequence of \autoref{thm:LDP} we obtain 
a pointwise LDP also for the dissipation measure $\mathcal{D}^\varepsilon$ appearing in the local energy balance of $\rho^\varepsilon$, that is \autoref{cor:LDP_dissipation}.
\begin{proof}[Proof of \autoref{cor:LDP_dissipation}]
Extending the drift $b$ to times $t \in (T,T+1]$ with $b_t \equiv 0$, without loss of generality we can suppose that the solutions $\rho^\varepsilon$ are defined on the time interval $[0,T+1]$, and \autoref{thm:LDP} holds for $\rho^\varepsilon$ on the space $\mathcal{E}_*$ defined in the proof of \autoref{stability_theorem}.
Let $R:=\| \rho_0\|_{L^2_x}^2$. For every $\delta > 0$, let us consider the set
\begin{align*}
T_\delta := \{ \rho \in \mathcal{E}_* : \| \rho_{T} \|_{L^2_x}^2 \leq R - \delta/2 \}.
\end{align*}
By \autoref{lem:beta_rho}, there exists a random measure ${\mathcal{D}}^{\varepsilon,*,T+1}$ taking values in $\mathcal{M}_+([0,T+1] \times \R^d)$ such that 
\begin{align*}
\langle|\rho^\varepsilon_t|^2 , \phi \rangle
&=
\langle|\rho_0^\varepsilon|^2 , \phi \rangle
+ 
\int_0^t  \langle b_s |\rho^\varepsilon_s|^2 , \nabla \phi \rangle ds 
\\
&\quad+  
\sum_{k \in \N}\varepsilon \int_0^t \langle \sigma_k  |\rho^\varepsilon_s|^2 , \nabla \phi \rangle dW^k_s 
+
\varepsilon^2 \int_0^t \langle|\rho^\varepsilon_s|^2 ,\Delta \phi \rangle ds
-
\langle \langle d{\mathcal{D}}^{\varepsilon,*,T+1} , \mathbf{1}_{[0,t]}\phi \rangle \rangle,
\end{align*}
$\PP$-almost surely for almost every $t \in [0,T+1]$ and for every $\phi \in C^\infty_c(\R^d)$, and moreover $\mathcal{D}^\varepsilon$ is the restriction of $\mathcal{D}^{\varepsilon,*,T+1}$ on $[0,T] \times \R^d$, namely ${\mathcal{D}}^{\varepsilon,*,T+1} |_{[0,T] \times \R^d} = \mathcal{D}^\varepsilon$.  
Taking a sequence $\{\phi^n\}_{n \in \N} \subset C^\infty_c(\R^d)$ such that $0 \leq \phi^n \leq 1$ and $\phi^n \uparrow 1$, from the previous line and the fact that ${\mathcal{D}}^{\varepsilon,*,T+1} \geq 0$  we deduce
\begin{align*}
\| \rho_0^\varepsilon \|_{L^2_x}^2 - \| \rho_{t}^\varepsilon \|_{L^2_x}^2
=
\langle \langle d{\mathcal{D}}^{\varepsilon,*,T+1}, \mathbf{1}_{[0,t]} \rangle \rangle,
\end{align*}
$\PP$-almost surely for almost every $t \in [0,T+1]$.
Now take a sequence $t_n \downarrow T$ in this full-measure set.
Since the $L^2_x$ norm is lower-semicontinuous with respect to the convergence in $\mathcal{B}$ and we know that $\rho^\varepsilon \in \mathcal{E}_*$, from the previous identity and the fact that ${\mathcal{D}}^{\varepsilon,*,T+1} |_{[0,T] \times \R^d} = \mathcal{D}^\varepsilon$ we deduce the $\PP$-almost sure global energy inequality at time $t=T$
\begin{align*}
\| \rho_0^\varepsilon \|_{L^2_x}^2 - \| \rho_T^\varepsilon \|_{L^2_x}^2
\geq
\| \rho_0^\varepsilon \|_{L^2_x}^2 - \liminf_{n \to \infty} \| \rho_{t_n}^\varepsilon \|_{L^2_x}^2
=
\limsup_{n \to \infty} \,
\langle \langle d{\mathcal{D}}^{\varepsilon,*,T+1}, \mathbf{1}_{[0,t_n]} \rangle \rangle
\geq
\langle \langle d{\mathcal{D}}^{\varepsilon}, \mathbf{1}_{[0,T]} \rangle \rangle.
\end{align*}
Since $\mathcal{D}^\varepsilon \geq 0$ we have $\langle \langle d{\mathcal{D}}^{\varepsilon}, \mathbf{1}_{[0,T]} \rangle \rangle =  d_{TV}(\mathcal{D}^\varepsilon,0)$.
In addition, since $\rho^\varepsilon_0 \to \rho_0$ in $B$, for every $\delta>0$ there exists $\varepsilon_\delta>0$ such that $\| \rho^\varepsilon_0 \|_{L^2_x}^2 \leq R+\delta/2$ for every $\varepsilon < \varepsilon_\delta$.
In particular, up to $\PP$-negligible sets we have for every $\varepsilon < \varepsilon_\delta$ 
\begin{align*}
\{ \omega \in \O : d_{TV}(\mathcal{D}^\varepsilon,0) \geq \delta \}
\subset 
\{ \omega \in \Omega : \rho^\varepsilon \in T_{\delta} \}
.
\end{align*}
By lower-semicontinuity of the $L^2_x$ norm with respect to the convergence in $\mathcal{B}$, the set $T_\delta$ is closed in $\mathcal{E}_*$. In particular, by \autoref{rmk:LDP} and the pointwise LPD upper bound \eqref{eq:upper_V} for $\rho^\varepsilon$ we get for every $\delta>0$
\begin{align} \label{eq:diss_upper}
\limsup_{\varepsilon \downarrow 0} \varepsilon^2 \log \PP \{ d_{TV}(\mathcal{D}^\varepsilon,0) \geq \delta \}
\leq
\limsup_{\varepsilon \downarrow 0} \varepsilon^2 \log \PP \{ \rho^\varepsilon \in T_\delta \}
\leq
-\inf_{\rho \in T_\delta} I(\rho) = - \infty,
\end{align}
because every renormalized solution of the transport equation with drift $b+g$, $g \in S^N$, satisfies $\| \rho_T\|_{L^2_x} = \| \rho_0 \|_{L^2_x}$.
Since $\PP \{ d_{TV}(\mathcal{D}^\varepsilon,0) \geq \delta \} + \PP \{ d_{TV}(\mathcal{D}^\varepsilon,0) < \delta \} = 1$, \eqref{eq:diss_upper} also implies
\begin{align} \label{eq:diss_lower}
\liminf_{\varepsilon \downarrow 0}
\varepsilon^2 \log \PP \{ d_{TV}(\mathcal{D}^\varepsilon,0) < \delta \}
=
0.
\end{align}
Indeed, suppose instead there is a sequence $\varepsilon_n \downarrow 0$ and $L \in (0,\infty)$ such that $
\PP \{ d_{TV}(\mathcal{D}^{\varepsilon_n},0) < \delta \}
\leq \exp(-\varepsilon_n^{-2}L)$ for every $\varepsilon_n$. 
Since $\limsup_{\varepsilon \downarrow 0} \varepsilon^2 \log \PP \{ d_{TV}(\mathcal{D}^\varepsilon,0) \geq \delta \} = - \infty$ there exists $n_L \in \N$ such that 
\begin{align*}
\PP \{ d_{TV}(\mathcal{D}^{\varepsilon_n},0) \geq \delta \}
\leq \exp(-\varepsilon_n^{-2}L),
\qquad
\forall n \geq n_L,
\end{align*}
implying for every $n \geq n_L$
\begin{align*}
1 
= 
\PP \{ d_{TV}(\mathcal{D}^{\varepsilon_n},0) \geq \delta \} + \PP \{ d_{TV}(\mathcal{D}^{\varepsilon_n},0) < \delta \}
\leq
2  \exp(-\varepsilon_n^{-2}L),
\end{align*}
which is a contradiction since $\varepsilon_n^2 \to \infty$ and $L>0$.

Next we show how \eqref{eq:diss_upper} and \eqref{eq:diss_lower} readily imply a pointwise LDP for $\mathcal{D}^\varepsilon$ with speed $\varepsilon^2$ and rate function $J$.
We will show that \eqref{eq:upper_V} and \eqref{eq:lower_V} hold.
Indeed, let $F$ be a closed set in $\mathcal{M}_+$ and $0 \notin F$ (otherwise $\inf_{\mu \in F}J(\mu) = 0$ and there is nothing to prove). 
Since $F$ is closed, $d_{TV}(0,F) \geq \delta$ for some $\delta>0$ and therefore 
\begin{align*}
\limsup_{\varepsilon \downarrow 0}
\varepsilon^2 \log \PP\{ \mathcal{D}^\varepsilon \in F \} \leq \limsup_{\varepsilon \downarrow 0}
\varepsilon^2 \log \PP\{ d_{TV}(\mathcal{D}^\varepsilon,0) \geq \delta \}=-\infty.
\end{align*}
On the other hand, let $G$ be an open set in $\mathcal{M}_+$ and suppose without loss of generality $0 \in G$. 
Then there exists $\delta>0$ such that the ball in $\mathcal{M}_+$ of radius $\delta$ centered in $0$ is contained in $G$, and therefore
\begin{align*}
\liminf_{\varepsilon \downarrow 0}
\varepsilon^2 \log \PP\{ \mathcal{D}^\varepsilon \in G \}
\geq
\liminf_{\varepsilon \downarrow 0}
\varepsilon^2 \log \PP\{ d_{TV}(\mathcal{D}^\varepsilon,0) < \delta \} = 0.
\end{align*}
\end{proof}

\begin{rmk}
As the rate function $J$ equals $+\infty$ except at $\mu=0$, we can heuristically interpret \autoref{cor:LDP_dissipation} as smallness of $L^2_x$ dissipation for solutions of \eqref{eq:stoch_transport}, up to events of probability $o(\exp(-\varepsilon^{-2}))$.
In our opinions, it would be interesting to understand if a different LDP holds with different speed and non trivial rate function.
\end{rmk}

\subsection{Uniform Large Deviations Principle in $\mathscr{E}$}\label{LDP_mathscr_E}

It only remains to give the:

\begin{proof}[Proof of \autoref{thm:strong_LDP}]
 
First of all, lower-semicontinuity of the map $\rho\rightarrow I_{\rho}(v)$ is valid for each $v\in \mathscr{E}$, by the same argument of the proof of \autoref{thm:LDP}.
Next, for each compact $K \subset B$ and $N \in (0,\infty)$, the set $\Gamma_{K,N}$ is compact in $\mathscr{E}$ by \autoref{compactness_transport_map}.
Moreover, repeating verbatim the argument of the proof of the first part of \cite[Theorem 5]{Bud_Dup}, one can show that the family of maps $I_{\rho_0} : \mathscr{E} \to [0,\infty]$ has compact level sets on compacts. 

Let us finally move to the Large Deviations estimates. 
Let us define the set $\mathscr{K} \subset C_b(\mathscr{E})$ as follows:
\begin{align*}
    \mathscr{K} := \{ h \in C_b(\mathscr{E}) \,:\, h(\rho) := d_\mathscr{E}(\rho,F) \wedge L, \,F \subset \mathscr{E} \mbox{ compact}, \,L > 0\}.
\end{align*} 
Notice that the second part of \autoref{stability_theorem} applies to every function $h \in \mathscr{K}$, since any compact set $F \subset \mathscr{E}$ is a fortiori closed and separable, and one can replace the constant $1$ with $L$ therein without any additional difficulty. 
The family $\mathscr{K}$ is rich enough in terms of Laplace asymptotics, at least when one is interested in Large Deviations estimates. Indeed, by looking at the proof of \cite[Proposition 1.14]{BuDu19}, the LDP uniformly on compacts is implied by the validity of 
\begin{align} \label{eq:Laplace_unif}
 \lim_{\varepsilon \downarrow 0}
 \sup_{\rho_0 \in K}
 \left|
\varepsilon^2 \log \mathbb{E} \left[ \exp(-\varepsilon^{-2}h(\rho^{\varepsilon,\rho_0}))  \right]
+
\inf_{\rho \in \mathscr{E}} \left\{ h(\rho) + I_{\rho_0}(\rho)\right\}
\right|
=
0,
\quad
\forall h \in \mathscr{K}.
\end{align} 
By \cite[Proposition 1.12]{BuDu19}, in order to prove \eqref{eq:Laplace_unif} it suffices to check the pointwise condition
\begin{align} \label{eq:Laplace}
 \lim_{\varepsilon \downarrow 0}
\varepsilon^2 \log \mathbb{E} \left[ \exp(-\varepsilon^{-2}h(\rho^{\varepsilon,\rho_0^\varepsilon}))  \right]
=-
\inf_{\rho \in \mathscr{E}} \left\{ h(\rho) + I_{\rho_0}(\rho)\right\},
\quad
\forall h \in \mathscr{K},
\end{align}
whenever $\rho^\varepsilon_0 \to \rho_0$ in $B$.
To prove the latter, we follow the proof of \cite[Theorem 5]{Bud_Dup}, for which we need the results of \autoref{Sec_Meas_Stab} since the space $\mathscr{E}$ is not Polish.

From now on, let $h \in \mathscr{K}$ be fixed and consider the map $h(\rho^{\varepsilon,\rho_0^\varepsilon}) : (\Omega,\mathcal{F},\PP) \to \R$. 
By \autoref{corollary_integrals}, this map is measurable with respect to the augmented $\sigma$-field generated by $\varepsilon W : (\Omega,\mathcal{F},\PP) \to C_t H$, where $H$ is an abstract Hilbert space hosting the trajectories of the Wiener process $W = \sqrt{\mathcal{Q}}\mathcal{W}$.
By Doob measurability theorem there exists a Borel measurable map $H^\varepsilon : C_t H \to \R$ such that $H^\varepsilon(\varepsilon W) = h(\rho^{\varepsilon,\rho_0^\varepsilon})$, $\PP$-almost surely. 
Moreover, since $h$ is bounded, $H^\varepsilon$ is bounded on a subset of $C_tH$ which has full-measure with respect to the law of $\varepsilon W$. 
Without any loss of generality, we can modify $H^\varepsilon$ in a measurable way so that $H^\varepsilon$ is bounded on the whole $C_t H$, and the $\PP$-almost sure relation $H^\varepsilon(\varepsilon W) = h(\rho^{\varepsilon,\rho_0^\varepsilon})$ continues to hold.

We shall make use of the following fundamental representation formula \cite[Theorem 3]{Bud_Dup}:
\begin{align*}
- \varepsilon^2 \log \mathbb{E}[\exp(-\varepsilon^{-2}H^\varepsilon(\varepsilon W))] 
=
\inf_{g \in \mathcal{P}_2^N}
\mathbb{E} \left[ 
\frac12 \|g \|_{L^2_t \mathcal{H}_0}^2
+
H^\varepsilon \left( \varepsilon W + \int_0^\cdot g_t dt \right)
\right],
\end{align*}
where $N$ is finite and large enough, depending only on $\sup |h|$, and $\mathcal{P}_2^N$ has been defined in the Introduction. 
By \autoref{prop_well_posed} and \cite[Theorem 3.14]{kurtz}, for every $\varepsilon \in (0,1)$ and $g \in \mathcal{P}_2^N \cap \mathcal{P}^N_{prog}$ uniqueness in law holds for $\rho^{\varepsilon,\rho_0^\varepsilon,g}$ as $\mathcal{E}$-valued random variable.
Moreover, by \autoref{corollary_integrals}, the law of the real-valued random variable $h(\rho^{\varepsilon,\rho_0^\varepsilon,g})$ only depends on the law of $\rho^{\varepsilon,\rho_0^\varepsilon,g}$ on $\mathcal{E}$. Therefore, by Girsanov Theorem, for every $\varepsilon \in (0,1)$ and $g \in \mathcal{P}_2^N \cap \mathcal{P}^N_{prog}$ it holds
\begin{align} \label{eq:kurtz}
\mathbb{E} \left[
 H^\varepsilon \left( \varepsilon W + \int_0^\cdot g_t dt \right)
 \right]
=
\mathbb{E} \left[ h (\rho^{\varepsilon,\rho_0^\varepsilon,g}) \right].
\end{align}

By definition of $H^{\varepsilon}$ and the previous two displays, we can rewrite
\begin{align*}
- \varepsilon^2 \log \mathbb{E}[\exp(-\varepsilon^{-2}h(\rho^{\varepsilon,\rho_0^\varepsilon}))] 
=
\inf_{g \in \mathcal{P}_2^N}
\mathbb{E} \left[ 
\frac12 \|g \|_{L^2_t \mathcal{H}_0}^2
+
h(\rho^{\varepsilon,\rho_0^\varepsilon,g})
\right].
\end{align*}
The proof of \eqref{eq:Laplace} comes in two steps. 

\emph{Upper bound.}
For every $\delta>0$ and $\varepsilon \in (0,1)$, we can find $g^\varepsilon \in \mathcal{P}_2^N \cap \mathcal{P}^N_{prog}$ such that
\begin{align*}
 \inf_{g \in \mathcal{P}_2^N}
\mathbb{E} \left[ 
\frac12 \|g \|_{L^2_t \mathcal{H}_0}^2
+
h(\rho^{\varepsilon,\rho_0^\varepsilon,g})
\right]
\leq
\mathbb{E} \left[ 
\frac12 \|g^\varepsilon \|_{L^2_t \mathcal{H}_0}^2
+
h(\rho^{\varepsilon,\rho_0^\varepsilon,g^\varepsilon})
\right] - \delta.
\end{align*}    

Up to the extraction of a (non-relabeled) subsequence, $g^\varepsilon$ converges in law as random variables in $S^N$, towards a stochastic process $g$ defined on a probability space $(\Omega^0,\mathcal{F}^0,\PP^0)$.
By lower-semicontinuity of the $L^2_t \mathcal{H}_0$ norm with respect to weak convergence and previous \autoref{stability_theorem}, we have
\begin{align*}
\liminf_{\varepsilon \downarrow 0}
\mathbb{E} \left[ 
\frac12 \|g^\varepsilon \|_{L^2_t \mathcal{H}_0}^2
+
h (\rho^{\varepsilon,\rho_0^\varepsilon,g^\varepsilon})
\right] 
\geq
\mathbb{E}^0 \left[ 
\frac12 \|g \|_{L^2_t \mathcal{H}_0}^2
+
h (\rho^{\rho_0,g})
\right] 
\geq
\inf_{\rho \in \mathscr{E}}
\left\{ I_{\rho_0}(\rho) + h(\rho) \right\}.
\end{align*}
Since $\delta>0$ is arbitrary this proves the upper bound.

\emph{Lower bound.}
Fix $\delta>0$ and let $\rho^\star \in \mathscr{E}$, $g^\star \in S^N$ be such that $\rho^\star = \rho^{\rho_0,g^\star}$ and
\begin{align*}
 I_{\rho_0}(\rho^\star) + h(\rho^\star) 
\leq
\inf_{\rho \in \mathscr{E}}
\left\{ I_{\rho_0}(\rho) + h(\rho) \right\} + \delta,
\qquad
\frac12\|g^\star \|_{L^2_t \mathcal{H}_0}^2
\leq
I_{\rho_0}(\rho^\star)+\delta.
\end{align*}
Clearly $g^\star \in \mathcal{P}^N_2 \cap \mathcal{P}^N_{prog}$. Then by \autoref{stability_theorem} we have the following chain of inequalities: 
\begin{align*}
\limsup_{\varepsilon \downarrow 0} & \inf_{g \in \mathcal{P}_2^N}
\mathbb{E} \left[ 
\frac12 \|g \|_{L^2_t \mathcal{H}_0}^2
+
h(\rho^{\varepsilon,\rho_0^\varepsilon,g^\varepsilon})
\right]
\\&\leq
\limsup_{\varepsilon \downarrow 0}
\mathbb{E} \left[ 
\frac12 \|g^\star \|_{L^2_t \mathcal{H}_0}^2
+
h(\rho^{\varepsilon,\rho_0^\varepsilon,g^\star})
\right]
\\
&= 
\frac12 \|g^\star \|_{L^2_t \mathcal{H}_0}^2
+\limsup_{\varepsilon \downarrow 0}\mathbb{E} \left[ 
h (\rho^{\varepsilon,\rho_0^{\varepsilon},g^\star})
\right]
\\
&=
\frac12 \|g^\star \|_{L^2_t \mathcal{H}_0}^2
+ 
h (\rho^{\star})
\leq
\inf_{\rho \in \mathscr{E}}
\left\{ I_{\rho_0}(\rho) + h(\rho) \right\} + 2\delta,
\end{align*}
The lower bound descends from arbitrariness of $\delta$.
\end{proof}

\begin{acknowledgements}
    We are grateful to Amarjit Budhiraja for pointing out reference \cite{BuDu19} and to Mario Maurelli for the many stimulating discussions.  GC is supported by SNF Project 212573 FLUTURA – Fluids, Turbulence, Advection and by the SPP 2410 “Hyperbolic Balance Laws in Fluid Mechanics: Complexity, Scales, Randomness (CoScaRa)” funded by the Deutsche Forschungsgemeinschaft (DFG, German Research Foundation) through the project 200021E\_217527 funded by the Swiss National Science Foundation.
	EL and UP have received funding from the European Research Council (ERC) under the European Union’s Horizon 2020 research and innovation programme (grant agreement No. 949981).  
\end{acknowledgements}

\appendix

\section{Proof of \autoref{prop_well_posed}}

\begin{definition} \label{def:weak_sol}
Let $\rho$ be a $\{ \mathcal{F}_t \}_{t \geq 0}$-progressively measurable process of class $L^\infty_\omega L^\infty_t L^2_x$ and with trajectories in $C_t H^{-\sigma}_x$ $\PP$-almost surely for some $\sigma>0$.
We say that $\rho$ is a probabilistically strong, analytically weak solution of \eqref{eq:rho_general} if for every test function $\varphi \in C_c^\infty(\R^d)$, it holds $\PP$-almost surely for every $t \in [0,T]$
\begin{align} \label{eq:notion_solution}
\langle \rho_t , \varphi \rangle
&=
\langle \rho_0 , \varphi \rangle
+
\int_0^t \langle (b_s+g_s) \rho_s , \nabla \varphi  \rangle ds
+
\int_0^t \langle \dvg(b_s)\rho_s , \varphi \rangle ds
\\ \nonumber
\quad&+
 \varepsilon \sum_{k \in \N}
 \int_0^t  \langle \sigma_k \rho_s, \nabla \varphi \rangle dW^k_s
+
(1+\kappa)\varepsilon^2 
\int_0^t \langle \rho_s , \Delta \varphi \rangle ds.
\end{align}
\end{definition}
By standard argument, see for example \cite[Theorem 1.7]{flandoli2023stochastic}, the definition above implies that for each $t\in [0,T]$ and $\varphi\in C^\infty([0,T]\times \R^d)$ it holds
\begin{align*}
\langle \rho_t , \varphi_t \rangle
&=
\langle \rho_0 , \varphi_0 \rangle+\int_0^t \langle\rho_s,\partial_s\varphi_s\rangle ds
+
\int_0^t \langle (b_s+g_s) \rho_s , \nabla \varphi_s  \rangle ds
+
\int_0^t \langle \dvg(b_s)\rho_s , \varphi_s \rangle ds
\\ \nonumber
\quad&+
 \varepsilon \sum_{k \in \N}
 \int_0^t  \langle \sigma_k \rho_s, \nabla \varphi_s \rangle dW^k_s
+
(1+\kappa)\varepsilon^2 
\int_0^t \langle \rho_s , \Delta \varphi_s \rangle ds\quad\mathbb{P}-a.s.    
\end{align*}

Existence and uniqueness in $L^\infty_\omega L^\infty_t L^2_x \cap L^{\infty}_\omega L^2_t H^1_x$ of solutions to \eqref{eq:rho_general} with $\kappa>0$ and smooth noise (obtained e.g. by taking only finitely many non-zero $\sigma_k$'s) is classical \cite[Chapters 3-5]{Flandoli_Book_95}.
Existence of solutions in the general case of rough noise is obtained by approximation and standard compactness arguments, and the bound in $L^2_t H^1_x$ guarantees uniqueness of solutions with $\kappa>0$ by standard Grownwall-type estimates.
In this section we present a proof of \autoref{prop_well_posed}. We focus on the case $\kappa=0$, the other being easier.

\subsection{Pathwise uniqueness when $g=0$} \label{ssec:pathwise_uniqueness}

The proof basically replicates the proof of the energy balance \cite[Theorem 5.1]{GaGrMa24+}, with the only difference that we prove an estimate in the inhomogeneous Sobolev space $H^{1-\alpha-\delta}_x$ and we do not assume that the initial condition $\rho_0$ belongs to a homogeneous Sobolev space of negative regularity.

Let $e_\xi := (2\pi)^{-d/2} e^{-i \xi \cdot x}$ and define $a_t(\xi) := \mathbb{E} [|\hat{\rho}_t(\xi)|^2]$, where the Fourier transform $\hat{\rho}(\xi) = \langle  \rho , e_\xi \rangle$ is a $L^\infty_\omega L^\infty_t L^2_\xi$-valued process in virtue of the assumption $\rho \in L^\infty_\omega L^\infty_t L^2_x$.

Notice that, with our choice of the coefficient in front of $e_\xi$, the Fourier transform of $Q$ satisfies $(2\pi)^{-d/2} \int_{\R^d} \hat{Q}(\eta) d\eta = Q(0) = 2 I_d$ by construction.
By It\=o formula and the same passages of \cite[Lemma 3.1]{GaGrMa24+}, we get for every compactly supported test function $\psi \in C^\infty_c(\R^d)$:
\begin{align} \label{eq:a_psi}
 \int_{\mathbb{R}^d}  &a_t(\xi) \psi(\xi) d\xi 
 - 
 \int_{\mathbb{R}^d} a_0(\xi)\psi(\xi) d\xi 
\\ \nonumber
&= 
\varepsilon^2 (2\pi)^{-d/2} \int_0^t \int_{\mathbb{R}^d}\int_{\mathbb{R}^d}  \frac{|P^\perp_{\xi-\eta} \xi|^2 a_s(\xi)}{\langle \xi - \eta \rangle^{d+2\alpha}} (\psi(\eta) - \psi(\xi)) d\xi d\eta ds
\\ \nonumber
   &\quad
 2\int_0^t \int_{\mathbb{R}^d} 
 \mathbb{E}\left[\mathfrak{Re} \left(\overline{\hat{\rho}_s(\xi)} \langle b_s \cdot \nabla \rho_s , e_\xi \rangle \right)\right] \psi(\xi) d\xi ds.
\end{align}

We are eventually interested in fixing $\psi(\xi) := \langle \xi \rangle^{-2\delta} := (1+|\xi|^2)^{-\delta}$ in \eqref{eq:a_psi}, so that we get an equation for the evolution of the negative Sobolev norm $\int_{\mathbb{R}^d}  a_t(\xi) \psi(\xi) d\xi = \mathbb{E}\| \rho_t \|_{H^{-\delta}_x}^2$.
Since $\langle \xi \rangle^{-2\delta}$ is not compactly supported, we need to argue by approximation. To do that, we follow the same steps of \cite[Proposition 3.2]{GaGrMa24+}. 
Let $\varphi \in C^\infty_c(\R_+)$ be a non-increasing function such that $0 \leq \varphi \leq 1$, $\varphi \equiv1$ on the interval $[0,1]$, and $\varphi \equiv 0$ on the interval $[2,\infty)$. Define $\varphi_n(r):=\varphi(r/n)$ and $\psi_n (\xi) := \varphi_n(|\xi|) \langle \xi \rangle^{-2\delta}$. 
Now we evaluate $\eqref{eq:a_psi}$ with $\psi = \psi_n$ and send $n \to \infty$. 
Recall that $\rho$ has continuous trajectories in $H^{-\delta}_x$ $\PP$-almost surely. In particular, for every $t \in [0,T]$ the left-hand-side of \eqref{eq:a_psi} converges to
\begin{align} \label{eq:conv_1}
\ \int_{\mathbb{R}^d}  a_t(\xi) \psi_n(\xi) d\xi
 -
\int_{\mathbb{R}^d} a_0(\xi)\psi_n(\xi) d\xi
 &\to
 \mathbb{E} \| \rho_t \|_{H^{-\delta}_x}^2
-
 \| \rho_0 \|_{H^{-\delta}_x}^2.
\end{align}

In order to study the convergence of the right-hand-side, we need the following lemmas.
\begin{lem} \label{lem:F_fullspace}
Let $\delta \in (0,1)$ and let us define the functions
\begin{align*}
    F^n(\xi) &:= (2\pi)^{-d/2}\int_{\mathbb{R}^d}  \frac{|P^\perp_{\xi-\eta} \xi|^2}{\langle \xi - \eta \rangle^{d+2\alpha}} 
    \left( \frac{\varphi_n(|\eta|)}{\langle \eta \rangle^{2\delta}} -\frac{\varphi_n(|\xi|)}{\langle \xi \rangle^{2\delta}} \right) d\eta,
    \\
     \dot{F}^n(\xi) &:= (2\pi)^{-d/2}\int_{\mathbb{R}^d}  \frac{|P^\perp_{\xi-\eta} \xi|^2}{\langle \xi - \eta \rangle^{d+2\alpha}} 
     \left( \frac{\varphi_n(|\eta|)}{| \eta |^{2\delta}} -\frac{\varphi_n(|\xi|)}{| \xi |^{2\delta}} \right) d\eta.
\end{align*}
Define similarly $F,\dot{F}$ replacing $\varphi_n$ with the constant $1$ in the previous lines. Then there exists $R\gg 1$ such that we can bound 
    \begin{align*}
     |F(\xi)-\dot{F}(\xi)| + \sup_{n \geq 1}  |F^n(\xi)-\dot{F}^n(\xi)|
        \lesssim
        \langle \xi \rangle^{-2\delta \left( \frac{d+2\alpha}{d+2}\right)} ,
        \quad
        \forall |\xi| \geq R.
    \end{align*}
\end{lem}

\begin{proof}
For notational simplicity, we write $\varphi_\infty \equiv 1$. In this way, $F = F^\infty$ and $\dot{F} = \dot{F}^\infty$. Now we have for every $n \in [1,\infty]$:
\begin{align*}
F^n(\xi)-\dot{F}^n(\xi) 
&= 
(2\pi)^{-d/2}\int_{\mathbb{R}^d}  \frac{|P^\perp_{\xi-\eta} \xi|^2}{\langle \xi - \eta \rangle^{d+2\alpha}} 
\left( \frac{\varphi_n(|\eta|)}{\langle \eta \rangle ^{2\delta}}-\frac{\varphi_n(|\eta|)}{| \eta |^{2\delta}} -\frac{\varphi_n(|\xi|)}{\langle \xi \rangle^{2\delta}} + \frac{\varphi_n(|\xi|)}{| \xi |^{2\delta}} \right) d\eta
\\
&= 
(2\pi)^{-d/2}\int_{\mathbb{R}^d}  \frac{|P^\perp_{\xi-\eta} \xi|^2}{\langle \xi - \eta \rangle^{d+2\alpha}} 
\left( 
\varphi_n(|\eta|)\frac{| \eta |^{2\delta}-(1+| \eta |^2)^{\delta}}{| \eta |^{2\delta}\langle \eta \rangle^{2\delta}}
-
\varphi_n(|\xi|)\frac{| \xi |^{2\delta}-(1+| \xi |^2)^{\delta}}{| \xi |^{2\delta}\langle \xi \rangle ^{2\delta}}
\right) d\eta.    
\end{align*}    
Now we use that $|\varphi_n| \leq 1$ and for the function
\begin{align*}
    G(t) := \frac{t^{{\delta}}-(1+t)^{{\delta}}}{t^{{\delta}}(1+t)^{{\delta}}},
    \quad
    t > 0,
\end{align*}
we have the elementary inequality 
\begin{align*}
|G(t)| \leq \delta t^{-1-\delta},
\quad
\forall t>0.
\end{align*}
For $|\xi| \geq R$ the term with $G(|\xi|^2)$ is then controlled by a constant times $|\xi|^{-2-2\delta}$, thus using that $\alpha>0$ we have uniformly in $n \in [1,\infty]$:
\begin{align*}
\int_{\mathbb{R}^d}  \frac{|P^\perp_{\xi-\eta} \xi|^2}{\langle \xi - \eta \rangle^{d+2\alpha}} 
\varphi_n(|\xi|)\left|    \frac{| \xi |^{2\delta}-(1+| \xi |^2)^{{\delta}}}{| \xi |^{2\delta}\langle \xi \rangle ^{2\delta}} \right|d\eta
\lesssim
\left|    \int_{\mathbb{R}^d}  
\frac{|\xi|^{-2\delta}}{\langle \xi - \eta \rangle^{d+2\alpha}}  d\eta
\right|
\lesssim
\langle \xi \rangle^{-2\delta}.  
\end{align*}
On the other hand, for some $r \ll \langle \xi \rangle$ to be chosen later, we can decompose the integral over the regions $\{|\eta|>r\}$ and $\{ |\eta| \leq r\}$, and bound uniformly in $n \in [1,\infty]$:
\begin{align*}
\int_{\mathbb{R}^d}  \frac{|P^\perp_{\xi-\eta} \xi|^2}{\langle \xi - \eta \rangle^{d+2\alpha}} 
\varphi_n(|\eta|)\left| 
\frac{| \eta |^{2\delta}-(1+| \eta |^2)^{{\delta}}}{| \eta |^{2\delta}\langle \eta \rangle ^{2\delta}}
\right| d\eta
&\leq
    \int_{|\eta| > r}  \frac{|P^\perp_{\xi-\eta} \xi|^2}{\langle \xi - \eta \rangle^{d+2\alpha}} 
\left| 
\frac{| \eta |^{2\delta}-(1+| \eta |^2)^{{\delta}}}{| \eta |^{2\delta}\langle \eta \rangle ^{2\delta}}
\right| d\eta
\\
\quad&+
    \int_{|\eta| \leq  r}  \frac{|P^\perp_{\xi-\eta} \xi|^2}{\langle \xi - \eta \rangle^{d+2\alpha}} 
\left| 
\frac{| \eta |^{2\delta}-(1+| \eta |^2)^{{\delta}}}{| \eta |^{2\delta}\langle \eta \rangle ^{2\delta}}
\right| d\eta.
\end{align*}
Using $|P^\perp_{\xi-\eta} \xi|^2= |P^\perp_{\xi-\eta} \eta|^2$, the former can be controlled by
\begin{align*}
    \int_{|\eta| > r}  \frac{|P^\perp_{\xi-\eta} \eta|^2}{\langle \xi - \eta \rangle^{d+2\alpha}} 
\left| 
\frac{| \eta |^{2\delta}-(1+| \eta |^2)^{{\delta}}}{| \eta |^{2\delta}\langle \eta \rangle ^{2\delta}}
\right| d\eta
\lesssim 
\frac{1}{r^{2\delta}} \int_{|\eta| > r}  \frac{1}{\langle \xi - \eta \rangle^{d+2\alpha}} d\eta \lesssim \frac{1}{r^{2\delta}}.
\end{align*}
The latter can be controlled by
\begin{align*}
\int_{|\eta| \leq r}  \frac{|P^\perp_{\xi-\eta} \eta|^2}{\langle \xi - \eta \rangle^{d+2\alpha}} 
\left| 
\frac{| \eta |^{2\delta}-(1+| \eta |^2)^{{\delta}}}{| \eta |^{2\delta}\langle \eta \rangle ^{2\delta}}
\right| d\eta
&\lesssim
\int_{|\eta| \leq r}  \frac{|P^\perp_{\xi-\eta} \eta|^2}{\langle \xi - \eta \rangle^{d+2\alpha}} 
| \eta |^{-2\delta} d\eta
\\
&\lesssim
\frac{r^{d+2-2\delta}}{\langle \xi \rangle^{d+2\alpha}}.
\end{align*}
Next we optimize over $r \ll \langle \xi \rangle$, namely we have to choose $r$ depending on $\langle \xi \rangle$ such that $r^{-2\delta} = r^{d+2-2\delta}{\langle \xi \rangle^{-d-2\alpha}}$. This gives $r = \langle \xi \rangle^\frac{d+2\alpha}{d+2}$ and we notice that $r \ll \langle \xi \rangle$ is indeed satisfied by this choice since $\alpha<1$. 
Plugging back this choice of $r$ into either of the two bounds above we arrive to
\begin{align*}
\int_{\mathbb{R}^d}  \frac{|P^\perp_{\xi-\eta} \xi|^2}{\langle \xi - \eta \rangle^{d+2\alpha}} 
\left| 
\frac{| \eta |^{2\delta}-(1+| \eta |^2)^{{\delta}}}{| \eta |^{2\delta}\langle \eta \rangle ^{2\delta}}
\right| d\eta
\lesssim
\langle \xi \rangle^{-2\delta\left(\frac{d+2\alpha}{d+2}\right)}.
\end{align*}
As $\langle \xi \rangle^{-2\delta} \ll \langle \xi \rangle^{-2\delta\left(\frac{d+2\alpha}{d+2}\right)}$ for $\langle \xi \rangle \gg 1$, we get the desired result.
\end{proof}

\begin{cor} \label{cor:conv_2}
Let $\delta \in (1-\alpha,1)$, then:
\begin{align*}
\lim_{n \to \infty}
&\varepsilon^2 (2\pi)^{-d/2} \int_0^t \int_{\mathbb{R}^d}\int_{\mathbb{R}^d}  \frac{|P^\perp_{\xi-\eta} \xi|^2 a_s(\xi)}{\langle \xi - \eta \rangle^{d+2\alpha}} (\psi_n(\eta) - \psi_n(\xi)) d\xi d\eta ds
\\
&=
\varepsilon^2 (2\pi)^{-d/2} \int_0^t \int_{\mathbb{R}^d}\int_{\mathbb{R}^d}  \frac{|P^\perp_{\xi-\eta} \xi|^2 a_s(\xi)}{\langle \xi - \eta \rangle^{d+2\alpha}} (\psi(\eta) - \psi(\xi)) d\xi d\eta ds.
\end{align*}
\end{cor}
\begin{proof}
By \autoref{lem:F_fullspace} and the proof of \cite[Proposition 3.2, Case 2]{GaGrMa24+}, the functions $F^n(\xi)$ are uniformly bounded by a finite constant for every $n \geq 1$ and $|\xi| \geq R$. On the other hand, for $|\xi|<R$ the trivial inequality holds
\begin{align*}
|F^n(\xi)| 
&\lesssim 
\int_{\mathbb{R}^d}  \frac{|P^\perp_{\xi-\eta} \xi|^2}{\langle \xi - \eta \rangle^{d+2\alpha}} 
    \left( \frac{1}{\langle \eta \rangle^{2\delta}} +\frac{1}{\langle \xi \rangle^{2\delta}} \right) d\eta
\\
&\lesssim 
R^2 \int_{\mathbb{R}^d}  \frac{1}{\langle \xi - \eta \rangle^{d+2\alpha}\langle \eta \rangle^{2\delta}} 
    d\eta
+
R^{2-2\delta} \int_{\mathbb{R}^d}  \frac{1}{\langle \xi - \eta \rangle^{d+2\alpha}} 
    d\eta
\lesssim 1,
\end{align*}
with implicit constant depending on the parameter $R$ from previous lemma. Since $a(\xi) \in L^\infty_t L^1_x$ we obtain the thesis by Dominated Convergence.
\end{proof}

We also need to control the term that involves the drift $b$.
\begin{lem} \label{lem:b}
There exists $\delta \in (1-\alpha,1)$, $0<\lambda \ll 1$ and an implicit constant depending on $d,\alpha,\delta,\lambda,q$ and the norms $\| b \|_{L^\infty_t W^{1,q}_x}$ and $\| \dvg \, b \|_{L^\infty_t H^{\vartheta}_x}$, such that for every $t \in [0,T]$:
\begin{align} \label{eq:ERe(b+g)}
\sup_{n \geq 1}\int_0^t \int_{\mathbb{R}^d} 
\left| \mathbb{E}\left[\mathfrak{Re} \left( \overline{\hat{\rho}_s(\xi)} \langle b_s \cdot \nabla  \rho_s , e_\xi \rangle \right)\right] \right| \psi_n(\xi) d\xi ds
& \lesssim
\int_0^t \mathbb{E} \| \rho_s \|_{H^{1-\alpha-\delta-\lambda}_x}^2 ds.
\end{align}
\end{lem}
\begin{proof}

First of all, notice that by integration by parts
\begin{align*}
\int_0^t &\int_{\mathbb{R}^d} 
\left| \mathbb{E}\left[\mathfrak{Re} \left( \overline{\hat{\rho}_s(\xi)} \langle b_s \cdot \nabla  \rho_s , e_\xi \rangle \right)\right] \right| \psi_n(\xi) d\xi ds
\\ 
&\lesssim
\int_0^t \int_{\mathbb{R}^d} 
\left| \mathbb{E}\left[\mathfrak{Re} \left( \overline{\hat{\rho}_s(\xi)} \langle b_s  \rho_s , e_\xi \rangle \right)\right] \right| |\xi| \psi(\xi) d\xi ds
\\
&+
\int_0^t \int_{\mathbb{R}^d} 
\left| \mathbb{E}\left[\mathfrak{Re} \left( \overline{\hat{\rho}_s(\xi)} \langle \dvg( b_s) \rho_s , e_\xi \rangle \right)\right] \right| \psi(\xi) d\xi ds.
\end{align*}

We control each term separately.
Since $q \leq 2$, by Sobolev embedding, see for example \cite[Section 2.7]{Triebel83}, 
\begin{align*}
W^{1,q}_x \subset L^{q^*}_x, \quad\,\,
\, 1/q^* &:= 1/q - 1/d, 
\\
W^{1,q}_x \subset H^{\beta}_x, \quad
\, 0 < \beta &:=  1 + d/2 - d/q \leq 1.
\end{align*}
Notice that the condition $\beta>0$ comes from the assumptions $q > \frac{d}{2(1-\alpha)}$ and $\alpha > 0$.
Let us fix $0<\lambda < \theta \ll 1$ sufficiently small, in such a way that $\delta := 1- \alpha+\theta-\lambda \in (1-\alpha,1)$ and $\theta < \beta$.
For this choice of parameters it holds
\begin{align} \label{eq:key_estimate}
\| b_s \rho_s \|_{H^{\beta'}_x}
\lesssim
\| b_s \|_{W^{1,q}_x}
\| \rho_s \|_{H^{-\theta}_x},
\quad
\beta' := -\theta-d/q+1.
\end{align}
Indeed, denoting by $\varodot$, $\varolessthan$, and $\varogreaterthan$ the three paraproducts, according to \cite[Theorems 27.5-27.10]{van2022theory}, see also \cite[Theorem 3.17]{Paraprod_1}:
\begin{align*}
\| b \varodot \rho \|_{B^{\beta-\theta}_{1,1}}
&\lesssim
\| b \|_{H^\beta_x} \| \rho \|_{H^{-\theta}_x},
\\
\| b \varolessthan \rho \|_{B^{-\theta}_{p_*,2}}
&\lesssim
\| b \|_{L^{q*}_x} \| \rho \|_{H^{-\theta}_x},
\quad
1/p_* = 1/q^* + 1/2,
\\
\| b \varogreaterthan \rho \|_{B^{\beta-\theta}_{1,1}}
&\lesssim
\| b \|_{H^\beta_x} \| \rho \|_{H^{-\theta}_x},
\end{align*}
and by Besov embeddings, see for example \cite[Sections 2.3.2-2.7]{Triebel83}:
\begin{align*}
B^{\beta-\theta}_{1,1} &\subset B^{\beta'}_{2,2} = H^{\beta'}_x,
\\
B^{-\theta}_{p_*,2} &\subset B^{\beta'}_{2,2} = H^{\beta'}_x.
\end{align*}
 
Therefore, in this case we have by H\"older inequality 
\begin{align*}
\int_0^t &\int_{\mathbb{R}^d}  \left|
 \mathbb{E}\left[\mathfrak{Re} \left( \overline{\hat{\rho}_s(\xi)} \langle b_s \rho_s , e_\xi \rangle \right)\right] \right| |\xi|\psi(\xi) d\xi ds
\\
&\lesssim
\int_0^t \int_{\mathbb{R}^d} 
 \mathbb{E}\left[ |{\hat{\rho}_s(\xi)}|^2 \right]^{1/2}
 \mathbb{E}\left[ \langle b_s \rho_s , e_\xi \rangle^2  \right]^{1/2} |\xi|\psi(\xi) d\xi ds
 \\
&\lesssim
\int_0^t \int_{\mathbb{R}^d} 
 \mathbb{E}\left[ \langle \xi \rangle^{-2\theta}|{\hat{\rho}_s(\xi)}|^2 \right]^{1/2}
 \mathbb{E}\left[ | \xi |^{2\beta'} \langle b_s \rho_s , e_\xi \rangle ^2 \right]^{1/2} 
 \langle \xi \rangle^{\theta-\beta'}
 |\xi| \psi(\xi) d\xi ds
 \\
&\lesssim
\sup_{\xi \in \R^d} \langle \xi \rangle^{\theta-\beta'-2\delta}  |\xi|    \int_0^t 
 \mathbb{E}\left[ \|{\rho}_s\|_{H^{-\theta}_x}^2 \right]
\|b_s \|_{W^{1,q}_x} 
 ds.
\end{align*}
The term $\|b_s \|_{W^{1,q}_x}$ can be brought outside the integral up to paying a factor $\| b \|_{L^\infty_t W^{1,q}_x} \lesssim 1$.
Next we claim that it is possible to take $\theta$ and $\lambda$ as above, in such a way that:
\begin{align} \label{eq:sup_xi}
\sup_{\xi \in \R^d} \langle \xi \rangle^{\theta-\beta'-2\delta}  |\xi| < \infty.
\end{align}
This requires that
\begin{align*}
0 
&\geq 
1+\theta-\beta'-2\delta
= 
1+ 2\theta + d/q - 1 - 2 + 2\alpha - 2\theta + 2 \lambda
= 
d/q  - 2 + 2\alpha + 2 \lambda,
\end{align*}
which is possible for some $\lambda > 0$ if and only if $q > \frac{d}{2(1-\alpha)}$.

For the term with the divergence, we argue in a similar fashion. The same bound can be obtained if we can control the product $\dvg (b_s) \rho_s$ in $H^{\beta'-1}_x$.
This can be done if we assume $\dvg \, b \in L^\infty_t H^{\vartheta}_x$ for some $\vartheta>0$. Indeed, in this case Sobolev embedding gives $H^{\vartheta}_x \subset L^{q_\vartheta}_x$ for $1/q_\vartheta = 1/2 - \vartheta/d$. Notice that $q_\vartheta>2$. Let $\theta$ as above such that $0<\theta \ll \vartheta$. Then paraproduct estimates yield 
\begin{align*}
\| \dvg(b) \varodot \rho \|_{B^{\vartheta-\theta}_{1,1}}
&\lesssim
\| \dvg(b) \|_{H^{\vartheta}_x} \| \rho \|_{H^{-\theta}_x},
\\
\| \dvg(b) \varolessthan \rho \|_{B^{-\theta}_{p_\vartheta,2}}
&\lesssim
\| \dvg(b) \|_{L^{q_\vartheta}_x} \| \rho \|_{H^{-\theta}_x},
\quad
1/p_\vartheta = 1/q_\vartheta + 1/2,
\\
\| \dvg(b) \varogreaterthan \rho \|_{B^{\vartheta-\theta}_{1,1}}
&\lesssim
\| \dvg(b) \|_{H^{\vartheta}_x} \| \rho \|_{H^{-\theta}_x}.
\end{align*}
and, since $\beta \leq 1$ by the assumption $q \leq 2$, by Besov embeddings:
\begin{align*}
B^{\vartheta-\theta}_{1,1} &\subset B^{\beta'+\vartheta-\beta}_{2,2} = H^{\beta'+\vartheta-\beta}_x \subset H^{\beta'-1}_x,
\\
B^{-\theta}_{p_\vartheta,2} &\subset B^{\beta'+\vartheta-\beta}_{2,2} = H^{\beta'+\vartheta-\beta}_x \subset H^{\beta'-1}_x.
\end{align*}
Putting all together,
\begin{align*}
\| \dvg (b_s) \rho_s \|_{H^{\beta'-1}_x}
\lesssim
\| \dvg \,b_s \|_{H^{\vartheta}_x} \| \rho_s\|_{H^{-\theta}_x}.
\end{align*}
\end{proof}

We are finally ready to prove pathwise uniqueness when $g=0$.
\begin{proof}[Proof of pathwise uniqueness, $g=0$]
By linearity of \eqref{eq:rho_general} it suffices to show that $\rho_t \equiv 0$ for every $t>0$ if $\rho_0 = 0$.

Let $\delta, \lambda$ be given by \autoref{lem:b} and consider \eqref{eq:a_psi} with $\psi_n$ defined as above.
Sending $n \to \infty$ we obtain by \eqref{eq:conv_1}, \autoref{cor:conv_2}, and \autoref{lem:b}:
\begin{align} \label{eq:inequality_1}
\mathbb{E}\| \rho_t \|_{H^{-\delta}_x}^2
-
\| \rho_0 \|_{H^{-\delta}_x}^2
\leq
\varepsilon^2
\int_0^t \int_{\R^d}F(\xi) a_s(\xi) d\xi ds 
+
C_{d,\alpha,\delta,\lambda,q,b} \int_0^t \mathbb{E}\|{\rho}_s\|^2_{H^{1-\alpha-\delta-\lambda}_x} \,ds,
\end{align}
for some finite constant $C_{d,\alpha,\delta,\lambda,q,b}$.
Next, let $R$ be as in \autoref{lem:b} and $C_F$ the implicit constant therein. By \cite[Proposition 4.2]{GaGrMa24+} there exists $K_{d,\alpha,\delta} \in (0,\infty)$ such that
\begin{align*}
\varepsilon^2\int_0^t \int_{\R^d}F(\xi) a_s(\xi) d\xi ds 
&=
\varepsilon^2\int_0^t \int_{|\xi|\leq R}F(\xi) \mathbb{E}|\hat{\rho}_s(\xi)|^2 d\xi ds
+
\varepsilon^2\int_0^t \int_{|\xi|>R} (F(\xi)-\dot{F}(\xi)) \mathbb{E}|\hat{\rho}_s(\xi)|^2 d\xi ds
\\
&\quad+
\varepsilon^2\int_0^t \int_{|\xi| > R} \dot{F}(\xi) \mathbb{E}|\hat{\rho}_s(\xi)|^2 d\xi ds
\\
&\leq
\varepsilon^2 C_{R,d,\alpha,\delta} \int_0^t \mathbb{E}\|{\rho}_s\|^2_{H^{-\delta}_x} ds
+ 
\varepsilon^2 C_F \int_0^t \int_{|\xi|>R} \langle\xi\rangle^{-2\delta\left(\frac{d+2\alpha}{d+2} \right)}
\mathbb{E}|\hat{\rho}_s(\xi)|^2 d\xi ds 
\\
&\quad-
\varepsilon^2
K_{d,\alpha,\delta}
\int_0^t \int_{|\xi|>R} |\xi|^{2-2\alpha-2\delta}
\mathbb{E}|\hat{\rho}_s(\xi)|^2 d\xi ds .
\end{align*}
Choosing $\lambda$ and $\theta=\alpha+\delta+\lambda-1$ in \autoref{lem:b} small enough, we have
\begin{align}
-2\delta \left( \frac{d+2\alpha}{d+2} \right) < 2(1-\alpha-\delta-\lambda).
\end{align}
In particular,
\begin{align*}
\varepsilon^2\int_0^t \int_{\R^d}F(\xi) a_s(\xi) d\xi ds
&\leq
\varepsilon^2 C_{R,d,\alpha,\delta} \int_0^t \mathbb{E}\|{\rho}_s\|^2_{H^{-\delta}_x} ds
+ 
\varepsilon^2 C_F \int_0^t 
\mathbb{E}\|{\rho}_s\|_{H^{1-\alpha-\delta-\lambda}_x}^2 ds 
\\
&\quad-
\varepsilon^2
K_{d,\alpha,\delta}
\int_0^t 
\mathbb{E}\|{\rho}_s\|_{H^{1-\alpha-\delta}_x}^2 ds.
\end{align*}
Plugging the previous line into \eqref{eq:inequality_1}, we obtain
\begin{align*}
\mathbb{E}\| \rho_t \|_{H^{-\delta}_x}^2
-
\| \rho_0 \|_{H^{-\delta}_x}^2
&\leq
C_{d,\alpha,\delta,\lambda,q,b,F} \int_0^t \mathbb{E}\|{\rho}_s\|^2_{H^{1-\alpha-\delta-\lambda}_x} \,ds
+
\varepsilon^2 C_{R,d,\alpha,\delta} \int_0^t \mathbb{E}\|{\rho}_s\|^2_{H^{-\delta}_x} ds
\\
&\quad-
\varepsilon^2
K_{d,\alpha,\delta}
\int_0^t 
\mathbb{E}\|{\rho}_s\|_{H^{1-\alpha-\delta}_x}^2 ds
\\
&\leq
C_{d,\alpha,\delta,\lambda,q,b,F,\varepsilon,R} \int_0^t \mathbb{E}\|{\rho}_s\|^2_{H^{-\delta}_x} ds
-
\varepsilon^2
\frac{K_{d,\alpha,\delta}}{2}
\int_0^t 
\mathbb{E}\|{\rho}_s\|_{H^{1-\alpha-\delta}_x}^2 ds,
\end{align*}
where in the last line we have rewritten the term involving the $H^{1-\alpha-\delta-\lambda}_x$ norm using interpolation between the Sobolev spaces $H_x^{-\delta}$ and $H^{1-\alpha-\delta}_x$, and we have applied Young inequality in order to reabsorb the term with $\mathbb{E}\|{\rho}_s\|_{H^{1-\alpha-\delta}_x}^2$.
Since the map $t \mapsto \mathbb{E}\| \rho_t \|_{H^{-\delta}_x}^2$ is continuous for $\delta>0$, we can rearrange the terms above and apply Gronwall inequality to get:
\begin{align} \label{eq:gronwall}
\mathbb{E}&\| \rho_t \|_{H^{-\delta}_x}^2
+
K
\int_0^t \mathbb{E}\| \rho_s \|_{H^{1-\alpha-\delta}_x}^2 ds
\leq
e^{CT}
\| \rho_0 \|_{H^{-\delta}_x}^2.
\end{align}
We conclude by recalling that $\rho_0 \equiv 0$.
\end{proof}

\subsection{Strong existence when $g=0$} \label{ssec:strong_existence}
We set up a vanishing viscosity scheme. To simplify our proofs we approximate also the initial condition $\rho_0 = \rho_0^\kappa$ in such a way that $\rho_0^\kappa \in C^\infty_c(\R^d)$ and $\rho_0^\kappa \to \rho_0$ in $L^2_x \cap L^p_x$ as $\kappa \downarrow 0$.

Next we are going to show that solutions associated with $\varepsilon,\kappa\in(0,1)$ enjoy uniform-in-$(\varepsilon,\kappa)$ estimates in the spaces $L^\infty_\omega L^\infty_t (L^2_x \cap L^p_x)$ and $L^2_\omega L^2_t H^{1-\alpha-\delta}_x$.

\emph{Estimate in $L^\infty_\omega L^\infty_t (L^2_x \cap L^p_x)$}. 
Let us consider first the case $p \in [2,\infty)$.  
Testing $\rho$ against $p\rho|\rho|^{p-2}$ and using $\dvg g = 0$ and 
\begin{align*}
\langle (b+g) \cdot \nabla \rho , \rho |\rho|^{p-2}\rangle  &=
\langle \dvg((b+g) \rho ) , \rho |\rho|^{p-2}\rangle- \langle \dvg \,b, |\rho|^p \rangle
\\
&=
-(p-1) \langle (b+g) \rho, |\rho|^{p-2} \nabla \rho \rangle - \langle \dvg \,b, |\rho|^p \rangle,
\end{align*}
 we get the following inequality $\PP$-almost surely:
\begin{align*}
\|\rho_t\|_{L^p_x}^p
&\leq
\| \rho_0^\kappa \|_{L^p_x}^p
+
\int_0^t \| \dvg \,b_s\|_{L^\infty_x} \sup_{r \in [0,s]}\| \rho_r \|_{L^p_x}^p ds
.
\end{align*}
Taking the supremum over times $r \in [0,t]$ at the left-hand side above, by Gronwall Lemma we get the $\PP$-almost sure bound
\begin{align} \label{eq:bound_rho_Linfty_Lp}
\sup_{r \in [0,t]}
\| \rho_r \|_{L^p_x}^p
\leq
e^{\| \dvg \,b \|_{L^1_tL^\infty_x}}
\| \rho_0^\kappa\|_{L^p_x}^p.
\end{align}
The case $p \in [1,2)$ can be handled similarly, approximating the function $| \cdot |^p$ with an increasing sequence of smooth functions. Finally, the remaining case $p=\infty$ can be handled using Markov inequality and the bounds just obtained for the $L^p_x$ norms, similarly to what done in the proof of \cite[Proposition 2.7]{DrGaPa25+}.

\emph{Estimate in $L^n_\omega C^\gamma_t H^{-\sigma}_x$}.
Due to the regularity of $\rho$, $b$, and $g$, it immediately follows that for each $\varphi\in H^{\sigma}_x$,\ $\sigma>d/2+1$:
\begin{align*}
    \langle g\cdot\nabla\rho,\varphi\rangle &= \langle g\cdot\nabla\varphi,\rho\rangle  \leq \lVert \rho\rVert_{L^2_x} \lVert g\rVert_{H^{d/2+\alpha}_x}\norm{\varphi}_{H^{\sigma}_x},\\ 
    \langle b\cdot\nabla\rho,\varphi\rangle &= \langle b\cdot\nabla\varphi,\rho\rangle+\langle\operatorname{div}( b)\varphi,\rho\rangle\\ & \leq \lVert \rho\rVert_{L^2_x}(\lVert b\rVert_{W^{1,q}_x}+\lVert \operatorname{div}b\rVert_{H^{\vartheta}_x})\norm{\varphi}_{H^{\sigma}_x}.
\end{align*}
Therefore $(g+b)\cdot\nabla \rho\in L^{\infty}_tH^{-\sigma}_x$ and 
\begin{align}\label{gamma_holder_time}
    \norm{(g+b)\cdot\nabla \rho}_{ L^{\infty}_tH^{-\sigma}_x}\leq \lVert \rho\rVert_{L^{\infty}_t L^2_x}(\lVert g\rVert_{L^2_t H^{d/2+\alpha}_x}+\lVert b\rVert_{L^{\infty}_t W^{1,q}_x}+\lVert \operatorname{div}\,b\rVert_{L^{\infty}_t  H^{\vartheta}_x}).
\end{align}
Moreover, by \cite[Lemma 2.3]{GaGrMa24+}, the stochastic integral
$\int_0^\cdot dW_s\cdot\nabla\rho_s$ takes values in $C_tH^{-1}_x$ and it holds
\begin{align*}
\mathbb{E}\left[\sup_{t\in [0,T]}\norm{\int_0^t dW_s\cdot\nabla\rho_s}_{H^{-1}_x}^2\right]\lesssim \mathbb{E}\left[\int_0^T \norm{\rho_s}_{L^2_x}^2 ds\right]\lesssim 1.
\end{align*}
Easily $\Delta \rho_s\in L^{\infty}_tH^{-\sigma}_x $. As a consequence, by standard density argument, there exists a zero probability set $\mathcal{N}$ such that on its complementary, $\rho$ satisfies \eqref{eq:notion_solution} for each $t\in [0,T],\ \varphi\in H^{\sigma}_x $. Therefore $\rho$ satisfies
\begin{align}\label{weak_formulation_Hsigma}
\rho_t=\rho_0+\int_0^t(g_s+b_s)\cdot\nabla \rho_s ds+\varepsilon\int_0^t dW_s\cdot\nabla\rho_s+(1+\kappa)\varepsilon^2\int_0^t \Delta \rho_s ds,
\end{align}
seen as an equality in $H^{-\sigma}_x$. The latter and the continuity in $H^{-\sigma}_x$ of all the terms appearing in the right hand side imply that $\rho\in C_tH^{-\sigma}_x$. Let us now fix $n\geq 2$ and consider $0\leq s\leq t\leq T$. Equations \eqref{weak_formulation_Hsigma} and H\"older inequality combined with \eqref{gamma_holder_time}, Burkholder-Davis-Gundy inequality, \eqref{eq:bound_rho_Linfty_Lp} and \cite[Lemma 2.3]{GaGrMa24+} imply 
\begin{align}\label{time_difference}
    \mathbb{E}\left[\norm{\rho_t-\rho_s}_{H^{-\sigma}}^n\right]& \lesssim (1+\norm{g}^n_{L^{\infty}_\omega L^2_t H^{d/2+\alpha}}+\lVert b\rVert_{W^{1,q}_x}^n+\lVert \operatorname{div}\,b\rVert_{H^{\vartheta}_x}^n)\norm{\rho_0^\kappa}_{L^2_x}^n\lvert t-s\rvert^{\frac{n}{2}}.
\end{align}
As a consequence, for each $n\geq 2$ and $\gamma<1/2$, $\rho\in L^{n}_{\omega}W^{\gamma,n}_t H^{-\sigma}_x$ uniformly in $\varepsilon \in (0,1)$ and $\kappa \in [0,1)$. By arbitrariness of $n$ and Sobolev embdedding:
\begin{align} \label{eq:bound_rho_holder}
\mathbb{E}[\| \rho\|_{C^\gamma_tH^{-\sigma}_x}^n] 
\lesssim \| \rho_0^\kappa \|_{L^2_x}^n
< \infty ,
\end{align}
uniformly in $\varepsilon\in(0,1)$ and $\kappa \in [0,1)$, or otherwise said $\rho \in L^n_\omega C^\gamma_t H^{-\sigma}_x$.

\emph{Estimate in $L^2_\omega L^2_t H^{1-\alpha-\delta}_x$}.
Let us recall \eqref{eq:a_psi} from \autoref{ssec:pathwise_uniqueness}. In the presence of diffusivity, it takes the form 
\begin{align} \label{eq:a_psi_bis}
 \int_{\mathbb{R}^d}  &a_t(\xi) \psi(\xi) d\xi 
 - 
 \int_{\mathbb{R}^d} a_0(\xi)\psi(\xi) d\xi
 +
 2\kappa \varepsilon^2 \int_0^t \int_{\mathbb{R}^d} |\xi|^2 a_s(\xi) \psi(\xi) d\xi ds 
\\ \nonumber
&= 
\varepsilon^2 (2\pi)^{-d/2} \int_0^t \int_{\mathbb{R}^d}\int_{\mathbb{R}^d}  \frac{|P^\perp_{\xi-\eta} \xi|^2 a_s(\xi)}{\langle \xi - \eta \rangle^{d+2\alpha}} (\psi(\eta) - \psi(\xi)) d\xi d\eta ds
\\ \nonumber
   &\quad
 2\int_0^t \int_{\mathbb{R}^d} 
 \mathbb{E}\left[\mathfrak{Re} \left(\overline{\hat{\rho}_s(\xi)} \langle b_s \cdot \nabla \rho_s , e_\xi \rangle \right)\right] \psi(\xi) d\xi ds,
\end{align}
and is valid for every compactly supported
smooth test function $\psi$.
As before, we want to put $\psi(\xi) = \langle \xi \rangle^{-2\delta}$ in the previous line, but now we allow for every $\delta \in (0,1)$ thanks to the extra regularity of the diffusive approximations. Introducing a sequence $\psi_n \uparrow \psi$ as above, we have the following:
\begin{lem} \label{lem:b_bis}
For every $\delta \in (0,\alpha)$ and $0<\lambda \ll 1$, there exists an implicit constant depending on $d,\alpha,\delta,\lambda,q,r$ and the norm $\| b \|_{L^\infty_t W^{1,q}_x}$ such that
\begin{align} \label{eq:ERe(b+g)_bis}
\sup_{n \geq 1} \int_0^T \int_{\mathbb{R}^d} 
\left| \mathbb{E}  
\left[\mathfrak{Re} \left(\overline{\hat{\rho}_s(\xi)} \langle b_s \cdot \nabla  \rho_s , e_\xi \rangle \right)\right] \right| \psi_n(\xi) d\xi ds
& \lesssim
\| \rho \|_{L^2_\omega L^2_t H^{1-\alpha-\delta-\lambda}_x}^2.
\end{align}
\end{lem}
\begin{proof}
The proof is similar to the proof of \autoref{lem:b}. Since $\delta$ is now small, by \cite[Lemma 2.55]{BaChDa2011} we can make sense of the product $b_s \cdot \nabla \rho_s$ with
\begin{align*}
\| b_s \cdot \nabla \rho_s \|_{\dot{H}^{\beta'}_x} 
&\lesssim
\| b_s \|_{\dot{H}^\beta_x} \| \rho_s \|_{\dot{H}^{1-\alpha-\delta-\lambda}_x}
\lesssim
\| b_s \|_{H^\beta_x} \| \rho_s \|_{H^{1-\alpha-\delta-\lambda}_x},
\\
\beta' &= \beta+1-\alpha-\delta-\lambda-d/2,
\\
\beta &< 1 + d/2 - d/q,
\\
0 &< \beta + 1-\alpha-\delta-\lambda.
\end{align*}
Thus
\begin{align*}
\sup_{n \geq 1}& \int_0^T \int_{\mathbb{R}^d} 
\left| \mathbb{E} 
\left[\mathfrak{Re} \left(\overline{\hat{\rho}_s(\xi)} \langle b_s \cdot \nabla  \rho_s , e_\xi \rangle \right)\right] \right| \psi_n(\xi) d\xi ds
\\
&\lesssim
\int_0^T \int_{\mathbb{R}^d} 
 \mathbb{E}\left[ |{\hat{\rho}_s(\xi)}|^2 \right]^{1/2}
 \mathbb{E}\left[ \langle b_s \cdot \nabla  \rho_s  , e_\xi \rangle^2  \right]^{1/2} \psi(\xi) d\xi ds
 \\
&\lesssim
\int_0^T \int_{\mathbb{R}^d} 
 \mathbb{E}\left[ \langle \xi \rangle^{2-2\alpha-2\delta-2\lambda}|{\hat{\rho}_s(\xi)}|^2 \right]^{1/2}
 \mathbb{E}\left[ | \xi |^{2\beta'} \langle b_s \cdot \nabla  \rho_s , e_\xi \rangle ^2 \right]^{1/2} \langle \xi \rangle^{-1+\alpha+\delta+\lambda}|\xi|^{-\beta'} \psi(\xi) d\xi ds
 \\
&\lesssim
\langle \xi \rangle^{-1+\alpha+\delta+\lambda}|\xi|^{-\beta'} \psi(\xi)  
\int_0^T 
 \mathbb{E}\left[ \|{\rho}_s\|_{H^{1-\alpha-\delta-\lambda}_x}^2 \right]^{1/2}
 \mathbb{E}\left[ \| b_s \cdot \nabla \rho_s \|_{\dot{H}^{\beta'}_x}^2 \right]^{1/2} 
 ds
  \\
&\lesssim
\| b_s \|_{L^\infty_t W^{1,q}_x}
\int_0^T \mathbb{E}\left[ \|{\rho}_s\|_{H^{1-\alpha-\delta-\lambda}_x}^2 \right]
 ds,
\end{align*}
where we have chosen $\beta$ and $0<\lambda \ll 1$ such that
\begin{align*}
\sup_{\xi \in \R^d} \langle \xi \rangle^{-1+\alpha+\delta+\lambda}|\xi|^{-\beta'} \psi(\xi)  < \infty.
\end{align*}
We point out that such $\beta,\lambda$ can always be found since $q>\frac{d}{2(1-\alpha)}$. Indeed,
\begin{align*}
 -1 + \alpha - \delta +\lambda - \beta' \leq 0
 &\iff
 \beta  \geq d/2 -2+ 2\alpha+2\lambda,
 \\
 - \beta'  \geq 0 &\iff \beta \leq d/2 - 1+\alpha+\delta+\lambda,
\end{align*}
which can be satisfied for some choice of $\lambda$ and $\beta$ because
\begin{align*}
d/2 - 2 + 2\alpha + 2\lambda &< 1+ d/2 - d/q, \quad
\mbox{ since } q>\frac{d}{2(1-\alpha)} \mbox{ and } \lambda \ll 1,
\\
d/2 - 1+\alpha+\delta+\lambda &>  - 1+\alpha+\delta+\lambda.
\end{align*}
\end{proof}
Next, we evaluate \eqref{eq:a_psi_bis} with $\psi = \psi_n$ and take the limit $n \to \infty$. Arguing as in \cite[Proposition 3.3]{GaGrMa24+} and taking advantage of the extra regularity of solutions coming from positive diffusivity and $\rho_0^\kappa \in C^\infty_c(\R^d)$, we obtain the inequality (neglecting positive terms on the left-hand-side)
\begin{align*}
0
&\leq
\| \rho_0^\kappa \|_{H^{-\delta}_x}^2
+
\varepsilon^2 \int_0^T \int_{\R^d} F(\xi) a_s(\xi) d\xi ds
+
C_{d,\alpha,\delta,\lambda,q,b}
\int_0^T \mathbb{E} \|\rho_s \|_{H^{1-\alpha-\delta-\lambda}_x}^2 ds
\\
&\leq
\| \rho_0^\kappa \|_{H^{-\delta}_x}^2
+
C_{d,\alpha,\delta,\lambda,q,b,F,\varepsilon,R,T} \| \rho\|_{L^\infty_\omega L^\infty_t L^2_x}
-
\varepsilon^2 \frac{K_{d,\alpha,\delta}}{2} \int_0^T \mathbb{E} \|\rho_s \|_{H^{1-\alpha-\delta}_x}^2 ds.
\end{align*}
We conclude recalling the estimate \eqref{eq:bound_rho_Linfty_Lp} on $\rho$, i.e.
\begin{align} \label{eq:bound_rho_sobolev}
\varepsilon^2\| \rho \|_{L^2_\omega L^2_t H^{1-\alpha-\delta}_x} 
 \lesssim \| \rho_0^\kappa\|_{L^2_x}^2
 < \infty.
\end{align}
Notice that \eqref{eq:bound_rho_sobolev} is uniform in $\kappa \in (0,1)$ but degenerates as $\varepsilon \to 0$.

\begin{proof}[Proof of strong existence, $g=0$]
The case $p=1$ is slightly more complicated and must be treated separately, so let $p \neq 1$ for the time being.

Let us recall the estimates \eqref{eq:bound_rho_Linfty_Lp} and \eqref{eq:bound_rho_sobolev} proved above.
For fixed $\varepsilon \in (0,1)$, both estimates are uniform in $\kappa$ since we have assumed $\rho_0^\kappa \to \rho_0$ in $L^2_x \cap L^p_x$ as $\kappa \downarrow 0$.
With these estimates in hand, for every fixed $\varepsilon \in (0,1)$ we can find a subsequence $\kappa_n \downarrow 0$ such that the unique solution to \eqref{eq:rho_general} associated with $\kappa=\kappa_n$ converges (e.g. weakly-$\ast$ in $L^\infty_\omega L^\infty_t (L^2_x \cap L^p_x)$ and weakly in $L^2_\omega L^2_t H^{1-\alpha-\delta}_x$) towards a process $\tilde{\rho}$ satisfying the same bounds.
Furthermore, the convergence above allows to take the limit in \eqref{eq:notion_solution} when $\kappa_n \downarrow 0$.
Up to taking a modification $\rho$ of the process $\tilde{\rho}$, one can assume that $\rho$ is a weakly continuous, $\{\mathcal{F}_t\}_{t \geq 0}$-progressively measurable process satisfying \eqref{eq:notion_solution} with $\kappa=0$, of class $L^\infty_\omega L^\infty_t (L^2_x \cap L^p_x)$. Namely, strong existence of solutions holds. By \eqref{second_compact_embedding}, $\rho$ has weakly continuous trajectories in $L^p_x$ $\PP$-almost surely.

When $p = 1$, by interpolation we have $\rho_0 \in L^p_x$ for every $p \in (1,2]$. Therefore, one can repeat the same argument as above to obtain uniform bounds for $\rho$ in $L^\infty_\omega L^\infty_t L^{1+1/n}_x$ for every $n \in \N$. We conclude by Fatou's Lemma that $\|\rho\|_{L^\infty_\omega L^\infty_t L^1_x} \leq \liminf_{n \to \infty} \|\rho\|_{L^\infty_\omega L^\infty_t L^{1+1/n}_x}^{1+1/n} < \infty$. 

It only remains to show the H\"older regularity the trajectories in $H^{-\sigma}$, but this immediately descends from \eqref{eq:bound_rho_holder}, completing the proof.
\end{proof}

\subsection{Proof of \autoref{prop_well_posed}}
\begin{proof}
Let us consider a filtered probability space $(\Omega,\mathcal{F},\{\mathcal{F}_t\}_{t\geq 0},\mathbb{P})$, $W$ being a Wiener process with covariance operator $\mathcal{Q}$ adapted to $\{\mathcal{F}_t\}_{t \geq 0}$ and $g$ as in the statement. Let us denote by \begin{align*}
    \gamma_g:=\exp{\left(-\frac{1}{\sqrt{\varepsilon}}\int_0^T \langle g_s,dW_s\rangle_{\mathcal{H}_0}-\frac{1}{2\varepsilon}\int_0^T \norm{g_s}_{\mathcal{H}_0}^2 ds\right)}.
\end{align*}
By Girsanov theorem, see for example \cite[Theorem 10.14]{DaPZab}, there exists a probability measure $\mathbb{P}_g$ on $(\O,\mathcal{F},\PP)$ such that 
\begin{align*}
    \frac{d \mathbb{P}_g}{d\mathbb{P}}=\gamma_g
\end{align*}
and $\hat{W}_t:=W_t+\frac{1}{\sqrt{\epsilon}}\int_0^t g_s ds $ is a Wiener process with covariance operator $\mathcal{Q}$ on the auxiliary filtered probability space $(\Omega,\mathcal{F},\{\mathcal{F}_t\}_{t\geq 0},\mathbb{P}_g)$. 
Moreover, the two probability measures $\mathbb{P}$ and $\mathbb{P}_g$ are equivalent. 
Lastly, for each choice of $r<+\infty$, we have $\gamma_g\in L^r_{\omega}$ with respect to the probability $\mathbb{P}$, and $\gamma_g^{-1}\in L^r_{\omega}$ with respect to the probability $\mathbb{P}_g$. By our previous result in case of $g=0$, on the auxiliary probability space $(\Omega,\mathcal{F},\{\mathcal{F}_t\}_{t\geq 0},\mathbb{P}_g)$ there exists a unique $\rho\in L^{\infty}_{\omega}L^{\infty}_tL^2_x \cap L^{2}_{\omega}L^{2}_tH^{1-\alpha-\delta}_x  \cap L^n_\omega C^\gamma_t H^{-\sigma}_x $, with trajectories in $C_tH^{-s}_x$ $\PP_g$-almost surely, solving \eqref{eq:stoch_trans_Ito} with initial condition $\rho_0$ and noise $\hat{W}.$ Therefore, in our original probability space, $\rho$ solves \eqref{eq:rho_general} with initial condition $\rho_0$, advecting velocity field $b+g$ and Wiener process ${W}.$ 
Moreover, due to the integrability properties of $\gamma_g^{-1}$ with respect to $\mathbb{P}_g$, it holds 
\begin{align*}
\rho\in L^{\infty}_{\omega}L^{\infty}_tL^2_x\cap L^{\frac{2(1-\alpha-2\delta)}{1-\alpha-\delta}}_{\omega}L^{2}_tH^{1-\alpha-\delta}_x 
\cap L^{n/2}_\omega C^\gamma_t H^{-\sigma}_x,
\quad
\forall \delta<\frac{1}{3}\wedge\alpha,
\end{align*}
with respect to the original probability measure $\mathbb{P}.$ 
By interpolation $\rho\in L^{\infty}_{\omega}L^{\infty}_tL^2_x\cap L^{2}_{\omega}L^{2}_tH^{1-\alpha-2\delta}_x \cap L^{n/2}_\omega C^\gamma_t H^{-\sigma}_x$  for each $\delta<\frac{1}{3}\wedge \alpha$ and $n \geq 2$. Due to the arbitrariness of $\delta$ and $n$, we have strong existence of solutions of \eqref{eq:rho_general} with the regularity stated in \autoref{prop_well_posed}. Concerning uniqueness, we argue similarly. By linearity it is enough to show uniqueness assuming $\rho_0=0$. Let $\rho \in L^{\infty}_{\omega}L^{\infty}_t L^{2}_x$ with trajectories in $C_t H^{-s}_x$ $\PP$-almost surely solve \eqref{eq:rho_general} on $(\Omega,\mathcal{F},\{\mathcal{F}_t\}_{t\geq 0},\mathbb{P})$ with null initial condition, advecting velocity field $b+g$ and Wiener process ${W}$. Therefore $\rho$ solves \eqref{eq:stoch_trans_Ito} on the auxiliary probability space $(\Omega,\mathcal{F},\{\mathcal{F}_t\}_{t\geq 0},\mathbb{P}_g)$ and has regularity $L^{\infty}_{\omega}L^{\infty}_t L^{2}_x$ and trajectories in $C_t H^{-s}_x$ $\PP_g$-almost surely. By the uniqueness result for the case $g=0$ it follows that $\rho\equiv 0$, $\mathbb{P}_g$-almost surely. Since $\PP_g$ and $\PP$ are equivalent, $\rho\equiv 0$ almost surely with respect to $\mathbb{P}$ and the result follows.
\end{proof}

\section{Dissipation measure}
\begin{lem} \label{lem:beta_rho}
Under the assumptions of \autoref{prop_well_posed}, assume in addition $\dvg\, b = 0$. Then for every $\rho_0 \in L^2_x$ and $\varepsilon \in (0,1)$ there exists a non-negative random measure $\mathcal{D}^\varepsilon : \Omega \to \mathcal{M}_+([0,T] \times \R^d)$ such that the unique solution of \eqref{eq:rho_general} with $\kappa=0$ satisfies the local energy balance
\begin{align} \label{eq:local_balance_appendix}
\langle|\rho^\varepsilon_t|^2 , \phi_t \rangle
&=
\langle|\rho_0|^2 , \phi_0 \rangle
+ 
\int_0^t  \langle(b_s+g_s) |\rho^\varepsilon_s|^2 , \nabla \phi_s \rangle ds 
+ 
\int_0^t  \langle |\rho^\varepsilon_s|^2 ,\partial_s \phi_s \rangle ds
\\
&\quad+  \nonumber
\sum_{k \in \N}\varepsilon \int_0^t \langle \sigma_k  |\rho^\varepsilon_s|^2 , \nabla \phi_s \rangle dW^k_s 
+
\varepsilon^2 \int_0^t \langle|\rho^\varepsilon_s|^2 ,\Delta \phi_s \rangle ds
-
\langle \langle d\mathcal{D}^{\varepsilon} ,\mathbf{1}_{[0,t]}\phi \rangle \rangle,
\end{align}
$\PP$-almost surely for almost every $t \in [0,T]$ and $\phi \in C^\infty_c([0,T] \times \R^d)$.
Moreover, one can choose $\mathcal{D}^\varepsilon$ is a unique way such that the following property holds: if $b^*$, $g^*$ satisfy the assumptions of \autoref{prop_well_posed} on the time interval $[0,T+1]$, with $b^*_t = b_t$, $g^*_t = g_t$ for almost every $t \in [0,T]$, and $\mathcal{D}^{\varepsilon,*} : \O \to \mathcal{M}_+([0,T+1]\times \R^d)$ is such that \eqref{eq:local_balance_appendix} holds $\PP$-almost surely for almost every $t \in [0,T+1]$ and $\phi \in C^\infty_c([0,T+1] \times \R^d)$, then it holds $\mathcal{D}^{\varepsilon,*}|_{[0,T] \times \R^d} = \mathcal{D}^\varepsilon$ $\PP$-almost surely.
\end{lem}
\begin{rmk}
Since equation \eqref{eq:local_balance_appendix} only holds for almost every $t \in [0,T]$, it can give no information about how much of the $L^2_x$ norm of $\rho^\varepsilon$ is dissipated exactly at time $t=T$. On the other hand, the restriction property $\mathcal{D}^{\varepsilon,*}|_{[0,T] \times \R^d} = \mathcal{D}^\varepsilon$ $\PP$-almost surely guarantees that $\mathcal{D}^\varepsilon$ takes into account any dissipation happening at time $t=T$.  
\end{rmk}
\begin{proof}
Let us work on the time interval $[0,T+1]$ first.
By similar arguments as those in \cite[Section 2.4]{DrGaPa25+} we have for every $\varepsilon \in (0,1)$ and $\kappa \in (0,1)$
\begin{align*}
\langle|\rho^{\varepsilon,\kappa}_t|^2 , \phi_t \rangle
&=
\langle|\rho_0|^2 , \phi_0 \rangle
+ 
\int_0^t  \langle(b^*_s+g^*_s) |\rho^{\varepsilon,\kappa}_s|^2 , \nabla \phi_s \rangle ds 
+ 
\int_0^t  \langle |\rho^{\varepsilon,\kappa}_s|^2 ,\partial_s \phi_s \rangle ds
\\
&\quad+  \nonumber
\sum_{k \in \N}\varepsilon \int_0^t \langle \sigma_k  |\rho^{\varepsilon,\kappa}_s|^2 , \nabla \phi_s \rangle dW^k_s 
+
\varepsilon^2 \int_0^t \langle|\rho^{\varepsilon,\kappa}_s|^2 ,\Delta \phi_s \rangle ds
-
2 \varepsilon^2\kappa  \int_0^t \langle |\nabla \rho^{\varepsilon,\kappa}_s|^2 ,\phi_s \rangle ds.
\end{align*}
By \autoref{prop_well_posed}, the family $\{{\rho^{\varepsilon,\kappa}}\}_{\kappa \in (0,1)}$ is uniformly bounded in $L^\infty_t L^2_x \cap L^2_t H^{1-\alpha-\delta}_x \cap C^\gamma_t H^{-\sigma}_x$ with arbitrary large probability, hence it is tight in $\mathcal{E}_* \cap  L^2_t \tilde{H}^{1-\alpha-\delta-\lambda}_x$, for any $0 < \lambda \ll 1$.  
Furthermore, the family of the laws of the processes $\{ \varepsilon^2\kappa |\nabla \rho^{\varepsilon,\kappa}|^2\}_{ \kappa \in (0,1)}$ is tight as measures on $[0,T+1] \times \R^d$ by standard energy estimates applied to the processes $\{\rho^{\varepsilon,\kappa} \chi^R\}_{\kappa \in (0,1), R \geq 1}$, where $\chi^R$ is a smooth spatial cutoff vanishing on the set $\{|x| \leq R\}$.
Up to extracting a subsequence $\kappa_n \downarrow 0$ and passing to an auxiliary probability space $(\tilde{\Omega},\tilde{F}, \tilde{\PP})$, we can suppose that $\rho^{\varepsilon,\kappa_n} \to \rho^\varepsilon$ as $\kappa_n \downarrow 0$ $\tilde{\PP}$-almost surely in $\mathcal{E} \cap L^2_t \tilde{H}^{1-\alpha-\delta-\lambda}_x$ and for almost every $t,x$, whereas $
\varepsilon^2\kappa_n |\nabla \rho^{\varepsilon,\kappa_n}_t|^2 dt \to \tilde{\mathcal{D}}^{\varepsilon,*}$ $\tilde{\PP}$-almost surely almost surely as measures in $\mathcal{M}_+([0,T+1] \times \R^d)$. 
Moreover, $\tilde{\PP}$-almost surely we have
\begin{align} \label{eq:properties_D}
\tilde{\mathcal{D}}^{\varepsilon,*} \geq 0,
\qquad
\| \tilde{\mathcal{D}}^{\varepsilon,*}\|_{TV} \leq \| \rho_0 \|_{L^2_x}^2,
\end{align}
since the same holds for $\varepsilon^2\kappa_n |\nabla \rho^{\varepsilon,\kappa_n}_t|^2 dt$. 
Passing to the limit in the equation above, we get \eqref{eq:local_balance_appendix} $\tilde{\PP}$-almost surely for almost every $t \in [0,T+1]$ with measure $\tilde{\mathcal{D}}^{\varepsilon,*}$ instead of $\mathcal{D}^{\varepsilon}$.
Using \eqref{eq:local_balance_appendix} and \eqref{eq:properties_D}, we deduce $\tilde{\PP}$-almost surely for almost every $t \in [0,T+1]$ and $\phi \in C^\infty_c([0,T+1] \times \R^d)$, $\phi \geq 0$:
\begin{align} \label{eq:Riesz}
0
&\leq
-\langle|\rho^\varepsilon_t|^2 , \phi_t \rangle
+
\langle|\rho^\varepsilon_0|^2 , \phi_0 \rangle
+ 
\int_0^t  \langle(b^*_s+g^*_s) |\rho^\varepsilon_s|^2 , \nabla \phi_s \rangle ds
+ 
\int_0^t  \langle |\rho^\varepsilon_s|^2 ,\partial_s \phi_s \rangle ds
\\
&\quad+  \nonumber
\sum_{k \in \N}\varepsilon \int_0^t \langle \sigma_k  |\rho^\varepsilon_s|^2 , \nabla \phi_s \rangle dW^k_s 
+
\varepsilon^2 \int_0^t \langle|\rho^\varepsilon_s|^2 ,\Delta \phi_s \rangle ds
\leq
\| \rho_0 \|_{L^2_x}^2
\sup |\phi|
,
\end{align}
and the same relation transfers back on the original probability space $(\Omega,\mathcal{F},\PP)$.
For a given countable dense $\{\phi^n\} \subset C^\infty_c([0,T+1] \times \R^d)$ we can find a further full-measure subset of Lebesgue points $t \in [0,T+1]$ of all the maps $t \mapsto \langle|\rho^\varepsilon_t|^2, \phi^n_t \rangle$. By standard density arguments, the same full-measure set is of Lebesgue points for all the maps $t \mapsto \langle|\rho^\varepsilon_t|^2, \phi_t \rangle$, with arbitrary $\phi \in C^\infty_c([0,T+1] \times \R^d)$.
By Riesz Representation Theorem and the previous observation, \eqref{eq:Riesz} defines uniquely a random measure $\mathcal{D}^{\varepsilon,*,T+1}$ in $\mathcal{M}_+([0,T+1)\times \R^d) \subset \mathcal{M}_+([0,T+1] \times \R^d)$ such that 
\begin{align} \label{eq:D_T+1}
\langle \langle \mathcal{D}^{\varepsilon,*,T+1} &,\mathbf{1}_{[0,t]}\phi \rangle \rangle
=
-\langle|\rho^\varepsilon_t|^2 , \phi_t \rangle
+
\langle|\rho^\varepsilon_0|^2 , \phi_0 \rangle
+ 
\int_0^t  \langle(b^*_s+g^*_s) |\rho^\varepsilon_s|^2 , \nabla \phi_s \rangle ds
\\
&+  \nonumber 
\int_0^t  \langle |\rho^\varepsilon_s|^2 ,\partial_s \phi_s \rangle ds
+
\sum_{k \in \N}\varepsilon \int_0^t \langle \sigma_k  |\rho^\varepsilon_s|^2 , \nabla \phi_s \rangle dW^k_s 
+
\varepsilon^2 \int_0^t \langle|\rho^\varepsilon_s|^2 ,\Delta \phi_s \rangle ds,
\end{align}
$\PP$-almost surely for almost every $t \in [0,T+1]$. 
Even if not needed in our work, we notice that $\mathcal{D}^{\varepsilon,*,T+1}$ and $\tilde{\mathcal{D}}^{\varepsilon,*} = \lim_{n \uparrow\infty}
\varepsilon^2\kappa_n |\nabla \rho^{\varepsilon,\kappa_n}_t|^2 dt $ have the same law.

In order to conclude the proof, we define $\mathcal{D}^{\varepsilon} := \mathcal{D}^{\varepsilon,*,T+1}|_{[0,T] \times \R^d}$, which clearly satisfies \eqref{eq:local_balance_appendix} by \eqref{eq:D_T+1}, the relation $b^*_t = b_t$, $g^*_t = g_t$ for almost every $t \in [0,T]$, and uniqueness of solutions to \eqref{eq:rho_general}. 
It only remains to prove that, if $\mathcal{D}^{\varepsilon,*} : \O \to \mathcal{M}_+([0,T+1]\times \R^d)$ is such that \eqref{eq:local_balance_appendix} holds $\PP$-almost surely for almost every $t \in [0,T+1]$ and $\phi \in C^\infty_c([0,T+1] \times \R^d)$, then it holds $\mathcal{D}^{\varepsilon,*}|_{[0,T] \times \R^d} = \mathcal{D}^\varepsilon$ $\PP$-almost surely.
Comparing \eqref{eq:local_balance_appendix} for $\mathcal{D}^{\varepsilon,*}$ and \eqref{eq:D_T+1} for $\mathcal{D}^{\varepsilon,*,T+1}$, we deduce that for $\PP$-almost surely for almost every $t \in [0,T+1]$ and $\phi \in C^\infty_c([0,T+1] \times \R^d)$
\begin{align*}
\langle \langle \mathcal{D}^{\varepsilon,*} ,\mathbf{1}_{[0,t]}\phi \rangle \rangle
=
\langle \langle \mathcal{D}^{\varepsilon,*,T+1} ,\mathbf{1}_{[0,t]}\phi \rangle \rangle.
\end{align*}
Therefore, letting $t_n \downarrow T$ with $\{t_n\}_{n \in \N}$ in a full-measure set such that the previous identity holds, we deduce that the restrictions $\mathcal{D}^{\varepsilon,*}|_{[0,T] \times \R^d}$ and $\mathcal{D}^{\varepsilon,*,T+1}|_{[0,T] \times \R^d} = \mathcal{D}^{\varepsilon}$ coincide $\PP$-almost surely.
This concludes the proof. 
\end{proof} 

\bibliography{biblio}{}

\begin{thebibliography}{DPFRV16}

\bibitem[ACM19]{albertiCrippaMazzuccato}
Giovanni Alberti, Gianluca Crippa, and {Anna L.} Mazzucato.
\newblock Loss of regularity for the continuity equation with non-lipschitz
  velocity field.
\newblock {\em Annals of PDE}, 5(1), June 2019.
\newblock Publisher Copyright: {\textcopyright} 2019, Springer Nature
  Switzerland AG.

\bibitem[AF09]{AtFl09}
Stefano Attanasio and Franco Flandoli.
\newblock Zero-noise solutions of linear transport equations without
  uniqueness: an example.
\newblock {\em Comptes Rendus Mathematique}, 347(13):753--756, 2009.

\bibitem[Agr24]{agresti2024global}
Antonio Agresti.
\newblock Global smooth solutions by transport noise of 3{D} {N}avier-{S}tokes
  equations with small hyperviscosity.
\newblock {\em arXiv preprint arXiv:2406.09267}, 2024.

\bibitem[Amb04]{Am04}
Luigi Ambrosio.
\newblock Transport equation and {C}auchy problem for {BV} vector fields.
\newblock {\em Inventiones mathematicae}, 158(2):227--260, 2004.

\bibitem[BB82]{BaBa81}
Roberto Bafico and Paolo Baldi.
\newblock Small random perturbations of {P}eano phenomena.
\newblock {\em Stochastics}, 6(3-4):279--292, 1981/82.

\bibitem[BCC22]{BoCiCr22}
Paolo Bonicatto, Gennaro Ciampa, and Gianluca Crippa.
\newblock On the advection-diffusion equation with rough coefficients: weak
  solutions and vanishing viscosity.
\newblock {\em Journal de Math\'ematiques Pures et Appliqu\'ees. Neuvi\`eme
  S\'erie}, 167:204--224, 2022.

\bibitem[BCD11]{BaChDa2011}
Hajer Bahouri, Jean-Yves Chemin, and Rapha\"{e}l Danchin.
\newblock {\em Fourier analysis and nonlinear partial differential equations},
  volume 343 of {\em Grundlehren der mathematischen Wissenschaften [Fundamental
  Principles of Mathematical Sciences]}.
\newblock Springer, Heidelberg, 2011.

\bibitem[BCDL21]{BrCoDe21}
Elia Bru\'e, Maria Colombo, and Camillo De~Lellis.
\newblock Positive solutions of transport equations and classical nonuniqueness
  of characteristic curves.
\newblock {\em Archive for Rational Mechanics and Analysis}, 240(2):1055--1090,
  2021.

\bibitem[BCK24]{BrCoKu}
Elia Bru\'{e}, Maria Colombo, and Anuj Kumar.
\newblock Sharp {N}onuniqueness in the {T}ransport {E}quation with {S}obolev
  {V}elocity {F}ield.
\newblock {\em arXiv preprint arXiv:2405.01670}, 2024.

\bibitem[BD00]{budhiraja2000variational}
Amarjit Budhiraja and Paul Dupuis.
\newblock A variational representation for positive functionals of infinite
  dimensional {B}rownian motion.
\newblock {\em Probability and mathematical statistics-Wroclaw University},
  20(1):39--61, 2000.

\bibitem[BD19]{BuDu19}
Amarjit Budhiraja and Paul Dupuis.
\newblock {\em Analysis and approximation of rare events}, volume~94 of {\em
  Probability Theory and Stochastic Modelling}.
\newblock Springer, New York, 2019.
\newblock Representations and weak convergence methods.

\bibitem[BDM08]{Bud_Dup}
Amarjit Budhiraja, Paul Dupuis, and Vasileios Maroulas.
\newblock {Large deviations for infinite dimensional stochastic dynamical
  systems}.
\newblock {\em The Annals of Probability}, 36(4):1390 -- 1420, 2008.

\bibitem[BFGM19]{Be_Fl}
Lisa Beck, Franco Flandoli, Massimiliano Gubinelli, and Mario Maurelli.
\newblock Stochastic {ODE}s and stochastic linear {PDE}s with critical drift:
  regularity, duality and uniqueness.
\newblock {\em Electronic Journal of Probability}, 24:Paper No. 136, 72, 2019.

\bibitem[BGJ17]{brzezniak2017large}
Zdzis{\l}aw Brze{\'z}niak, Ben Goldys, and Terence Jegaraj.
\newblock Large deviations and transitions between equilibria for stochastic
  {L}andau--{L}ifshitz--{G}ilbert equation.
\newblock {\em Archive for Rational Mechanics and Analysis}, 226:497--558,
  2017.

\bibitem[BGM24a]{BaGaMa24}
Marco Bagnara, Lucio Galeati, and Mario Maurelli.
\newblock {Regularization by rough Kraichnan noise for the generalised SQG
  equations}.
\newblock {\em arXiv 2405.12181}, 2024.

\bibitem[BGM24b]{BaGrMa24}
Marco Bagnara, Francesco Grotto, and Mario Maurelli.
\newblock Anomalous {R}egularization in {K}azantsev-{K}raichnan {M}odel.
\newblock {\em arXiv:2411.09482}, 2024.

\bibitem[Bil13]{billingsley2013convergence}
Patrick Billingsley.
\newblock {\em Convergence of probability measures}.
\newblock John Wiley \& Sons, 2013.

\bibitem[BL24]{butori2024large}
Federico Butori and Eliseo Luongo.
\newblock Large deviations principle for the inviscid limit of fluid dynamic
  systems in 2{D} bounded domains.
\newblock {\em Electronic Journal of Probability}, 29:1--42, 2024.

\bibitem[BM12]{bessaih2012large}
Hakima Bessaih and Annie Millet.
\newblock Large deviations and the zero viscosity limit for 2{D} stochastic
  {N}avier--{S}tokes equations with free boundary.
\newblock {\em SIAM Journal on Mathematical Analysis}, 44(3):1861--1893, 2012.

\bibitem[BM14]{brzezniak2014existence}
Zdzis{\l}aw Brze{\'z}niak and El{\.z}bieta Motyl.
\newblock The existence of martingale solutions to the stochastic {B}oussinesq
  equations.
\newblock {\em Global and Stochastic Analysis}, 1(2):175--216, 2014.

\bibitem[BO13]{Brzezniak_skoro}
Zdzis{\l}aw Brze{\'z}niak and Martin Ondrej{\'a}t.
\newblock Stochastic geometric wave equations with values in compact
  {R}iemannian homogeneous spaces.
\newblock {\em The Annals of Probability}, 41(3B):1938--1977, 2013.

\bibitem[CCK25]{Colombo^2Kumar}
Maria Colombo, Roberto Colombo, and Anuj Kumar.
\newblock A convex integration scheme for the continuity equation past the
  {S}obolev embedding threshold.
\newblock {\em arXiv:2504.03578}, 2025.

\bibitem[CCS21]{CiCrSp21}
Gennaro Ciampa, Gianluca Crippa, and Stefano Spirito.
\newblock Strong convergence of the vorticity for the 2{D} {E}uler equations in
  the inviscid limit.
\newblock {\em Archive for Rational Mechanics and Analysis}, 240:295–326,
  2021.

\bibitem[CD19]{CeDe19}
Sandra Cerrai and Arnaud Debussche.
\newblock Large deviations for the two-dimensional stochastic {N}avier-{S}tokes
  equation with vanishing noise correlation.
\newblock {\em Annales de l'Institut Henri Poincar\'e{} Probabilit\'es et
  Statistiques}, 55(1):211--236, 2019.

\bibitem[CDE22]{CoDrEl22}
Peter Constantin, Theodore~D. Drivas, and Tarek~M. Elgindi.
\newblock Inviscid {L}imit of {V}orticity {D}istributions in the {Y}udovich
  {C}lass.
\newblock {\em Communications on Pure and Applied Mathematics}, 75(1):60--82,
  2022.

\bibitem[CDPF13]{cerrai2013pathwise}
Sandra Cerrai, Giuseppe Da~Prato, and Franco Flandoli.
\newblock Pathwise uniqueness for stochastic reaction-diffusion equations in
  {B}anach spaces with an {H}{\"o}lder drift component.
\newblock {\em Stochastic Partial Differential Equations: Analysis and
  Computations}, 1:507--551, 2013.

\bibitem[CGT24]{cerrai2024nonlinear}
Sandra Cerrai, Giuseppina Guatteri, and Gianmario Tessitore.
\newblock Nonlinear random perturbations of {PDE}s and quasi-linear equations
  in {H}ilbert spaces depending on a small parameter.
\newblock {\em Journal of Functional Analysis}, 286(12):110418, 2024.

\bibitem[CL21]{ChLu21}
Alexey Cheskidov and Xiaoyutao Luo.
\newblock Nonuniqueness of weak solutions for the transport equation at
  critical space regularity.
\newblock {\em Annals of PDE}, 7(1):Paper No. 2, 45, 2021.

\bibitem[CL24]{ChLu24}
Alexey Cheskidov and Xiaoyutao Luo.
\newblock Extreme temporal intermittency in the linear {S}obolev transport:
  almost smooth nonunique solutions.
\newblock {\em Analysis \& PDE}, 17(6):2161--2177, 2024.

\bibitem[CM24]{CoMa23+}
Michele Coghi and Mario Maurelli.
\newblock {Existence and uniqueness by Kraichnan noise for 2D Euler equations
  with unbounded vorticity}.
\newblock {\em arXiv 2308.03216}, 2024.

\bibitem[CP19]{CePa19}
Sandra Cerrai and Nicholas Paskal.
\newblock Large deviations for fast transport stochastic {RDE}s with
  applications to the exit problem.
\newblock {\em The Annals of Applied Probability}, 29(4):1993--2032, 2019.

\bibitem[CS15]{CrSp15}
Gianluca Crippa and Stefano Spirito.
\newblock Renormalized solutions of the 2{D} {E}uler equations.
\newblock {\em Communications in Mathematical Physics}, 339:191–198, 2015.

\bibitem[DF14]{DeFl14}
Fran{\c{c}}ois Delarue and Franco Flandoli.
\newblock The transition point in the zero noise limit for a 1{D} {P}eano
  example.
\newblock {\em Discrete and Continuous Dynamical Systems. Series A},
  34(10):4071--4083, 2014.

\bibitem[DFV14]{delarue2014noise}
Fran{\c{c}}ois Delarue, Franco Flandoli, and Dario Vincenzi.
\newblock Noise prevents collapse of {V}lasov-{P}oisson point charges.
\newblock {\em Communications on Pure and Applied Mathematics},
  67(10):1700--1736, 2014.

\bibitem[DGP]{DrGaPa25+}
Theodore~D. Drivas, Lucio Galeati, and Umberto Pappalettera.
\newblock to appear.

\bibitem[DL89]{diperna1989ordinary}
Ronald~J DiPerna and Pierre-Louis Lions.
\newblock Ordinary differential equations, transport theory and {S}obolev
  spaces.
\newblock {\em Inventiones mathematicae}, 98(3):511--547, 1989.

\bibitem[DM19]{delarue2019zero}
Fran{\c{c}}ois Delarue and Mario Maurelli.
\newblock Zero noise limit for multidimensional {SDE}s driven by a pointy
  gradient.
\newblock {\em arXiv preprint arXiv:1909.08702}, 2019.

\bibitem[DPFPR13]{DP_regular_1}
Giuseppe Da~Prato, Frando Flandoli, Enrico Priola, and Michael R\"ockner.
\newblock Strong uniqueness for stochastic evolution equations in {H}ilbert
  spaces perturbed by a bounded measurable drift.
\newblock {\em The Annals of Probability}, 41(5):3306--3344, 2013.

\bibitem[DPFRV16]{DP_regular_2}
Giuseppe Da~Prato, Franco Flandoli, Michael R\"ockner, and A.~Yu. Veretennikov.
\newblock Strong uniqueness for {SDE}s in {H}ilbert spaces with nonregular
  drift.
\newblock {\em The Annals of Probability}, 44(3):1985--2023, 2016.

\bibitem[DPZ14]{DaPZab}
Giuseppe Da~Prato and Jerzy Zabczyk.
\newblock {\em Stochastic equations in infinite dimensions}, volume 152 of {\em
  Encyclopedia of Mathematics and its Applications}.
\newblock Cambridge University Press, Cambridge, second edition, 2014.

\bibitem[FG23]{FeGe23}
Benjamin Fehrman and Benjamin Gess.
\newblock {Non-equilibrium large deviations and parabolic-hyperbolic PDE with
  irregular drift}.
\newblock {\em Inventiones mathematicae}, 234(4):573 -- 636, 2023.

\bibitem[FGL21a]{flandoli2021delayed}
Franco Flandoli, Lucio Galeati, and Dejun Luo.
\newblock Delayed blow-up by transport noise.
\newblock {\em Communications in Partial Differential Equations},
  46(9):1757--1788, 2021.

\bibitem[FGL21b]{flandoli2021scaling}
Franco Flandoli, Lucio Galeati, and Dejun Luo.
\newblock Scaling limit of stochastic 2{D} {E}uler equations with transport
  noises to the deterministic {N}avier--{S}tokes equations.
\newblock {\em Journal of Evolution Equations}, 21(1):567--600, 2021.

\bibitem[FGP10]{FlaGubPri10}
Franco Flandoli, Massimiliano Gubinelli, and Enrico Priola.
\newblock Well-posedness of the transport equation by stochastic perturbation.
\newblock {\em Inventiones Mathematicae}, 180(1):1--53, 2010.

\bibitem[FGP11]{flandoli2011full}
Franco Flandoli, Massimiliano Gubinelli, and Enrico Priola.
\newblock Full well-posedness of point vortex dynamics corresponding to
  stochastic 2{D} euler equations.
\newblock {\em Stochastic Processes and their Applications}, 121(7):1445--1463,
  2011.

\bibitem[FL21]{flandoli2021high}
Franco Flandoli and Dejun Luo.
\newblock High mode transport noise improves vorticity blow-up control in 3{D}
  {N}avier--{S}tokes equations.
\newblock {\em Probability Theory and Related Fields}, 180:309--363, 2021.

\bibitem[FL23]{flandoli2023stochastic}
Franco Flandoli and Eliseo Luongo.
\newblock {\em Stochastic partial differential equations in fluid mechanics},
  volume 2330.
\newblock Springer, 2023.

\bibitem[Fla95]{Flandoli_Book_95}
Franco Flandoli.
\newblock {\em Regularity theory and stochastic flows for parabolic {SPDE}s},
  volume~9 of {\em Stochastics Monographs}.
\newblock Gordon and Breach Science Publishers, Yverdon, 1995.

\bibitem[Fla11]{flandoli_book}
Franco Flandoli.
\newblock {\em Random perturbation of {PDE}s and fluid dynamic models}, volume
  2015 of {\em Lecture Notes in Mathematics}.
\newblock Springer, Heidelberg, 2011.
\newblock Lectures from the 40th Probability Summer School held in Saint-Flour,
  2010, \'Ecole d'\'Et\'e{} de Probabilit\'es de Saint-Flour. [Saint-Flour
  Probability Summer School].

\bibitem[FW12]{FrWe_book}
Mark~I. Freidlin and Alexander~D. Wentzell.
\newblock {\em Random perturbations of dynamical systems}, volume 260 of {\em
  Grundlehren der mathematischen Wissenschaften [Fundamental Principles of
  Mathematical Sciences]}.
\newblock Springer, Heidelberg, third edition, 2012.
\newblock Translated from the 1979 Russian original by Joseph Sz\"ucs.

\bibitem[Gaw02]{Ga02}
Krzysztof Gawedzki.
\newblock Soluble models of turbulent advection.
\newblock {\em arXiv preprint nlin/0207058}, 2002.

\bibitem[GGM24]{GaGrMa24+}
Lucio Galeati, Francesco Grotto, and Mario Maurelli.
\newblock Anomalous regularization in {K}raichnan's passive scalar model.
\newblock {\em arXiv 2407.16668}, 2024.

\bibitem[GK96]{gyongy1996existence}
Istv{\'a}n Gy{\"o}ngy and Nicolai Krylov.
\newblock Existence of strong solutions for {I}t{\^o}'s stochastic equations
  via approximations.
\newblock {\em Probability theory and related fields}, 105(2):143--158, 1996.

\bibitem[GL23]{GaLu23+}
Lucio Galeati and Dejun Luo.
\newblock Weak well-posedness by transport noise for a class of 2{D} fluid
  dynamics equations.
\newblock {\em arXiv 2305.08761}, 2023.

\bibitem[GL24]{galeati2024ldp}
Lucio Galeati and Dejun Luo.
\newblock {LDP} and {CLT} for {SPDE}s with transport noise.
\newblock {\em Stochastics and Partial Differential Equations: Analysis and
  Computations}, 12(1):736--793, 2024.

\bibitem[GP22]{grotto2022burst}
Francesco Grotto and Umberto Pappalettera.
\newblock Burst of point vortices and non-uniqueness of 2{D} {E}uler equations.
\newblock {\em Archive for Rational Mechanics and Analysis}, 245(1):89--125,
  2022.

\bibitem[GRV24]{grotto2024zero}
Francesco Grotto, Marco Romito, and Milo Viviani.
\newblock Zero-noise dynamics after collapse for three point vortices.
\newblock {\em Physica D: Nonlinear Phenomena}, 457:133947, 2024.

\bibitem[Jak98]{jakubowski1998almost}
Adam Jakubowski.
\newblock The almost sure {S}korokhod representation for subsequences in
  nonmetric spaces.
\newblock {\em Theory of Probability \& Its Applications}, 42(1):167--174,
  1998.

\bibitem[JL25a]{JiLu25+}
Shuaijie Jiao and Dejun Luo.
\newblock Strong uniqueness by {K}raichnan transport noise for the {2D
  B}oussinesq equations with zero viscosity.
\newblock {\em arXiv:2504.19153}, 2025.

\bibitem[JL25b]{JiLu24}
Shuaijie Jiao and Dejun Luo.
\newblock Well-posedness of stochastic m{SQG} equations with {K}raichnan noise
  and l{p} data.
\newblock {\em Journal of Differential Equations}, 438:113362, 2025.

\bibitem[Kra68]{Kr68}
Robert~H. Kraichnan.
\newblock {Small‐Scale Structure of a Scalar Field Convected by Turbulence}.
\newblock {\em The Physics of Fluids}, 11(5):945--953, 1968.

\bibitem[Kur07]{kurtz}
Thomas Kurtz.
\newblock {The Yamada-Watanabe-Engelbert theorem for general stochastic
  equations and inequalities}.
\newblock {\em Electronic Journal of Probability}, 12:951 -- 965, 2007.

\bibitem[Mod20]{Mo20}
Stefano Modena.
\newblock On some recent results concerning non-uniqueness for the transport
  equation.
\newblock In {\em Hyperbolic problems: theory, numerics, applications},
  volume~10 of {\em AIMS Ser. Appl. Math.}, pages 562--568. Am. Inst. Math.
  Sci. (AIMS), Springfield, MO, [2020] \copyright 2020.

\bibitem[MP12]{MaPu}
Carlo Marchioro and Mario Pulvirenti.
\newblock {\em Mathematical theory of incompressible nonviscous fluids},
  volume~96.
\newblock Springer Science \& Business Media, 2012.

\bibitem[MS18]{MoSz18}
Stefano Modena and L\'aszl\'o Sz\'ekelyhidi, Jr.
\newblock Non-uniqueness for the transport equation with {S}obolev vector
  fields.
\newblock {\em Annals of PDE}, 4(2):Paper No. 18, 38, 2018.

\bibitem[MS20]{MoSa20}
Stefano Modena and Gabriel Sattig.
\newblock Convex integration solutions to the transport equation with full
  dimensional concentration.
\newblock {\em Annales de l'Institut Henri Poincar\'e{} C. Analyse Non
  Lin\'eaire}, 37(5):1075--1108, 2020.

\bibitem[MW17]{Paraprod_1}
Jean-Christophe Mourrat and Hendrik Weber.
\newblock Global well-posedness of the dynamic {$\Phi^4$} model in the plane.
\newblock {\em The Annals of Probability}, 45(4):2398--2476, 2017.

\bibitem[NLSW21]{NuSeWi21}
Helena~J. Nussenzveig~Lopes, Christian Seis, and Emil Wiedemann.
\newblock On the vanishing viscosity limit for 2{D} incompressible flows with
  unbounded vorticity.
\newblock {\em Nonlinearity}, 34(5):3112, may 2021.

\bibitem[Pit25]{pitcho2025zero}
Jules Pitcho.
\newblock On the zero-noise limit for {SDE}'s singular at the initial time.
\newblock {\em arXiv preprint arXiv:2503.22905}, 2025.

\bibitem[SBD19]{salins2019uniform}
Michael Salins, Amarjit Budhiraja, and Paul Dupuis.
\newblock Uniform large deviation principles for {B}anach space valued
  stochastic evolution equations.
\newblock {\em Transactions of the American Mathematical Society},
  372(12):8363--8421, 2019.

\bibitem[Sim86]{simon1986compact}
Jacques Simon.
\newblock Compact sets in the space ${L^p (0, T; B)}$.
\newblock {\em Annali di Matematica pura ed applicata}, 146:65--96, 1986.

\bibitem[Tem24]{temam2024navier}
Roger Temam.
\newblock {\em Navier--Stokes equations: theory and numerical analysis}, volume
  343.
\newblock American Mathematical Society, 2024.

\bibitem[Tri10]{Triebel83}
Hans Triebel.
\newblock {\em Theory of function spaces}.
\newblock Modern Birkh\"{a}user Classics. Birkh\"{a}user/Springer Basel AG,
  Basel, 2010.
\newblock Reprint of 1983 edition [MR0730762], Also published in 1983 by
  Birkh\"{a}user Verlag [MR0781540].

\bibitem[Ver81]{veretennikov1981strong}
Alexander~Ju Veretennikov.
\newblock {O}n strong solutions and explicit formulas forsolutions of
  stochastic integral equations.
\newblock {\em Mathematics of the USSR-Sbornik}, 39(3):387, 1981.

\bibitem[Vis18a]{vishik2018instability}
Misha Vishik.
\newblock Instability and non-uniqueness in the {C}auchy problem for the
  {E}uler equations of an ideal incompressible fluid. {P}art {I}.
\newblock {\em arXiv preprint arXiv:1805.09426}, 2018.

\bibitem[Vis18b]{vishik2018instability2}
Misha Vishik.
\newblock Instability and non-uniqueness in the {C}auchy problem for the
  {E}uler equations of an ideal incompressible fluid. {P}art {II}.
\newblock {\em arXiv preprint arXiv:1805.09440}, 2018.

\bibitem[vZ22]{van2022theory}
Willem van Zuijlen.
\newblock Theory of function and distribution spaces.
\newblock {\em Lecture notes. https://www. wias-berlin. de/people/vanzuijlen},
  2022.

\end{thebibliography}
\bibliographystyle{alpha}

\end{document}